\documentclass[twoside,11pt,reqno]{amsart}
\usepackage{amsmath,amssymb,amscd,mathrsfs,epic,wasysym,latexsym,tikz,mathrsfs,cite,hyperref}
\makeatletter

\numberwithin{equation}{section}

\newtheorem{Proposition}[equation]{Proposition}
\newtheorem{Lemma}[equation]{Lemma}
\newtheorem{Theorem}[equation]{Theorem}
\newtheorem{Corollary}[equation]{Corollary}
\theoremstyle{definition}  
\newtheorem{Definition}[equation]{Definition}
\newtheorem{Remark}[equation]{Remark}

\let\<\langle
\let\>\rangle

\newcommand\Comment[2][\relax]{\space\par\medskip\noindent%
   \fbox{\begin{minipage}{\textwidth}\textbf{Comment\ifx\relax#1\else---#1\fi}\newline%
        #2\end{minipage}}\medskip
}

\DeclareMathOperator\ldom {\triangleleft}
\DeclareMathOperator\ledom{\trianglelefteq}
\DeclareMathOperator\gdom {\triangleright}
\DeclareMathOperator\gedom{\trianglerighteq}

\def\bi{\text{\boldmath$i$}}
\def\bj{\text{\boldmath$j$}}

\def\bs{\text{\boldmath$s$}}

\def\bc{\text{\boldmath$c$}}
\def\br{\text{\boldmath$r$}}
\def\b1{\text{\boldmath$1$}}
\def\pmod#1{\text{ }(\text{\rm mod } #1)\,}

\def\bijection{\overset{\sim}{\longrightarrow}}

\newcommand{\Hom}{\operatorname{Hom}}

\newcommand{\ind}{\operatorname{ind}}

\def\sgn{\mathtt{sgn}}
\newcommand{\res}{\operatorname{res}}

\newcommand{\cha}{\operatorname{char}}
\newcommand{\St}{\operatorname{St}}

\newcommand{\cont}{\operatorname{cont}}

\newcommand{\Stab}{\operatorname{Stab}}

\newcommand{\Z}{\mathbb{Z}}

\def\phi{{\varphi}}

\newcommand{\Ga}{\Gamma}
\newcommand{\la}{\lambda}
\newcommand{\La}{\Lambda}
\newcommand{\al}{\alpha}

\def\Si{\mathfrak{S}}
\newcommand{\si}{\sigma}

\newcommand{\de}{\delta}

\def\triv#1{\O_{#1}}
\def\sign#1{\O_{-#1}}
\def\Belt{\mathbf{B}}

\newcommand{\Ind}{{\mathrm {Ind}}}

\newcommand{\Res}{{\mathrm {Res}}}

\newcommand{\C}{{\mathbb C}}

\newcommand{\D}{{\mathscr D}}

\def\h{{\mathfrak h}}

\newcommand\Par[1][]{{\mathscr P}^{\kappa#1}}

\def\T{{\mathtt T}}
\def\U{{\mathtt U}}
\def\Stab{{\mathtt S}}
\def\G{{\mathtt G}}

\def\defect{\operatorname{def}}
\def\codeg{\operatorname{codeg}}

\def\height{{\operatorname{ht}}}

\def\RPar{{\mathscr{RP}}^\kappa}
\def\surj{{\twoheadrightarrow}}
\def\Gar{{\operatorname{Gar}}}
\def\onto{{\twoheadrightarrow}}
\def\into{{\hookrightarrow}}
\def\Mod#1{#1\!\operatorname{-Mod}}
\renewcommand\O{\mathcal O}

\def\iso{\stackrel{\sim}{\longrightarrow}}

\newcommand\SetBox[2][35mm]{\Big\{\vcenter{\hsize#1\centering#2}\Big\}}

{\catcode`\|=\active
  \gdef\set#1{\mathinner{\lbrace\,{\mathcode`\|"8000%
  \let|\midvert #1}\,\rbrace}}
}
\def\midvert{\egroup\mid\bgroup}

\colorlet{darkgreen}{green!50!black}
\tikzset{dots/.style={very thick,loosely dotted},
         greendot/.style={fill,circle,color=darkgreen,inner sep=1.5pt,outer sep=0}
}
\def\greendot(#1,#2){\node[greendot] at(#1,#2){}}

\newenvironment{braid}{
  \begin{tikzpicture}[baseline=6mm,blue,line width=1pt, scale=0.4,
                      draw/.append style={rounded corners},
                      every node/.append style={font=\fontsize{5}{5}\selectfont}]%
  }{\end{tikzpicture}
}
\newcommand\iicrossing[1][i]{%
  \begin{braid}\tikzset{baseline=0mm,scale=0.5}
    \useasboundingbox (-0.6,0) rectangle (1,1.9);
    \draw(0,1)node[left=-1mm]{$#1$}--(1,0);
    \draw(1,1)node[right=-1mm]{$#1$}--(0,0);
  \end{braid}%
}
\newcommand\iisquare[1][i]{%
  \begin{braid}\tikzset{baseline=0mm,scale=0.7}
    \useasboundingbox (-0.6,0) rectangle (1,1.9);
    \draw(0,1)node[left=-1mm]{$#1$}--(1,0.5)--(0,0);
    \draw(1,1)node[right=-1mm]{$#1$}--(0,0.5)--(1,0);
  \end{braid}%
}
\def\Grid(#1,#2){
  \draw[very thin,gray,step=2mm] (0,0)grid(#1,#2);
  \draw[very thin,darkgreen,step=10mm] (0,0)grid(#1,#2);
}

\newcommand\Tableau[2][\relax]{
  \begin{tikzpicture}[scale=0.5,draw/.append style={thick,black}]
    \ifx\relax#1\relax%
    \else 
      \foreach\box in {#1} { \filldraw[blue!30]\box+(-.5,-.5)rectangle++(.5,.5); }
    \fi
    \newcount\row\newcount\col
    \row=0
    \foreach \Row in {#2} {
       \col=1
       \foreach\k in \Row {
          \draw(\the\col,\the\row)+(-.5,-.5)rectangle++(.5,.5);
          \draw(\the\col,\the\row)node{\k};
          \global\advance\col by 1
       }
       \global\advance\row by -1
    }
  \end{tikzpicture}
}

\newcommand\YoungDiagram[2][\relax]{
  \begin{tikzpicture}[scale=0.5,draw/.append style={thick,black}]
    \ifx\relax#1\relax%
    \else 
    \foreach\box in {#1} {
      \filldraw[blue!30]\box rectangle ++(1,1);
    }
    \fi
    \newcount\row
    \row=0
    \foreach \col in {#2} {
       \draw(1,\the\row)grid ++(\col,1);
       \global\advance\row by -1
    }
  \end{tikzpicture}
}

\begin{document}


\title[Universal Specht modules]{{\bf Universal graded Specht modules for cyclotomic Hecke algebras}}

\author{\sc Alexander S. Kleshchev}
\address{Department of Mathematics\\ University of Oregon\\
Eugene\\ OR~97403, USA}
\email{klesh@uoregon.edu}

\author{\sc Andrew Mathas}
\address{School of Mathematics and Statistics F07, 
University of Sydney, NSW 2006, Australia}
\email{andrew.mathas@sydney.edu.au}

\author{\sc Arun Ram}
\address{Department of Mathematics and Statistics,
University of Melbourne, VIC 3010, Australia}
\email{aram@unimelb.edu.au}


\thanks{
Research supported by the NSF (grant DMS-0654147), the Australian Research Council (DP0986774, DP0986349, DP110100050),  an International
Visiting Professorship at the University of Sydney, the University of Melbourne, and the Hausdorff Institute for Mathematics. The first author thanks the University of Melbourne and the University of Sydney for hospitality.}

\begin{abstract}
The graded Specht module $S^\lambda$ for a cyclotomic Hecke algebra comes with a distinguished generating vector $z^\lambda\in S^\lambda$, which can be thought of as a ``highest weight vector of weight $\lambda$''. This paper describes the {\em defining relations} for the Specht module $S^\lambda$ as a graded module generated by $z^\lambda$. The first three relations say precisely what it means for $z^\lambda$ to be a highest weight vector of weight~$\lambda$. The remaining relations are homogeneous analogues of the classical {\em Garnir relations}. 
The homogeneous Garnir relations, which are {\em simpler} than the classical ones, are
associated with a remarkable family of homogeneous operators on the Specht module which satisfy the braid
relations.
\end{abstract}

\maketitle

\section{Introduction}
Let $S_d$ be the symmetric group on $d$ letters. A central role in representation theory of $S_d$ is played by certain $\Z S_d$-modules $S^\la$ labelled by the partitions $\la$ of $d$. These modules are called {\em Specht modules} and their  construction goes back to \cite{Young1,Young2,Specht}. Specht modules also arise naturally as cell modules in the cellular structure on the group algebra of $S_d$ constructed by Murphy in \cite{Murphy}, see \cite{MathasB,DJM,JM} for further development of these ideas which will be important in this paper. 

It was shown recently by Brundan and the first author \cite{BK} that over an arbitrary {\em  field}\, $F$, the group algebra $FS_d$ is explicitly isomorphic to a certain cyclotomic Khovanov-Lauda-Rouquier (KLR) algebra $R^\Lambda_d$. The algebra $R^\Lambda_d$  is $\Z$-graded, and this grading can be transferred to $FS_d$ using the Brundan-Kleshchev isomorphism. Moreover, in \cite{BKW}, the Specht modules over $F$ were also explicitly graded, which played a crucial role in the graded categorification theorem of \cite{BKllt} generalizing the Ariki's categorification theorem \cite{Ariki}. We refer the reader to \cite{Kbull} for description of these ideas and further references. 

Hu and the second author \cite{HM} have completed the picture by constructing a {\em graded cellular}\, structure on the group algebra of the symmetric group, so that the graded Specht modules of \cite{BKW} arise as the corresponding cell modules. 

In all constructions above, the Specht
module $S^\la$ comes together with a remarkable generating vector $z^\la\in S^\la$, which can be
thought of, informally, as a ``highest weight vector of weight $\la$''. The goal of this paper is to describe the {\em defining relations} of the Specht module $S^\la$
over $\Z$ as a graded module over the KLR algebra $R^\Lambda_d$ generated by $z^\la$. This idea of presenting Specht modules by generators and relations is responsible for our terminology {\em universal Specht modules}. 

Our homogeneous relations for $S^\lambda$ are given in Definition~\ref{DSpecht}. The first
three relations say precisely what we mean by $z^\lambda$ being a highest weight vector of
weight $\la$. The fourth and final relation is a remarkable family of {\em homogeneous Garnir
relations}, which we consider to be the key innovation of this paper. 

We point out that the classical Garnir relations, which go back to \cite{Young2,Garnir}, are very far
from being homogeneous with respect to the gradings under consideration. 
The classical Garnir relations have the form of an alternating sum of elements of the Specht
module corresponding to certain tableaux (being equated to zero). 

Even though substantial initial work is required to define the homogeneous Garnir relations, they are actually much {\em simpler} than the classical ones. For example, if the underlying Lie type of the KLR algebra is $A_{\infty}$, which under the isomorphism of \cite{BK} corresponds to the case where the field $F$ has characteristic $0$, then the homogeneous Garnir relation has the form of just {\em one}\, element corresponding to a special Garnir tableaux (being equated to zero). In the case where the Lie type is $A_{p-1}^{(1)}$, which under the isomorphism of \cite{BK} corresponds to the field $F$ having characteristic $p>0$, the homogeneous Garnir relation does look like a sum, but it has roughly $p$ times as few summands as the classical Garnir relation. For the case of the so-called calibrated representations of the affine Hecke algebra in characteristic zero this phenomenon has been known, see for example \cite[(5.4)]{Ram}. 

Even though so far we have been talking only about the symmetric groups, the story of Specht modules generalizes to all {\em cyclotomic Hecke algebras}, both 
degenerate and non-degenerate. This is the generality which we work with throughout this
paper. 


In section~\ref{SComb} we collect various combinatorial facts and notation. The key notion
here is that of the degree of a standard tableau which was first defined in \cite{BKW}. In
section~\ref{SKLR}, we recall the definition of the affine and cyclotomic KLR algebras and
define ``permutation modules" for these algebras using induction from one-dimensional modules
of the parabolic subalgebras in the affine setting. 

In the crucial section~\ref{SBI}, we define certain elements which we call {\em block
intertwiners}. These intertwiners will later be fed into the definition of the homogeneous
Garnir relations. They permute blocks (or bricks) of size $e$, where $e$ can be thought of as
the analogue of the characteristic of the ground field when working with Specht modules for the
symmetric groups, and this part of the story is trivial when $e=0$. The block intertwiners
$\tau_r$ are defined in terms of products of the large number  of the KLR generators. The KLR
generators do not satisfy Coxeter relations, so we find it truly remarkable that the brick
intertwiners $\tau_r$ do! See the key Theorem~\ref{tau braid}. 

In section~\ref{SHGR}, we define (row) Garnir relations and {\em universal (row) Specht modules} $S^\la$ for the algebra $R^\Lambda_d$ by generators and relations, see Definition~\ref{DSpecht}. Our next goal is to prove that if we identify the cyclotomic KLR algebra $R_d^\La$ with the cyclotomic Hecke algebra $H_d^\La$ via the Brundan-Kleshchev isomorphism, which is only valid over a field, then the universal Specht modules are identified with the usual graded Specht modules of \cite{BKW}. This is done in section~\ref{SNot}.

Section~\ref{SDualSpecht}  develops the parallel story for the column Specht modules $S_\la$, which turn out to be dual to the row Specht modules $S^\la$. 
Accidentally, 
what we call a column Specht module $S_\la$ is what was called a Specht module in James' book \cite{Jbook}. 

The final section~\ref{SApp} contains two applications. One is the description of Specht modules for higher level cyclotomic Hecke algebras as modules induced from Specht modules of level $1$, see Theorem~\ref{TSpechtHigherLevel}. 
In fact, these induced modules were sometimes taken as a definition of Specht modules for higher levels, which was problematic because the connection with the Specht modules as cell modules had not been established in full generality before.

Our second application is a generalization of the famous very useful result from James' book \cite[Theorem 8.15]{Jbook} for symmetric groups: 
$$(S^\la)^*\cong S^{\la'}\otimes \sgn.$$ 
The analogue of this is proved for arbitrary cyclotomic Hecke algebras in the graded setting, see Theorem~\ref{TSignDualSpecht}. 


\section{Combinatorics}\label{SComb}

\subsection{
Lie theoretic notation}\label{SSLTN}
Let $e\in\{0,2,3,4,\dots\}$ and $I:=\Z/e\Z.$ 
Let $\Ga$ be the quiver with vertex set  $I$,
and a directed edge from $i$ to $j$ if $j  = i-1$ (the orientation differs from the one in \cite{BK,HM}). 
Thus $\Gamma$ is a quiver of type~$A_\infty$ if $e=0$
or $A_{e-1}^{(1)}$ if $e > 0$. The corresponding {\em Cartan matrix} 
$(a_{i,j})_{i, j \in I}$ is defined by 
\begin{equation}\label{ECM}
a_{i,j} := \left\{
\begin{array}{rl}
2&\text{if $i=j$},\\
0&\text{if $j \neq i, i \pm 1$},\\
-1&\text{if $i \rightarrow j$ or $i \leftarrow j$},\\
-2&\text{if $i \rightleftarrows j$}.
\end{array}\right.
\end{equation}
(The case $a_{i,j}=-2$ only occurs if $e=2$.) 

Following \cite{Kac}, let $(\h,\Pi,\Pi^\vee)$ be a realization of the Cartan matrix $(a_{ij})_{i,j\in I}$, so we have the simple roots 
$\{\al_i\mid i\in I\},$ 
the fundamental dominant weights 
$\{\La_i\mid i\in I\},$ and the normalized invariant form $(\cdot,\cdot)$ such that
$$
(\al_i,\al_j)=a_{ij}, \quad (\La_i,\al_j)=\de_{ij}\qquad(i,j\in I).
$$
If $e>0$, the {\em null-root} is
\begin{equation}\label{E:nullRoot}
\de:=\al_0+\al_1+\dots+\al_{e-1}.
\end{equation}
Let $P_+$ be the set of dominant integral weights, and 
$
Q_+ := \bigoplus_{i \in I} \Z_{\geq 0} \alpha_i
$ 
the positive part of the root lattice. For $\alpha \in Q_+$ let $\height(\alpha)$ be 
the {\em height of~$\al$}. That is, $\height(\al)$ is the sum of the 
coefficients when $\al$ is expanded in terms of the $\alpha_i$'s.

Let $\Si_d$ be the {\em symmetric group} on $d$ letters and let $s_r=(r,r+1)$, for $1\le r<d$, be the simple transpositions of~$\Si_d$. Then $\Si_d$
acts from the left on the set $ I^d$ by place permutations. If $\bi=(i_1,\dots,i_d) \in I^d$
then its {\em weight} is $|\bi|:=\alpha_{i_1}+\cdots+\alpha_{i_d}\in Q_+$.  Then the
$\Si_d$-orbits on $I^d$ are the sets 
\begin{equation*} I^\alpha := \{\bi \in I^d\mid \al=|\bi|\} 
\end{equation*} 
parametrized by all $\alpha \in Q_+$ of height $d$. 

Throughout the paper, we fix a positive integer $l$, referred to as the {\em level},  and an ordered $l$-tuple  
\begin{equation}\label{EKappa}
\kappa=(k_1,\dots,k_l)\in I^l.
\end{equation}
Define the dominant weight $\La$ (of level $l$) as follows:  
\begin{equation}\label{ELa}
\La=\La(\kappa):=\La_{k_1}+\dots+\La_{k_l}\in P_+.
\end{equation}
Finally, for $\al\in Q_+$, define the {\em defect} of $\al$ (relative to $\La$) to be
\begin{equation}\label{defdef}
\defect(\al)=(\La,\al)-\frac12(\al,\al).
\end{equation}

\subsection{Partitions}\label{SSPar}
Recall that in (\ref{EKappa}) we have fixed a level $l$ and an $l$-tuple
$\kappa=(k_1,\dots,k_l)$.  An {\em $l$-multipartition} of $d$ is an ordered $l$-tuple
of partitions $\mu = (\mu^{(1)} , \dots,\mu^{(l)})$ such that $\sum_{m=1}^l
|\mu^{(m)}|=d$.  We call $\mu^{(m)}$ the $m$th {\em component} of $\mu$. 
Let $\Par_d$ be the set of all $l$-multipartitions of $d$  and put $\Par:=\bigsqcup_{d\geq 0}\Par_d$. Of course, $\Par$ only depends on $l$, and not on $\kappa$, but as soon as we consider residues of nodes of multipartitions, the dependence on~$\kappa$ becomes  crucial. 

The {\em Young diagram} of the multipartition $\mu = (\mu^{(1)} , \dots,\mu^{(l)})\in \Par$ is 
$$
\{(a,b,m)\in\Z_{>0}\times\Z_{>0}\times \{1,\dots,l\}\mid 1\leq b\leq \mu_a^{(m)}\}.
$$
The elements of this set
are the {\em nodes of $\mu$}. More generally, a {\em node} is any element of $\Z_{>0}\times\Z_{>0}\times \{1,\dots,l\}$. 
Usually, we identify the multipartition~$\mu$ with its 
Young diagram and visualize it as a column vector of Young diagrams. 
For example, $((3,1),\emptyset,(4,2))$ is the Young diagram
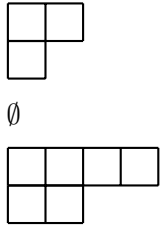
\begin{align*}
&\YoungDiagram{2,1} \\ &\emptyset\\ &\YoungDiagram{4,2}
\end{align*}

To each node $A=(a,b,m)$ 
we associate its {\em residue}, which is the following element of $I=\Z/e\Z$: 
\begin{equation}\label{ERes}
\res A=\res^\kappa A=k_m+(b-a)\pmod{e}.
\end{equation}
An {\em $i$-node} is a node of residue $i$.
Define the {\em residue content of $\mu$} to be
\begin{equation}\label{EContent}
\cont(\mu):=\sum_{A\in\mu}\al_{\res A} \in Q_+.
\end{equation}
Denote
$$
\Par_\al:=\{\mu\in\Par\mid \cont(\mu)=\al\}\qquad(\al\in Q_+).
$$

A node $A\in\mu$ is a {\em removable node (of~$\mu$)}\, if $\mu\setminus \{A\}$ is (the diagram of) a multipartition. A node $B\not\in\mu$ is an {\em addable node (for~$\mu$)}\, if $\mu\cup \{B\}$ is a multipartition. We use the notation
$$
\mu_A:=\mu\setminus \{A\},\qquad \mu^B:=\mu\cup\{B\}.
$$

Let $\mu,\nu\in \Par_d$. Then $\mu$ {\em dominates} $\nu$, and we write $\mu\unrhd\nu$, if 
$$
\sum_{a=1}^{m-1}|\mu^{(a)}|+\sum_{b=1}^c\mu_b^{(m)}\geq 
\sum_{a=1}^{m-1}|\nu^{(a)}|+\sum_{b=1}^c\nu_b^{(m)}
$$
for all $1\leq m\leq l$ and $c\geq 1$.
In other words, $\mu$ is obtained from $\nu$ by moving nodes up in
the diagram.

We define
\begin{equation}\label{EKappaPrime}
\kappa':=(-k_l,\dots,-k_1).
\end{equation}
Now, let $\mu=(\mu^{(1)},\dots,\mu^{(l)})\in\Par$. 
The {\em conjugate} of $\mu$ is the multipartition
$$\mu'=(\mu^{(l)'},\dots,\mu^{(1)'})\in \Par['],$$ 
where each $\mu^{(m)'}$ is the partition conjugate to $\mu^{(m)}$ in the usual sense, that is, $\mu^{(m)'}$ is obtained by swapping the rows and columns of $\mu^{(m)}$.  

\subsection{Tableaux}\label{SS:tableaux}
Let $\mu=(\mu^{(1)},\dots,\mu^{(l)})\in\Par_d$. 
A {\em $\mu$-tableau} 
$\T=(\T^{(1)},\dots,\T^{(l)})$ is obtained from the diagram of $\mu$ by 
inserting the integers $1,\dots,d$ into the nodes, allowing no repeats. 
For each $m=1,\dots,l$, $\T^{(m)}$ is a $\mu^{(m)}$-tableau, called the $m$th {\em component} of $\T$. 
If the node $A=(a,b,m)\in\mu$ is occupied by the integer $r$ in $\T$ then we write $r=\T(a,b,m)$ and set
$\res_\T(r)=\res A$. The {\em residue sequence} of~$\T$ is
\begin{equation}\label{EResSeq}
\bi(\T)=\bi^\kappa(\T)=(i_1,\dots,i_d)\in I^d,
\end{equation}
where $i_r=\res_\T(r)$ is the residue of the node occupied by 
$r$ in $\T$ ($1\leq r\leq d$). 


A $\mu$-tableau $\T$ is {\em row-strict} (resp. {\em column-strict}) if its entries increase from left to right (resp. from top to bottom) along the rows (resp. columns) of each component of $\T$. 
A $\mu$-tableau $\T$ is {\em standard} if it is row- and column-strict. 
 Let $\St(\mu)$ be the set of standard $\mu$-tableaux.

Let $\T$ be a $\mu$-tableau and suppose that $1\leq r\neq s\leq d$ and that $r=\T(a_1,b_1,m_1)$ and that
$s=\T(a_2,b_2,m_2)$. We write $r\nearrow_\T s$ if $m_1=m_2$, $a_1>a_2$, and $b_1<b_2$; informally, $r$ and $s$
are in the same component and~$s$ is strictly to the north-east of $r$ within that component.  The symbols $\rightarrow_\T,\searrow_\T,\downarrow_\T$ have the  similar obvious meanings. For example,
$r\downarrow_\T s$ means that $r$ and $s$ are located in the same column of the same component of~$\T$ and
that~$s$ is in a strictly lower row of~$\T$ than~$r$. 

Let $\mu\in\Par$, $i\in I$, $A$ be a removable $i$-node and $B$ be an addable $i$-node 
of $\mu$. We set
\begin{align}\label{EDMUA}
d_A(\mu)&=\#\SetBox{addable $i$-nodes of $\mu$\\[-4pt] strictly below $A$}
                -\#\SetBox[38mm]{removable $i$-nodes of $\mu$\\[-4pt] strictly below $A$},\\
\intertext{and}
d^B(\mu)&=\#\SetBox{addable $i$-nodes of $\mu$\\[-4pt] strictly above $B$}
                -\#\SetBox[38mm]{removable $i$-nodes of $\mu$\\[-4pt] strictly above $B$}.
\end{align} 

Given $\mu \in \Par_d$ and $\T \in \St(\mu)$, the {\em degree} of $\T$ is defined in
\cite[section~3.5]{BKW} inductively as follows. If $d=0$, then $\T$ is the empty tableau $\emptyset$, and
we set $\deg(\T):=0$.  Otherwise, let $A$ be the node occupied by $d$ in $\T$. Let $\T_{<d}\in\St(\mu_A)$
be the tableau obtained by removing this node and set
\begin{equation}\label{EDegTab}
\deg(\T):=d_A(\mu)+\deg(\T_{<d}).
\end{equation}
Similarly, define a dual notion of {\em codegree} by
\begin{equation}\label{ECodeg}
\codeg(\emptyset):=0,\quad \codeg(\T):=d^A(\mu_A)+\codeg(\T_{<d}).
\end{equation}
The definition of the degree and codegree of a tableau depend on the residues and so,
ultimately, they depend on~$\kappa$ by (\ref{ERes}). We write $\deg^\kappa(\T)$ and
$\codeg^\kappa(\T)$ when we wish to emphasize this dependence.

By  \cite[Lemma 3.12]{BKW}, using codegree instead of degree for a tableau leads only to a negation and  ``global shift'' by the defect: more precisely, we have
\begin{equation}\label{EDegCodeg}
\deg(\T)+\codeg(\T)=\defect(\al)\qquad(\T\in\St(\mu), \mu\in\Par_\al).
\end{equation}

The group $\Si_d$ acts on the set of $\mu$-tableaux from the left by acting on the entries of the tableaux.
Let $\T^\mu$ be the $\mu$-tableau in which the numbers $1,2,\dots,d$ appear in order from left to right along the successive rows,
working from top row to bottom row. Let $\T_\mu$ be the $\mu$-tableau in which the numbers $1,2,\dots,d$ appear in
from top to bottom along the successive columns, working from the leftmost column to the rightmost column within a component
and moving from the $l$th component up to the first component. 

For example, if $\mu=((3,1),(2,2))$ then
$$
\T^\mu=\begin{array}{l} \Tableau{{1,2,3},{4}}\\[2mm]\Tableau{{5,6},{7,8}}\end{array}
\quad\text{and}\quad
\T_\mu=\begin{array}{l} \Tableau{{5,7,8},{6}}\\[4mm]\Tableau{{1,3},{2,4}}\end{array}.
$$
Set 
\begin{equation}\label{EBIMu}
\bi^\mu:=\bi(\T^\mu) \quad\text{and}\quad
\bi_\mu:=\bi(\T_\mu).
\end{equation} 
For each $\mu$-tableau $\T$ define permutations $w^\T$ and $w_\T$ in~$\Si_d$ by the equations
\begin{equation}\label{EWT}
w^\T  \T^\mu=\T=w_\T \T_\mu.
\end{equation}

If
$\T=(\T^{(1)},\dots,T^{(l)})\in\St(\mu)$ then the {\em conjugate} of $\T$ is the standard $\mu'$-tableau
$\T'=(\T^{(l)'},\dots,\T^{(1)'})$, where $\T^{(m)'}$ is the $\mu^{(m)'}$-tableau obtained by swapping the rows
and columns of $\T^{(m)}$, for $1\le m\le l$. For example, $(\T^\mu)'=\T_{\mu'}$.

\subsection{Bruhat order}
Let $\ell$ be the length function on $\Si_d$ with respect to the Coxeter generators
$s_1,s_2,\dots,s_{d-1}$. Let $\le$ be the {\em Bruhat order} on $\Si_d$ (so that $1\le w$ for all
$w\in\Si_d)$. Define a related
partial order on $\St(\mu)$ as follows: if $\Stab,\T\in\St(\mu)$ then
\begin{align}\label{EBruhat}
\Stab\ledom\T\quad\text{if and only if}\quad w_\Stab\le w_\T.
\end{align}
If $\Stab\ledom\T$ then we also write $\T\gedom\Stab$. If $\Stab\ledom\T$ and $\Stab\ne\T$ we write $\Stab\ldom\T$ and $\T\gdom\Stab$.

Observe that if  $\T\in\St(\mu)$ then $\T_\mu\ledom\T\ledom\T^\mu$. There is a similar connection between the 
relation $w^\Stab\le w^\T$ and the corresponding tableaux. To describe this,   recall conjugate multipartitions and
tableaux.

\begin{Lemma}\label{eBruhat}
  Suppose that $\mu\in\Par_d$ and that $\Stab,\T\in\St(\mu)$. Then:
  \begin{enumerate}
    \item $w_\T=w^{\T'}$;
    \item $\T\ledom\Stab$ if and only if $w^\T\ge w^\Stab$;
    \item $w_{T^\mu}=(w^\T)^{-1}w_\T$ and $w^{\T_\mu}=w_\T^{-1}w^\T$ with
      $\ell(w_{T^\mu})=\ell(w^\T)+\ell(w_\T)$ and $\ell(w^{\T_\mu})=\ell(w_\T)+\ell(w^\T)$.
  \end{enumerate}
\end{Lemma}

\begin{proof}
(i) Observe that $(\T^\mu)'=\T_{\mu'}$, $(\T_\mu)'=\T^{\mu'}$ and
$\St(\mu)=\set{\T|\T'\in\St(\mu')}$. Now, conjugating the equation
$\T=w^\T\T^\mu$ shows that $w_{\T'}=w^\T$, for $\T\in\St(\mu)$.

(ii)  If $\U\in\St(\mu)$ and $1\le k\le d$, let $\U_{\le k}$ be the subtableau of $\T$ containing the entries $1,2\dots,k$. Then it follows from \cite[Theorem~3.8]{MathasB} that $\Stab\gedom\T$ if and only if the shape
of $\Stab_{\le k}$ dominates the shape of $\T_{\le k}$ for all $1\le k\le d$. So $\T\ledom\Stab$ if and
only if $\Stab'\ledom\T'$. Therefore, by part (i) and (\ref{EBruhat}), we get
  $$\T\ledom\Stab\Longleftrightarrow \Stab'\ledom\T'\Longleftrightarrow 
     w_{\Stab'}\le w_{\T'} \Longleftrightarrow w^\Stab\le w^\T.$$

(iii) Since $w_\T\T_\mu=\T=w^\T\T^\mu$, we have $w_{\T^\mu}\T_\mu=(w^\T)^{-1}w_\T\T_\mu$ which implies that
$w_{T^\mu}=(w^\T)^{-1}w_\T$. Since $\T_\mu\ledom\T\ledom\T^\mu$ we obtain
$\ell(w_{T^\mu})=\ell(w^\T)+\ell(w_\T)$ using the description of the Bruhat order given in~(ii). The remaining
claims are proved similarly.
\end{proof}

We will also need the following result.

\begin{Lemma}\label{L211108} {\rm \cite[Lemma 3.7]{BKW}} 
Suppose that $\mu\in\Par_d$, $\T\in\St(\mu)$, and  $1\leq r<d$ such that $r\downarrow_\T r+1$ or $r\rightarrow_\T r+1$. Suppose that $\Stab\in\St(\mu)$ and $\Stab\gdom s_r\T$. Then $\Stab\gedom \T$. 
\end{Lemma}

\section{KLR algebras and permutation modules}\label{SKLR}

Throughout this paper a \textit{graded
algebra}  will mean a $\Z$-graded algebra and a \textit{graded module} will be a $\Z$-graded module. If $A$ is a graded algebra then $\Mod{A}$ is
the category of finitely generated graded (left) $A$-modules with degree preserving maps.
We use the standard notation of graded representation theory. 
In particular, if $M=\bigoplus_{d\in\Z}M_d$ then $v\in M_d$ is
\textit{homogeneous} of \textit{degree} $d=\deg v$. Further, if $n\in\Z$ then $M\<n\>$ is the
graded module obtained by shifting the grading on $M$ up by $n$ so that $M\<n\>_d=M_{d-n}$.


\subsection{KLR algebras}
Let $\O$ be a commutative ring with identity and $\al\in Q_+$. Recall from \cite{KL1,KL2,R} that the (affine) {\em Khovanov-Lauda-Rouquier algebra}, or KLR algebra, $R_\al=R_\al(\O)$, is defined to be the unital $\O$-algebra generated by the elements 
\begin{equation}\label{EKLGens}
\{e(\bi)\:|\: \bi\in I^\al\}\cup\{y_1,\dots,y_{d}\}\cup\{\psi_1, \dots,\psi_{d-1}\},
\end{equation}
subject only to the following relations:
\begin{align}
e(\bi) e(\bj) &= \de_{\bi,\bj} e(\bi);
\hspace{11.3mm}{\textstyle\sum_{\bi \in I^\al}} e(\bi) = 1;\label{R1}
\\
y_r e(\bi) &= e(\bi) y_r;
\hspace{20mm}\psi_r e(\bi) = e(s_r{ }\bi) \psi_r;\label{R2PsiE}\\
\label{R3Y}
y_r y_s &= y_s y_r;\\
\label{R3YPsi}
\psi_r y_s  &= y_s \psi_r\hspace{42.4mm}\text{if $s \neq r,r+1$};\\
\psi_r \psi_s &= \psi_s \psi_r\hspace{41.8mm}\text{if $|r-s|>1$};\label{R3Psi}\\
\psi_r y_{r+1} e(\bi) &=  (y_r\psi_r+\delta_{i_r,i_{r+1}})e(\bi)
\label{ypsi}\\
y_{r+1} \psi_re(\bi) &= (\psi_ry_r+\delta_{i_r,i_{r+1}})e(\bi)
\label{psiy}\\
\psi_r^2e(\bi) &= 
\left\{
\begin{array}{ll}
0&\text{if $i_r = i_{r+1}$},\\
e(\bi)&\text{if $i_{r+1} \neq i_r, i_r \pm 1$
},\\
(y_{r+1}-y_r)e(\bi)&\text{if $i_r \rightarrow i_{r+1}$},\\
(y_r - y_{r+1})e(\bi)&\text{if $i_r \leftarrow i_{r+1}$},\\
(y_{r+1} - y_{r})(y_{r}-y_{r+1}) e(\bi)\!\!\!&\text{if $i_r \rightleftarrows i_{r+1}$};
\end{array}
\right.
 \label{quad}\\
\psi_{r}\psi_{r+1} \psi_{r} e(\bi)
&=
\left\{\begin{array}{ll}
(\psi_{r+1} \psi_{r} \psi_{r+1} +1)e(\bi)&\text{if $i_{r+2}=i_r \rightarrow i_{r+1}$},\\
(\psi_{r+1} \psi_{r} \psi_{r+1} -1)e(\bi)&\text{if $i_{r+2}=i_r \leftarrow i_{r+1}$},\\
\big(\psi_{r+1} \psi_{r} \psi_{r+1} -2y_{r+1}
\\\qquad\:\quad +y_r+y_{r+2}\big)e(\bi)
\hspace{2.4mm}&\text{if $i_{r+2}=i_r \rightleftarrows i_{r+1}$},\\
\psi_{r+1} \psi_{r} \psi_{r+1} e(\bi)&\text{otherwise}.
\end{array}\right.
\label{braid}
\end{align}

Recall from (\ref{ELa}) that we have fixed $\La=\La(\kappa)\in P_+$. The corresponding {\em cyclotomic KLR algebra} $R_\al^\La=R_\al^\La(\O)$ is generated by the same elements (\ref{EKLGens}) subject only to the relations (\ref{R1})--(\ref{braid}) and  the additional {\em cyclotomic relations}
\begin{equation}\label{ERCyc}
y_1^{(\La,\al_{i_1})}e(\bi)=0\qquad(\bi=(i_1,\dots,i_d)\in I^\al).
\end{equation}
Thus $R_\al^\La$ is the quotient of $R_\al$ by the relations (\ref{ERCyc}). 

The algebras $R_\al$ and $R_\al^\La$ have $\Z$-gradings determined by setting 
$e(\bi)$ to be of degree 0,
$y_r$ of degree $2$, and
$\psi_r e(\bi)$ of degree $-a_{i_r,i_{r+1}}$
for all $r$ and $\bi \in I^\al$.

Note that $R_\al(\Z)\otimes_\Z \O\cong R_\al(\O)$ and $R_\al^\La(\Z)\otimes_\Z \O\cong R_\al^\La(\O)$.  In this
paper~$\O$ will usually be $\Z$ or a field $F$. 

\subsection{Graded duality}\label{SSisomorphisms}
Let $\al\in Q_+$ be of height $d$. 
It is easy to check using generators and relations that there exists a homogeneous algebra anti-involution 
\begin{equation}\label{ECircledast}
\tau:R_\al\longrightarrow R_\al,\quad e(\bi)\mapsto e(\bi),\quad y_r\mapsto y_r,\quad \psi_s\mapsto \psi_s. 
\end{equation}
for all $\bi\in I^\al,\ 1\leq r\leq d$, and $1\leq s<d$. Note that $\tau$ factors through to an anti-involution of the cyclotomic quotient $R_\al^\La$, which we also denote by $\tau$. 

If $M=\bigoplus_{d\in\Z}M_d$ is a finite rank graded $R_\al$-module, then the {\em graded dual}
$M^\circledast$ is the graded $\O$-module such that $(M^\circledast)_d:=\Hom_\O(M_{-d},\O)$, for all
$d\in\Z$, and where the $R_\al$-action is given by $(xf)(m)=f(\tau(x)m)$, for all $f\in M^\circledast, m\in M, x\in
R_\al$. 



\subsection{The sign map}\label{SSSign}
For $\bi=(i_1,\dots,i_d)\in I^d$, 
set 
\begin{equation}
\label{E-I}
-\bi:=(-i_1,\dots,-i_d).
\end{equation} 
If 
$\al=\sum_{i\in I}a_i\alpha_i\in Q_+$, then define
$$\alpha'=\sum_{i\in I}a_i\alpha_{-i}.$$ 
We clearly have $\al'\in Q_+$ and $\height(\al')=\height(\al)$. Moreover, $\bi\in I^\al$ if and only if $-\bi\in I^{\al'}$. 
Now, inspecting the relations, there is a
unique homogeneous algebra isomorphism 
\begin{equation}\label{epsilon}
\sgn:{R_\al}\longrightarrow{R_{\al'}},\quad e(\bi)\mapsto e(-\bi), \quad
  y_r\mapsto -y_r,\quad
  \psi_s\mapsto-\psi_s
\end{equation}
for all $\bi\in I^\al,\ 1\leq r\leq d$, and $1\leq s<d$, where $d=\height(\al)$. 

Recall  $\kappa=(\kappa_1,\dots,\kappa_l)$ from (\ref{EKappa}) and $\kappa'=(-\kappa_l,\dots,-\kappa_1)$ from (\ref{EKappaPrime}). 
Then, as in (\ref{ELa}),
$\kappa'$ determines the dominant weight
$$\Lambda'=\Lambda(\kappa')=\Lambda_{-\kappa_l}+\dots+\Lambda_{-\kappa_1}\in P_+.$$
Equivalently, if $\Lambda=\sum_{i\in I}l_i\Lambda_i$, then
$\Lambda'=\sum_{i\in I}l_i\Lambda_{-i}$. 

The algebra $R_\al^\Lambda$ is the quotient of $R_\al$ by the cyclotomic
relations~(\ref{ERCyc}). Applying the involution~$\sgn$ to (\ref{ERCyc}) we obtain
$$0=\sgn\Big(y_1^{(\Lambda,\alpha_{i_1})}e(\bi)\Big)
=\pm y_1^{(\Lambda,\alpha_{i_1})}e(-\bi) =\pm y_1^{(\Lambda',\alpha_{-i_1})}e(-\bi),$$ where
the right hand side is, up to sign, the cyclotomic relation for $R_{\al'}^{\Lambda'}$.  Hence $\sgn$ factors through
to a graded algebra isomorphism 
$$\sgn:R_{\al}^\Lambda\bijection R_{\alpha'}^{\Lambda'}.$$

The isomorphism $\sgn$ induces equivalences
$$\Mod{R_{\alpha'}}\bijection\Mod{R_\alpha}\quad\text{and}\quad
  \Mod{R_{\alpha'}^{\Lambda'}}\bijection\Mod{R_\alpha^\Lambda}$$
of the corresponding categories of graded modules. These equivalences send the
$R_{\al'}$-module $M$ to the $R_\al$-module $M^\sgn$, where $M^\sgn=M$ as
a graded vector space and where the $R_\al$-action on $M^\sgn$ is given by $a\cdot m=\sgn(a)m$, for $a\in R_\al$
and $m\in M^\sgn$.

\subsection{\boldmath Basis Theorem}
Suppose that $\al\in Q_+$ is of height~$d$.  For the rest of this
paper we fix a {\em preferred reduced decomposition} $w=s_{r_1}\dots s_{r_m}$ for each element $w\in \Si_d$, 
where $m\ge0$ is as small as possible and $1\le r_1,\dots,r_m<d$. Define the elements
$$
\psi_w:=\psi_{r_1}\dots\psi_{r_m}\in R_\al \qquad (w\in \Si_d).
$$
In general, $\psi_w$ depends on the choice of a preferred reduced decomposition of~$w$, but:

\begin{Proposition}\label{PSubtle}
Suppose that $\bi\in I^\al$, and 
$$w=s_{t_1}\dots s_{t_m}=s_{r_1}\dots s_{r_m}$$ 
are two reduced decompositions of an element $w\in \Si_d$. Then in $R_\al$, we have 
$$\psi_{t_1}\dots\psi_{t_m}e(\bi)=\psi_{r_1}\dots\psi_{r_m}e(\bi)+X,$$
where $X$ is a linear combination of elements of the form $\psi_uf(y)e(\bi)$ such that $u<w$, $f(y)$ is a polynomial in $y_1,\dots,y_d$,  and $$\deg(\psi_uf(y)e(\bi))=\deg(\psi_{r_1}\dots\psi_{r_m}e(\bi))=\deg(\psi_{t_1}\dots\psi_{t_m}e(\bi)).$$
\end{Proposition}

\begin{proof} This is proved in  \cite[Proposition 2.5]{BKW} for corresponding elements of the cyclotomic KLR
  algebra $R^\Lambda_\alpha$.  As the argument in~\cite{BKW} does not use the relation~(\ref{ERCyc}) the result
  holds in~$R_\alpha$.
\end{proof}

Suppose now that $\mu\in\Par_\al$ and that $\T$ is a $\mu$-tableau. In (\ref{EWT}) we defined the permutations
$w_\T,w^\T\in\Si_d$. Define
\begin{equation}\label{EPsiT}
  \psi^\T:=\psi_{w^\T}\quad\text{and}\quad
\psi_\T:=\psi_{w_\T}.
\end{equation}
These elements will be used to produce bases of various modules below. 

By (\ref{R3Psi}), there is one important case where the elements $\psi_w$ are independent of
the choice of preferred reduced decomposition of $w$. An element $w\in\Si_d$ is {\em fully
commutative} if one can go from any reduced decomposition of $w$ to any other reduced
decomposition of $w$ using only the commuting braid relations; that is, the relations of the form
$s_rs_t=s_ts_r$, for $|r-t|>1$. We refer the reader to \cite{Stembridge} for more details on
fully commutative elements. We record the following easy result for future reference:

\begin{Lemma} \label{LFullComm}
Suppose that $1\leq s< k$  and  let $\D$ be the set of the minimal length left coset representatives of the parabolic subgroup $\Si_s\times\Si_{k-s}$ in the symmetric group $\Si_k$. Then every element of $\D$ is fully commutative. 
\end{Lemma}


In general we have the following important result:

\begin{Theorem}\label{TBasis}{\cite[Theorem 2.5]{KL1}}, \cite[Theorem 3.7]{R} 
Let $\al\in Q_+$. Then  
$$ \{\psi_w y_1^{m_1}\dots y_d^{m_d}e(\bi)\mid w\in \Si_d,\ m_1,\dots,m_d\in\Z_{\geq 0}, \ \bi\in I^\al\}
$$ 
is an $\O$-basis of  $R_\al$. 
\end{Theorem}

\subsection{Induction and restriction for affine KLR algebras}
Given $\alpha, \beta \in Q_+$, we set $
R_{\alpha,\beta} := R_\alpha \otimes 
R_\beta,$ 
viewed as an algebra in the usual way. 
Let $M \boxtimes N$ be 
the outer tensor product of the $R_\alpha$-module $M$ and the $R_\beta$-module 
$N$.
There is an obvious injective homogeneous (non-unital) algebra homomorphism 
$R_{\alpha,\beta}\,\into\, R_{\alpha+\beta}$ 
mapping $e(\bi) \otimes e(\bj)$ to $e(\bi\bj)$,
where $\bi\bj$ is the concatenation of the two sequences. The image of the identity
element of $R_{\alpha,\beta}$ under this map is
$
e_{\alpha,\beta}:= \sum_{\bi \in I^\alpha,\ \bj \in I^\beta} e(\bi\bj).
$ 
Let $\Ind_{\alpha,\beta}^{\alpha+\beta}$ and $\Res_{\alpha,\beta}^{\alpha+\beta}$
be the corresponding induction and restriction functors between the corresponding categories of {\em graded}
modules: 
\begin{align*}
\Ind_{\alpha,\beta}^{\alpha+\beta} &:= R_{\alpha+\beta} e_{\alpha,\beta}
\otimes_{R_{\alpha,\beta}} ?:\Mod{R_{\alpha,\beta}} \rightarrow \Mod{R_{\alpha+\beta}},\\
\Res_{\alpha,\beta}^{\alpha+\beta} &:= e_{\alpha,\beta} R_{\alpha+\beta}
\otimes_{R_{\alpha+\beta}} ?:\Mod{R_{\alpha+\beta}}\rightarrow \Mod{R_{\alpha,\beta}}.
\end{align*}
These have obvious generalizations to $n\geq 2$ factors: 
\begin{align*}
\Ind_{\beta_1,\dots,\beta_n}^{\beta_1+\dots+\beta_n} :\Mod{R_{\beta_1,\dots,\beta_n}} \rightarrow \Mod{R_{\beta_1+\dots+\beta_n}},\\
\Res_{\beta_1,\dots,\beta_n}^{\beta_1+\dots+\beta_n} :\Mod{R_{\beta_1+\dots+\beta_n}}\rightarrow \Mod{R_{\beta_1,\dots,\beta_n}}.
\end{align*}
The functor $\Res_{\beta_1,\dots,\beta_n}^{\beta_1+\dots+\beta_n}$ is left multiplication by
the idempotent $e_{\beta_1,\dots,\beta_n}$, so it is exact and sends finite dimensional modules to
finite dimensional modules. 
The functor $\Ind_{\beta_1,\dots,\beta_n}^{\beta_1+\dots+\beta_n}$ is left adjoint to $\Res_{\beta_1,\dots,\beta_n}^{\beta_1+\dots+\beta_n}$. 
Moreover,  $R_{\beta_1+\dots+\beta_n} e_{\beta_1,\dots,\beta_n}$ is a
free graded right $R_{\beta_1,\dots,\beta_n}$-module of finite rank, so
$\Ind_{\beta_1,\dots,\beta_n}^{\beta_1+\dots+\beta_n}$ sends finite dimensional graded modules to finite dimensional graded modules. Finally, if $M_a\in\Mod{R_{\beta_a}}$, for $a=1,\dots,n$, we define 
\begin{equation}\label{ECircProd}
M_1\circ\dots\circ M_n:=\Ind_{\beta_1,\dots,\beta_n}^{\beta_1+\dots+\beta_n}M_1\boxtimes\dots\boxtimes M_n. 
\end{equation}

\subsection{\boldmath Permutation Modules $M(\vec{\bs})$}\label{SSPerm}
For $i\in I$, and $N\in \Z_{\geq 1}$, let $\bs(i,N)\in I^N$ be the tuple $(j_1,\dots,j_N)$ with $j_r=i+r-1\pmod
e$. In other words, $\bs(i,N)$ is the {\em segment} of length $N$ starting at $i$. Similarly, if $N\in\Z_{<0}$
define $\bs(i,N)\in I^{-N}$ be the tuple $(j_1,\dots,j_{-N})$ with $j_r=i-r+1\pmod e$. For example
$\bs(0,e)=(0,1,\dots,e-1)$ and $\bs(0,-e)=(0,-1,\dots,1-e)$. 

Suppose that $\bs:=\bs(i,N)$ is a segment and let $\al=|\bs|\in Q_+$.  Define the corresponding {\em segment module}
$M(\bs):=\O\cdot m(\bs)$ to be the graded $R_\al$-module which is the free $\O$-module of rank one on the
generator $m(\bs)$ of degree $0$ with action  
\begin{equation*}
e(\bi)m(\bs)=\de_{\bi,\bs}m(\bs), \quad \psi_rm(\bs)=0\quad\text{ and }\quad y_tm(\bs)=0
\end{equation*}
for all admissible $\bi$, $r$ and $t$.  Equivalently, $M(\bs)=R_\al/K(\bs)$, where $K(\bs)$ is
the left ideal of $R_\al$ generated by the elements $e(\bi)-\de_{\bi,\bs}$, $\psi_r$, and $y_t$, for all
admissible $\bi$, $r$ and $t$. 

Let $\vec{\bs}=(\bs(1),\dots,\bs(n))$ be an ordered tuple of segments. Set $\al_r:=|\bs(r)|$, and let $\la_r:=\height(\al_r)$ be the length of the segment $\bs(r)$, for $r=1,\dots,n$. Also set $\al=\al_1+\dots+\al_n$ and $d:=\height(\al)$. Note that $(\la_1,\dots,\la_n)$ is a composition of $d$. Define the {\em permutation module} 
$$
M(\vec{\bs})=M(\bs(1),\dots,\bs(n)):=M(\bs(1))\circ\dots\circ M(\bs(n)).
$$
This is the graded $R_{\al}$-module generated by the vector
\begin{equation}\label{EZ}
m(\vec{\bs}):=1\otimes m(\bs(1))\otimes\dots\otimes m(\bs(n)) 
\end{equation}
in $\Ind^{\al}_{\al_1,\dots,\al_n}M(\bs(1))\boxtimes\dots\boxtimes M(\bs(n))$. 
Let $\Si_{\vec{\bs}}$ be the parabolic subgroup $\Si_{\la_1}\times\dots\times\Si_{\la_n}$ of~$\Si_d$.  
Define 
\begin{equation}\label{EBJ}
\bj(\vec{\bs}):=\bs(1)\dots\bs(n)\in I^\al
\end{equation}
where the product on the right hand side is the concatenation. Now let 
$K(\vec{\bs})$ to be the left ideal of $R_\alpha$ generated by
$$\{e(\bi)-\delta_{\bi,\bj(\vec{\bs})}, y_r, \psi_t \mid \bi\in I^\alpha,\ 1\le r\le d,\ 1\leq t<d 
      \text{ such that $s_t\in\Si_{\vec{\bs}}$}\}.
$$
Then we have:
\begin{equation}\label{E:M(s)}
M(\vec{\bs})\cong R_\al/K(\vec{\bs}).
\end{equation}
Under this isomorphism $m(\vec{\bs})$ is identified with $1+K(\vec{\bs})$. With the notation as above, we have as an immediate consequence of the Basis Theorem~\ref{TBasis}:

\begin{Theorem} \label{TMBasis}
Let ${\mathscr D}_{\vec{\bs}}$ be the set of the shortest length left coset representatives of $\Si_{\vec{\bs}}$ in $\Si_d$. Then 
$
\{\psi_{w}m(\vec{\bs})\mid w\in {\mathscr D}_{\vec{\bs}}\}
$ 
is an $\O$-basis of $M(\vec{\bs})$. Moreover each basis element $\psi_{w}m(\vec{\bs})$ is homogeneous of degree equal to the degree of the element $\psi_w e(\bj(\vec{\bs}))\in R_\al$, and $\psi_{w}m(\vec{\bs})\in e(w\cdot\bj(\vec{\bs}))M(\vec{\bs})$.
\end{Theorem}

\section{Block intertwiners}\label{SBI}
Throughout this section we assume that $e>0$. 
Recall from (\ref{E:nullRoot}) that $\de\in Q_+$ is the null-root.  
We fix $i\in I$ and a {\em composition} $\la=(\la_1,\dots,\la_n)$ of $k$. Define
$$\vec{\bs}(i,\la):=(\bs(i,e\la_1),\dots,\bs(i,e\la_n))$$ 
the tuple of segment of lengths $e\la_1,\dots,e\la_n$, all starting at $i$. We consider the corresponding permutation module 
$$M(i,\la):=M(\vec{\bs}(i,\la))$$ 
over the algebra $R_{k\de}$ as in section~\ref{SSPerm}. Let 
$$\bj=(j_1,\dots,j_{ke}):=\bj(\vec{\bs}(i,\la))\in I^{k\de}$$ 
as defined in (\ref{EBJ}). We have 
$\bj=\bs(i,ke). 
$
Finally, let the corresponding idempotent be 
$$
e(i,\la):=e(\bj(\vec{\bs}(i,\la)))\in R_{k\de}
$$ 
and 
$$m(i,\la):=m(\vec{\bs}(i,\la))\in M(i,\la),$$ 
the generator of $M(\bs)$ as in (\ref{EZ}).

\subsection{\boldmath The elements $\si$}\label{SSSi}
We consider the element $w_r$ of the symmetric group $\Si_{ke}$ defined as the product of transpositions
\begin{equation}\label{w_r}
w_r:=\prod_{a=re-e+1}^{re}(a,a+e)\qquad(1\leq r<k).
\end{equation}
Informally, $w_r$ permutes the $r$th ``$e$-block" and the $(r+1)$st ``$e$-block". If we write $w_r=w_r's_{re}$ then $\ell(w_r)=\ell(w_r')+1$. 

Define
\begin{equation}\label{ESi_r}
\si_r:=\psi_{w_r}e(i,\la)\in R_{k\de}
\qquad(1\leq r<k).
\end{equation}
Note by Lemma~\ref{LFullComm} that $w_r$ and $w_r'$ are fully commutative elements so the elements 
$\psi_{w_r}$ and $\psi_{w_r'}$ of $R_{k\de}$ do not depend on the choice of preferred reduced decompositions for these
permutations. Furthermore, $\psi_{w_r}=\psi_{w_r'}\psi_{re}$. 

To prove the results in this section we will use the graphical representation of elements of $R_{k\de}$ and $M(i,\la)=R_{k\de}/K(\vec{\bs}(i,\la))=R_{k\de}m(i,\la)$ following \cite{KL1}. In fact, the diagram used to represent an element $he(i,\la)\in R_{k\de}$ in \cite{KL1}, here will represent the element $hv\in M(i,\la)$, for some $v\in e(i,\la)M(i,\la)$. Of course, the element $v$ needs to be specified before this makes sense.   
For example, if $v=m(i,\la)$, then 
$$
m(i,\la)=
\begin{braid}\tikzset{baseline=7mm}
  \draw (0,4)node[above]{$j_1$}--(0,0);
  \draw (1,4)node[above]{$j_2$}--(1,0);
  \draw[dots] (1.2,4)--(3.8,4);
  \draw[dots] (1.2,0)--(3.8,0);
  \draw (4,4)node[above]{$j_{ke}$}--(4,0);
\end{braid},
\quad \psi_rm(i,\la)=
\begin{braid}\tikzset{baseline=7mm}
  \draw (0,4)node[above]{$j_1$}--(0,0);
  \draw[dots] (0.2,4)--(1.8,4);
  \draw[dots] (0.2,0)--(1.8,0);
  \draw (2,4)node[above]{$j_{r-1}$}--(2,0);
  \draw (3,4)node[above]{$j_r$}--(4,0);
  \draw (4,4)node[above]{$j_{r+1}$}--(3,0);
  \draw (5,4)node[above]{}--(5,0);
  \draw[dots] (5.2,4)--(6.8,4);
  \draw[dots] (5.2,0)--(6.8,0);
  \draw (7,4)node[above]{$j_{ke}$}--(7,0);
\end{braid},
$$
and 
$$
y_sm(i,\la) =
\begin{braid}\tikzset{baseline=7mm}
  \draw (0,4)node[above]{$j_1$}--(0,0);
  \draw[dots] (0.2,4)--(1.8,4);
  \draw[dots] (0.2,0)--(1.8,0);
  \draw (2,4)node[above]{$j_{s-1}$}--(2,0);
  \draw (3,4)node[above]{$j_s$}--(3,0);
  \greendot(3,2);
  \draw (4,4)node[above]{$j_s$}--(4,0);
  \draw[dots] (4.2,4)--(5.8,4);
  \draw[dots] (4.2,0)--(5.8,0);
  \draw (6,4)node[above]{$j_{ke}$}--(6,0);
\end{braid},
$$
where $1\le r<d$ and $1\le s\le d$. 
Also, setting $r'=(r-1)e$, $r''=(r+1)e+1$, we have
\begin{equation}\label{sigma}
\si_rm(i,\la)=
\begin{braid}\tikzset{baseline=7mm}
  \draw(0,4)node[above]{$j_1$}--(0,0);
  \draw[dots](0.2,4)--(1.8,4);
  \draw[dots](0.2,0)--(1.8,0);
  \draw (2,4)node[above]{$j_{r'}$}--(2,0);
  \draw (3,4)node[above]{$i$}--(7,0);
  \draw (4,4)node[above]{$i{+}1$}--(8,0);
  \draw[dots](4.2,4)--(5.8,4);
  \draw[dots](4.2,0)--(5.8,0);
  \draw (6,4)node[above]{$i{-}1$}--(10,0);
  \draw (7,4)node[above]{$i$}--(3,0);
  \draw (8,4)node[above]{$i{+}1$}--(4,0);
  \draw[dots](8.2,4)--(9.8,4);
  \draw[dots](8.2,0)--(9.8,0);
  \draw (10,4)node[above]{$i{-}1\,$}--(6,0);
  \draw (11,4)node[above]{$\,j_{r''}$}--(11,0);
  \draw[dots](11.2,4)--(12.8,4);
  \draw[dots](11.2,0)--(12.8,0);
  \draw (13,4)node[above]{$j_{ke}$}--(13,0);
\end{braid}
\end{equation}
We will 
colour the strings of the diagrams to improve readability, but these colours will have no
mathematical meaning (and will not be distinguishable in black and white!).


\subsection{\boldmath The block permutation subspace}\label{SSBlPermSp}
Consider the {\em block permutation subspace} 
$$
T(i,\la):=\O\text{-span}\{\si_{r_1}\si_{r_2}\dots\si_{r_a}m(i,\la)\mid 1\le r_1,\dots,r_a<k\}.
$$
It is not hard to see using Theorem~\ref{TMBasis} that 
$$
T(i,\la)= e(i,\la)M(i,\la).
$$

%
It is easy to see that $\deg(\si_re(i,\la))=0$. Therefore, 
$$T(i,\la)\subseteq M(i,\la)_0,$$ 
the degree zero component of $M(i,\la)$. 

\begin{Lemma}\label{LNew}
Suppose that $1\le s, t\le ke$ with  $t\not\equiv0\pmod e$. Then the elements $y_s$ and $\psi_t$ act as zero on $T(i,\la)$.
\end{Lemma}
\begin{proof}
Let $v\in T(i,\la)$. We have 
$$y_sv\in e(i,\la)M(i,\la)=T(i,\la)\subseteq M(i,\la)_0.$$ 
On the other hand, $\deg(y_sv)=\deg(y_s)+\deg(v)=2$. Hence $y_sv=0$. Moreover, $\psi_tv\in e(s_t\cdot\bj)M(i,\la)=0$, the last equality holding by Theorem~\ref{TMBasis}. 
\end{proof}

\subsection{Quadratic relation} 
We want to study relations satisfied by the elements $\si_r$ acting on  $T(i,\la)$. Our main goal is to show that the elements $\tau_r:=\si_r+1$ satisfy the Coxeter relations on $T(i,\la)$. We begin with the quadratic relations.

\begin{Lemma} \label{LPsiSi}
Suppose that $1\leq r<k$ and $v\in T(i,\la)$. Then 
$
\psi_{re}\si_rv=-2\psi_{re}v.
$
Equivalently, in terms of diagrams we have
$$
\begin{braid}\tikzset{yscale=0.5,baseline=-0mm}
  \draw(0,4)node[above]{$j_1$}--(0,0)--(0,-4);
  \draw[dots](0.2,4)--(1.8,4);
  \draw[dots](0.2,0)--(1.8,0);
  \draw[dots](0.2,-4)--(1.8,-4);
  \draw (2,4)node[above]{$j_{r'}$}--(2,0)--(2,-4);
  \draw (3,4)node[above=0.4mm]{$i$}--(7,0)--(6,-4);
  \draw (4,4)node[above]{$i{+}1$}--(8,0)--(8,-4);
  \draw[dots](4.2,4)--(5.8,4);
  \draw[dots](4.2,0)--(5.8,0);
  \draw[dots](4.2,-4)--(5.8,-4);
  \draw (6,4)node[above]{$i{-}1$}--(10,0)--(10,-4);
  \draw (7,4)node[above=0.4mm]{$i$}--(3,0)--(3,-4);
  \draw (8,4)node[above]{$i{+}1$}--(4,0)--(4,-4);
  \draw[dots](8.2,4)--(9.8,4);
  \draw[dots](8.2,0)--(9.8,0);
  \draw[dots](8.2,-4)--(9.8,-4);
  \draw (10,4)node[above]{$i{-}1\,$}--(6,0)--(7,-4);
  \draw (11,4)node[above]{$\,j_{r''}$}--(11,0)--(11,-4);
  \draw[dots](11.2,4)--(12.8,4);
  \draw[dots](11.2,0)--(12.8,0);
  \draw[dots](11.2,-4)--(12.8,-4);
  \draw (13,4)node[above]{$j_{ke}$}--(13,0)--(13,-4);
\end{braid}
=-2
\begin{braid}
  \draw(0,4)node[above]{$j_1$}--(0,0);
  \draw[dots](0.2,4)--(1.8,4);
  \draw[dots](0.2,0)--(1.8,0);
  \draw (2,4)node[above]{$j_{r'}$}--(2,0);
  \draw (3,4)node[above=0.4mm]{$i$}--(3,0);
  \draw (4,4)node[above]{$i{+}1$}--(4,0);
  \draw[dots](4.2,4)--(5.8,4);
  \draw[dots](4.2,0)--(5.8,0);
  \draw (6,4)node[above]{$i{-}1$}--(7,0);
  \draw (7,4)node[above=0.4mm]{$i$}--(6,0);
  \draw (8,4)node[above]{$i{+}1$}--(8,0);
  \draw[dots](8.2,4)--(9.8,4);
  \draw[dots](8.2,0)--(9.8,0);
  \draw (10,4)node[above]{$i{-}1\,$}--(10,0);
  \draw (11,4)node[above]{$\,j_{r''}$}--(11,0);
  \draw[dots](11.2,4)--(12.8,4);
  \draw[dots](11.2,0)--(12.8,0);
  \draw (13,4)node[above]{$j_{ke}$}--(13,0);
\end{braid}
$$
where $r'=(r-1)e$, $r''=(r+1)e+1$ and $\bi=\bs(i,ke)$.
\end{Lemma}

\begin{proof} For typographical convenience, we only consider the case where $r=1$ and
$i=0$. 
We first treat the case $e=2$ which is exceptional because in this case the quiver $\Gamma$ is not simply
laced.  Using the relation (\ref{quad}), and then (\ref{ypsi}) and (\ref{psiy}), we have: 
\begin{align*}
\psi_{2}\si_1&=
  \begin{braid}\tikzset{scale=0.8}
    \draw(0,4)node[above]{0}--(2,2)--(2,1)--(1,0);
    \draw(1,4)node[above]{1}--(3,2)--(3,1)--(3,0);
    \draw(2,4)node[above]{0}--(0,2)--(0,1)--(0,0);
    \draw(3,4)node[above]{1}--(1,2)--(1,1)--(2,0);
  \end{braid}
=2\begin{braid}\tikzset{scale=0.8}
    \draw(0,4)node[above]{0}--(1,2)--(1,0);
    \draw(1,4)node[above]{1}--(3,2)--(3,1)--(3,0);
    \draw(2,4)node[above]{0}--(0,2)--(0,1)--(0,0);
    \draw(3,4)node[above]{1}--(2,2)--(2,0);
    \greendot(1,1);
    \greendot(2,1);
  \end{braid}
 -\begin{braid}\tikzset{scale=0.8}
    \draw(0,4)node[above]{0}--(1,2)--(1,0);
    \draw(1,4)node[above]{1}--(3,2)--(3,1)--(3,0);
    \draw(2,4)node[above]{0}--(0,2)--(0,1)--(0,0);
    \draw(3,4)node[above]{1}--(2,2)--(2,0);
    \greendot(1,0.5);
    \greendot(1,1.5);
  \end{braid}
 -\begin{braid}\tikzset{scale=0.8}
    \draw(0,4)node[above]{0}--(1,2)--(1,0);
    \draw(1,4)node[above]{1}--(3,2)--(3,1)--(3,0);
    \draw(2,4)node[above]{0}--(0,2)--(0,1)--(0,0);
    \draw(3,4)node[above]{1}--(2,2)--(2,0);
    \greendot(2,0.5);
    \greendot(2,1.5);
  \end{braid}
\\& = 2\begin{braid}\tikzset{scale=0.8}
    \draw(0,4)node[above]{0}--(1,2)--(1,0);
    \draw(1,4)node[above]{1}--(3,2)--(3,1)--(3,0);
    \draw(2,4)node[above]{0}--(0,2)--(0,1)--(0,0);
    \draw(3,4)node[above]{1}--(2,2)--(2,0);
    \greendot(1,1);
    \greendot(2,1);
  \end{braid}
=2\begin{braid}\tikzset{scale=0.8}
    \draw(0,4)node[above]{0}--(0,0);
    \draw(1,4)node[above]{1}--(3,2)--(3,1)--(3,0);
    \draw(2,4)node[above]{0}--(1,2)--(1,0);
    \draw(3,4)node[above]{1}--(2,2)--(2,0);
    \greendot(2,1);
  \end{braid}
=-2\begin{braid}\tikzset{scale=0.8}
    \draw(0,4)node[above]{0}--(0,0);
    \draw(1,4)node[above]{1}--(2,2)--(2,0);
    \draw(2,4)node[above]{0}--(1,2)--(1,0);
    \draw(3,4)node[above]{1}--(3,0);
  \end{braid}
  =-2\psi_2v,
\end{align*}
as required.

Now suppose that $e>2$. 
To start, using (\ref{quad}) we see that $\psi_e\sigma_1v$ equals 
\begin{align*}
-\begin{braid}
  \draw (1,3) node[above]{$0$}--(7,0)--(7,-1);
  \draw (2,3) node[above]{$1$}--(9,0)--(9,-1);
  \draw (3,3) node[above]{$2$}--(10,0)--(10,-1);
  \draw (4,3) node[above]{$3$}--(11,0)--(11,-1);
  \draw[dots] (4.3,3)--(5.8,3);
  \draw[dots] (4.3,0)--(5.8,0);
  \draw[dots] (4.3,-1)--(5.8,-1);
  \draw (6,3) node[above]{$-2$}--(13,0)--(13,-1);
  \draw (7,3) node[above]{$-1$}--(14,0)--(14,-1);
  \draw (8,3) node[above]{$0$}--(1,0)--(1,-1);
  \draw (9,3) node[above]{$1$}--(2,0)--(2,-1);
  \draw (10,3) node[above]{$2$}--(3,0)--(3,-1);
  \draw (11,3) node[above]{$3$}--(4,0)--(4,-1);
  \draw[dots] (11.3,3)--(12.6,3);
  \draw[dots] (11.3,0)--(12.6,0);
  \draw[dots] (11.3,-1)--(12.6,-1);
  \draw (13,3) node[above]{$-2$} -- (6,0) --  (6,-1);
  \draw (14,3) node[above]{$-1$} -- (8,0) --  (8,-1);
  \greendot(7,-0.5);
\end{braid}
+
\begin{braid}
  \draw (1,3) node[above]{$0$}--(7,0)--(7,-1);
  \draw (2,3) node[above]{$1$}--(9,0)--(9,-1);
  \draw (3,3) node[above]{$2$}--(10,0)--(10,-1);
  \draw (4,3) node[above]{$3$}--(11,0)--(11,-1);
  \draw[dots] (4.3,3)--(5.8,3);
  \draw[dots] (4.3,0)--(5.8,0);
  \draw[dots] (4.3,-1)--(5.8,-1);
  \draw (6,3) node[above]{$-2$}--(13,0)--(13,-1);
  \draw (7,3) node[above]{$-1$}--(14,0)--(14,-1);
  \draw (8,3) node[above]{$0$}--(1,0)--(1,-1);
  \draw (9,3) node[above]{$1$}--(2,0)--(2,-1);
  \draw (10,3) node[above]{$2$}--(3,0)--(3,-1);
  \draw (11,3) node[above]{$3$}--(4,0)--(4,-1);
  \draw[dots] (11.3,3)--(12.6,3);
  \draw[dots] (11.3,0)--(12.6,0);
  \draw[dots] (11.3,-1)--(12.6,-1);
  \draw (13,3) node[above]{$-2$} -- (6,0) --  (6,-1);
  \draw (14,3) node[above]{$-1$} -- (8,0) --  (8,-1);
  \greendot(8,-0.5);
\end{braid}.
\end{align*}

Let $D_1$ be the first diagram and let $D_2$ be the second diagram. To complete the proof, we show that   $D_1=\psi_ev$ and  $D_2=-\psi_e v$. 
In fact, the two equalities are proved similarly, so we give details only for the first one. 
Using (\ref{ypsi}), we see that 
\begin{align*}
D_1=&\begin{braid}\tikzset{scale=1.2}
  \draw (1,3) node[above]{$0$}--(7,0)--(7,-1);
  \draw (2,3) node[above]{$1$}--(9,0)--(9,-1);
  \draw (3,3) node[above]{$2$}--(10,0)--(10,-1);
  \draw (4,3) node[above]{$3$}--(11,0)--(11,-1);
  \draw[dots] (4.3,3)--(5.8,3);
  \draw[dots] (4.3,0)--(5.8,0);
  \draw[dots] (4.3,-1)--(5.8,-1);
  \draw (6,3) node[above]{$-2$}--(13,0)--(13,-1);
  \draw (7,3) node[above]{$-1$}--(14,0)--(14,-1);
  \draw (8,3) node[above]{$0$}--(1,0)--(1,-1);
  \draw (9,3) node[above]{$1$}--(2,0)--(2,-1);
  \draw (10,3) node[above]{$2$}--(3,0)--(3,-1);
  \draw (11,3) node[above]{$3$}--(4,0)--(4,-1);
  \draw[dots] (11.3,3)--(12.6,3);
  \draw[dots] (11.3,0)--(12.6,0);
  \draw[dots] (11.3,-1)--(12.6,-1);
  \draw (13,3) node[above]{$-2$} -- (6,0) --  (6,-1);
  \draw (14,3) node[above]{$-1$} -- (8,0) --  (8,-1);
  \greendot(4.45,1.3);
\end{braid}
\\=&
\begin{braid}\tikzset{baseline=3mm}
  \draw[red] (1,3) node[above]{$0$}--(4.0,1.375)--(1,0)--(1,-1);
  \draw (2,3) node[above]{$1$}--(9,0)--(9,-1);
  \draw (3,3) node[above]{$2$}--(10,0)--(10,-1);
  \draw (4,3) node[above]{$3$}--(11,0)--(11,-1);
  \draw[dots] (4.3,3)--(5.8,3);
  \draw[dots] (4.3,0)--(5.8,0);
  \draw[dots] (4.3,-1)--(5.8,-1);
  \draw (6,3) node[above]{$-2$}--(13,0)--(13,-1);
  \draw (7,3) node[above]{$-1$}--(14,0)--(14,-1);
  \draw[red] (8,3) node[above]{$0$}--(4.25,1.375)--(7,0)--(7,-1);
  \draw (9,3) node[above]{$1$}--(2,0)--(2,-1);
  \draw (10,3) node[above]{$2$}--(3,0)--(3,-1);
  \draw (11,3) node[above]{$3$}--(4,0)--(4,-1);
  \draw[dots] (11.3,3)--(12.6,3);
  \draw[dots] (11.3,0)--(12.6,0);
  \draw[dots] (11.3,-1)--(12.6,-1);
  \draw (13,3) node[above]{$-2$} -- (6,0) --  (6,-1);
  \draw (14,3) node[above]{$-1$} -- (8,0) --  (8,-1);
\end{braid}
+
\begin{braid}\tikzset{baseline=3mm}
  \draw (1,3) node[above]{$0$}--(7,0)--(7,-1);
  \greendot(3.5,1.75);
  \draw (2,3) node[above]{$1$}--(9,0)--(9,-1);
  \draw (3,3) node[above]{$2$}--(10,0)--(10,-1);
  \draw (4,3) node[above]{$3$}--(11,0)--(11,-1);
  \draw[dots] (4.3,3)--(5.8,3);
  \draw[dots] (4.3,0)--(5.8,0);
  \draw[dots] (4.3,-1)--(5.8,-1);
  \draw (6,3) node[above]{$-2$}--(13,0)--(13,-1);
  \draw (7,3) node[above]{$-1$}--(14,0)--(14,-1);
  \draw (8,3) node[above]{$0$}--(1,0)--(1,-1);
  \draw (9,3) node[above]{$1$}--(2,0)--(2,-1);
  \draw (10,3) node[above]{$2$}--(3,0)--(3,-1);
  \draw (11,3) node[above]{$3$}--(4,0)--(4,-1);
  \draw[dots] (11.3,3)--(12.6,3);
  \draw[dots] (11.3,0)--(12.6,0);
  \draw[dots] (11.3,-1)--(12.6,-1);
  \draw (13,3) node[above]{$-2$} -- (6,0) --  (6,-1);
  \draw (14,3) node[above]{$-1$} -- (8,0) --  (8,-1);
\end{braid}.
\end{align*}
The second summand is zero as $y_1v=0$ by Lemma~\ref{LNew}. Applying the braid relations (\ref{braid}) to the first summand, we get that $D_1$ equals 
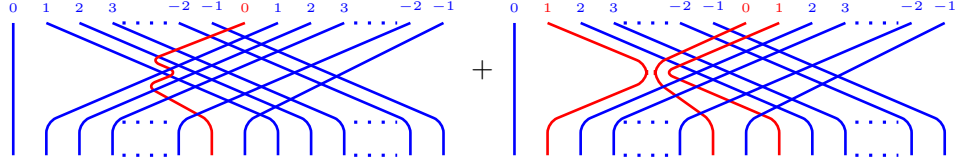
\begin{align*}
\begin{braid}\tikzset{scale=1.1}
  \draw (1,3) node[above]{$0$}--(1,-1);
  \draw (2,3) node[above]{$1$}--(9,0)--(9,-1);
  \draw (3,3) node[above]{$2$}--(10,0)--(10,-1);
  \draw (4,3) node[above]{$3$}--(11,0)--(11,-1);
  \draw[dots] (4.3,3)--(5.8,3);
  \draw[dots] (4.3,0)--(5.8,0);
  \draw[dots] (4.3,-1)--(5.8,-1);
  \draw (6,3) node[above]{$-2$}--(13,0)--(13,-1);
  \draw (7,3) node[above]{$-1$}--(14,0)--(14,-1);
  \draw[red] (8,3) node[above]{$0$}--(5.125,1.875)--(6.0,1.5)--(5,1.1)--(7,0)--(7,-1);
  \draw (9,3) node[above]{$1$}--(2,0)--(2,-1);
  \draw (10,3) node[above]{$2$}--(3,0)--(3,-1);
  \draw (11,3) node[above]{$3$}--(4,0)--(4,-1);
  \draw[dots] (11.3,3)--(12.6,3);
  \draw[dots] (11.3,0)--(12.6,0);
  \draw[dots] (11.3,-1)--(12.6,-1);
  \draw (13,3) node[above]{$-2$} -- (6,0) --  (6,-1);
  \draw (14,3) node[above]{$-1$} -- (8,0) --  (8,-1);
\end{braid}
+
\begin{braid}\tikzset{scale=1.1}
  \draw (1,3) node[above]{$0$}--(1,-1);
  \draw[red] (2,3) node[above]{$1$}--(5,1.75)--(5,1.25)--(2,0)--(2,-1);
  \draw (3,3) node[above]{$2$}--(10,0)--(10,-1);
  \draw (4,3) node[above]{$3$}--(11,0)--(11,-1);
  \draw[dots] (4.3,3)--(5.8,3);
  \draw[dots] (4.3,0)--(5.8,0);
  \draw[dots] (4.3,-1)--(5.8,-1);
  \draw (6,3) node[above]{$-2$}--(13,0)--(13,-1);
  \draw (7,3) node[above]{$-1$}--(14,0)--(14,-1);
  \draw[red] (8,3) node[above]{$0$}--(5.25,1.75)--(5.25,1.25)--(7,0)--(7,-1);
  \draw[red] (9,3) node[above]{$1$}--(5.5,1.5)--(9,0)--(9,-1);
  \draw (10,3) node[above]{$2$}--(3,0)--(3,-1);
  \draw (11,3) node[above]{$3$}--(4,0)--(4,-1);
  \draw[dots] (11.3,3)--(12.6,3);
  \draw[dots] (11.3,0)--(12.6,0);
  \draw[dots] (11.3,-1)--(12.6,-1);
  \draw (13,3) node[above]{$-2$} -- (6,0) --  (6,-1);
  \draw (14,3) node[above]{$-1$} -- (8,0) --  (8,-1);
\end{braid}.
\end{align*}
Using the braid relations to pull the second $0$-string through shows that the first summand equals 
$$
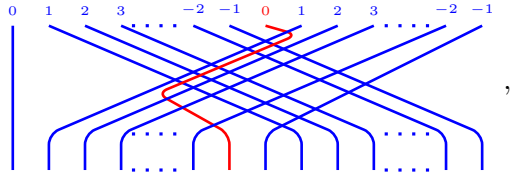
\begin{braid}\tikzset{scale=1.2}
  \draw (1,3) node[above]{$0$}--(1,-1);
  \draw (2,3) node[above]{$1$}--(9,0)--(9,-1);
  \draw (3,3) node[above]{$2$}--(10,0)--(10,-1);
  \draw (4,3) node[above]{$3$}--(11,0)--(11,-1);
  \draw[dots] (4.3,3)--(5.8,3);
  \draw[dots] (4.3,0)--(5.8,0);
  \draw[dots] (4.3,-1)--(5.8,-1);
  \draw (6,3) node[above]{$-2$}--(13,0)--(13,-1);
  \draw (7,3) node[above]{$-1$}--(14,0)--(14,-1);
  \draw[red] (8,3) node[above]{$0$}--(8.875,2.75)--(5,1.125)--(7,0)--(7,-1);
  \draw (9,3) node[above]{$1$}--(2,0)--(2,-1);
  \draw (10,3) node[above]{$2$}--(3,0)--(3,-1);
  \draw (11,3) node[above]{$3$}--(4,0)--(4,-1);
  \draw[dots] (11.3,3)--(12.6,3);
  \draw[dots] (11.3,0)--(12.6,0);
  \draw[dots] (11.3,-1)--(12.6,-1);
  \draw (13,3) node[above]{$-2$} -- (6,0) --  (6,-1);
  \draw (14,3) node[above]{$-1$} -- (8,0) --  (8,-1);
\end{braid},$$
showing that this element is zero 
since $\psi_{e+1}v=0$, by Lemma~\ref{LNew}. Applying the braid relations to the second summand, we get  
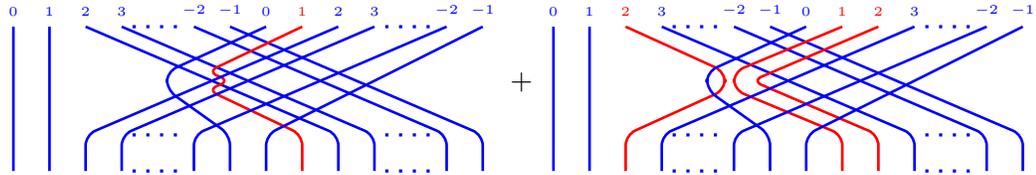
\begin{align*}
\begin{braid}\tikzset{scale=1.2}
  \draw (1,3) node[above]{$0$}--(1,-1);
  \draw (2,3) node[above]{$1$}--(2,-1);
  \draw (3,3) node[above]{$2$}--(10,0)--(10,-1);
  \draw (4,3) node[above]{$3$}--(11,0)--(11,-1);
  \draw[dots] (4.3,3)--(5.8,3);
  \draw[dots] (4.3,0)--(5.8,0);
  \draw[dots] (4.3,-1)--(5.8,-1);
  \draw (6,3) node[above]{$-2$}--(13,0)--(13,-1);
  \draw (7,3) node[above]{$-1$}--(14,0)--(14,-1);
  \draw (8,3) node[above]{$0$}--(5.25,1.75)--(5.25,1.25)--(7,0)--(7,-1);
  \draw[red] (9,3) node[above]{$1$}--(6.375,1.75)--(7,1.5)--(6.375,1.25)--(9,0)--(9,-1);
  \draw(10,3) node[above]{$2$}--(3,0)--(3,-1);
  \draw (11,3) node[above]{$3$}--(4,0)--(4,-1);
  \draw[dots] (11.3,3)--(12.6,3);
  \draw[dots] (11.3,0)--(12.6,0);
  \draw[dots] (11.3,-1)--(12.6,-1);
  \draw (13,3) node[above]{$-2$} -- (6,0) --  (6,-1);
  \draw (14,3) node[above]{$-1$} -- (8,0) --  (8,-1);
\end{braid}
+
\begin{braid}\tikzset{scale=1.2}
  \draw (1,3) node[above]{$0$}--(1,-1);
  \draw (2,3) node[above]{$1$}--(2,-1);
  \draw[red] (3,3) node[above]{$2$}--(5.75,1.75)--(5.75,1.25)--(3,0)--(3,-1);
  \draw (4,3) node[above]{$3$}--(11,0)--(11,-1);
  \draw[dots] (4.3,3)--(5.8,3);
  \draw[dots] (4.3,0)--(5.8,0);
  \draw[dots] (4.3,-1)--(5.8,-1);
  \draw (6,3) node[above]{$-2$}--(13,0)--(13,-1);
  \draw (7,3) node[above]{$-1$}--(14,0)--(14,-1);
  \draw (8,3) node[above]{$0$}--(5.25,1.75)--(5.25,1.25)--(7,0)--(7,-1);
  \draw[red] (9,3) node[above]{$1$}--(6.0,1.75)--(6.0,1.25)--(9,0)--(9,-1);
  \draw[red] (10,3) node[above]{$2$}--(6.5,1.5)--(10,0)--(10,-1);
  \draw (11,3) node[above]{$3$}--(4,0)--(4,-1);
  \draw[dots] (11.3,3)--(12.6,3);
  \draw[dots] (11.3,0)--(12.6,0);
  \draw[dots] (11.3,-1)--(12.6,-1);
  \draw (13,3) node[above]{$-2$} -- (6,0) --  (6,-1);
  \draw (14,3) node[above]{$-1$} -- (8,0) --  (8,-1);
\end{braid}.
\end{align*}
As before, using the braid relations to pull the second $1$-string to the top of the
first diagram shows that the first summand is zero. So, by (\ref{quad}), 
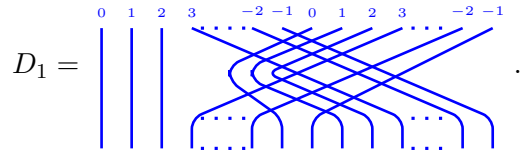
$$D_1=
\begin{braid}
  \draw (1,3) node[above]{$0$}--(1,-1);
  \draw (2,3) node[above]{$1$}--(2,-1);
  \draw (3,3) node[above]{$2$}--(3,-1);
  \draw (4,3) node[above]{$3$}--(11,0)--(11,-1);
  \draw[dots] (4.3,3)--(5.8,3);
  \draw[dots] (4.3,0)--(5.8,0);
  \draw[dots] (4.3,-1)--(5.8,-1);
  \draw (6,3) node[above]{$-2$}--(13,0)--(13,-1);
  \draw (7,3) node[above]{$-1$}--(14,0)--(14,-1);
  \draw (8,3) node[above]{$0$}--(5.25,1.75)--(5.25,1.25)--(7,0)--(7,-1);
  \draw (9,3) node[above]{$1$}--(6.0,1.75)--(6.0,1.25)--(9,0)--(9,-1);
  \draw (10,3) node[above]{$2$}--(6.5,1.5)--(10,0)--(10,-1);
  \draw (11,3) node[above]{$3$}--(4,0)--(4,-1);
  \draw[dots] (11.3,3)--(12.6,3);
  \draw[dots] (11.3,0)--(12.6,0);
  \draw[dots] (11.3,-1)--(12.6,-1);
  \draw (13,3) node[above]{$-2$} -- (6,0) --  (6,-1);
  \draw (14,3) node[above]{$-1$} -- (8,0) --  (8,-1);
\end{braid}.
$$
The argument so far has straightened the first three strings in the diagram. Continuing
in this way straightens the first $e-1$ strings so that
$$
D_1=
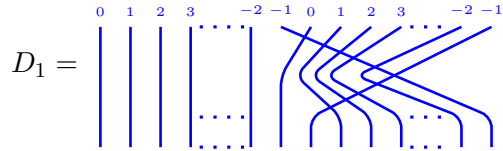
\begin{braid}
  \draw (1,3) node[above]{$0$}--(1,-1);
  \draw (2,3) node[above]{$1$}--(2,-1);
  \draw (3,3) node[above]{$2$}--(3,-1);
  \draw (4,3) node[above]{$3$}--(4,-1);
  \draw[dots] (4.3,3)--(5.8,3);
  \draw[dots] (4.3,0)--(5.8,0);
  \draw[dots] (4.3,-1)--(5.8,-1);
  \draw (6,3) node[above]{$-2$}--(6,-1);
  \draw (7,3) node[above]{$-1$}--(14,0)--(14,-1);
  \draw (8,3) node[above]{$0$}--(7,1.375)--(7,0)--(7,-1);
  \draw (9,3) node[above]{$1$}--(7.5,1.375)--(9,0)--(9,-1);
  \draw (10,3) node[above]{$2$}--(8,1.375)--(10,0)--(10,-1);
  \draw (11,3) node[above]{$3$}--(8.5,1.375)--(11,0)--(11,-1);
  \draw[dots] (11.3,3)--(12.6,3);
  \draw[dots] (11.3,0)--(12.6,0);
  \draw[dots] (11.3,-1)--(12.6,-1);
  \draw (13,3) node[above]{$-2$} -- (9.5,1.375)--(13,0) --  (13,-1);
  \draw (14,3) node[above]{$-1$} -- (8,0) --  (8,-1);
\end{braid}
$$
Now applying the braid relation for the last time shows that $D_1$ equals 
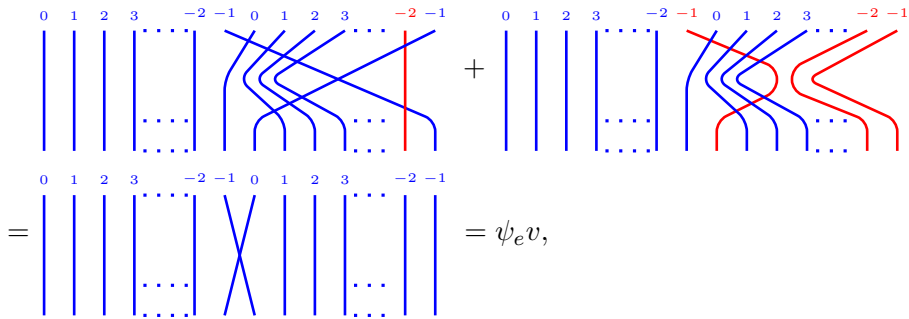
\begin{align*}
&\begin{braid}
  \draw (1,3) node[above]{$0$}--(1,-1);
  \draw (2,3) node[above]{$1$}--(2,-1);
  \draw (3,3) node[above]{$2$}--(3,-1);
  \draw (4,3) node[above]{$3$}--(4,-1);
  \draw[dots] (4.3,3)--(5.8,3);
  \draw[dots] (4.3,0)--(5.8,0);
  \draw[dots] (4.3,-1)--(5.8,-1);
  \draw (6,3) node[above]{$-2$}--(6,-1);
  \draw (7,3) node[above]{$-1$}--(14,0)--(14,-1);
  \draw (8,3) node[above]{$0$}--(7,1.375)--(7,0)--(7,-1);
  \draw (9,3) node[above]{$1$}--(7.5,1.375)--(9,0)--(9,-1);
  \draw (10,3) node[above]{$2$}--(8,1.375)--(10,0)--(10,-1);
  \draw (11,3) node[above]{$3$}--(8.5,1.375)--(11,0)--(11,-1);
  \draw[dots] (11.3,3)--(12.6,3);
  \draw[dots] (11.3,0)--(12.6,0);
  \draw[dots] (11.3,-1)--(12.6,-1);
  \draw[red] (13,3) node[above]{$-2$}--(13,-1);
  \draw (14,3) node[above]{$-1$}--(8,0)--(8,-1);
\end{braid}
+
\begin{braid}
  \draw (1,3) node[above]{$0$}--(1,-1);
  \draw (2,3) node[above]{$1$}--(2,-1);
  \draw (3,3) node[above]{$2$}--(3,-1);
  \draw (4,3) node[above]{$3$}--(4,-1);
  \draw[dots] (4.3,3)--(5.8,3);
  \draw[dots] (4.3,0)--(5.8,0);
  \draw[dots] (4.3,-1)--(5.8,-1);
  \draw (6,3) node[above]{$-2$}--(6,-1);
  \draw[red] (7,3) node[above]{$-1$}--(10,1.75)--(10,1)--(8,0)--(8,-1);
  \draw (8,3) node[above]{$0$}--(7,1.375)--(7,0)--(7,-1);
  \draw (9,3) node[above]{$1$}--(7.5,1.375)--(9,0)--(9,-1);
  \draw (10,3) node[above]{$2$}--(8,1.375)--(10,0)--(10,-1);
  \draw (11,3) node[above]{$3$}--(8.5,1.375)--(11,0)--(11,-1);
  \draw[dots] (11.3,3)--(12.6,3);
  \draw[dots] (11.3,0)--(12.6,0);
  \draw[dots] (11.3,-1)--(12.6,-1);
  \draw[red] (13,3) node[above]{$-2$} -- (10.5,1.75)--(10.5,1)--(13,0) --  (13,-1);
  \draw[red] (14,3) node[above]{$-1$} -- (11,1.375)--(14,0) --  (14,-1);
\end{braid}
\\=&
\begin{braid}
  \draw (1,3) node[above]{$0$}--(1,-1);
  \draw (2,3) node[above]{$1$}--(2,-1);
  \draw (3,3) node[above]{$2$}--(3,-1);
  \draw (4,3) node[above]{$3$}--(4,-1);
  \draw[dots] (4.3,3)--(5.8,3);
  \draw[dots] (4.3,0)--(5.8,0);
  \draw[dots] (4.3,-1)--(5.8,-1);
  \draw (6,3) node[above]{$-2$}--(6,-1);
  \draw (7,3) node[above]{$-1$}--(8,-1);
  \draw (8,3) node[above]{$0$}--(7,-1);
  \draw (9,3) node[above]{$1$}--(9,-1);
  \draw (10,3) node[above]{$2$}--(10,-1);
  \draw (11,3) node[above]{$3$}--(11,-1);
  \draw[dots] (11.3,3)--(12.6,3);
  \draw[dots] (11.3,0)--(12.6,0);
  \draw[dots] (11.3,-1)--(12.6,-1);
  \draw (13,3) node[above]{$-2$} -- (13,-1);
  \draw (14,3) node[above]{$-1$} --(14,-1);
\end{braid}
=\psi_ev,
\end{align*}
as required. 
\end{proof}

Recall from section~\ref{SSSi} that $\psi_{w_r}=\psi_{w_r'}\psi_{re}$.

\begin{Corollary} \label{CSiSq}
Suppose that $1\leq r<k$ and $v\in T(i,\la)$. Then 
$
\si_r^2v= -2\si_rv.
$
\end{Corollary}
\begin{proof}
Using Lemma~\ref{LPsiSi}, we get 
$$\si_r^2v=\psi_{w_r'}\psi_{re}\si_rv=-2\psi_{w_r'}\psi_{re}v
       =-2\psi_{w_r}v=-2\si_rv,
$$
as desired. 
\end{proof}

\subsection{Braid relations}
This section is dedicated to the proof of the following 

\begin{Theorem}\label{sigma braid}
  Suppose that $1\le r< k$ and $v\in T(i,\la)$. Then 
  $$(\sigma_r\sigma_{r+1}\sigma_r-
    \sigma_{r+1}\sigma_r\sigma_{r+1}
          - \sigma_r+\sigma_{r+1})v=0.$$
\end{Theorem}

In the proof, for typographical reasons, we assume that $i=0$ and  $k=3$ (this corresponds to ignoring vertical strings to the left and to the right of the relation we are interested in). 
As in the Lemma~\ref{LPsiSi}, all diagrams represent elements of $T(i,\la)$ obtained by applying the corresponding elements of $R_{k\de}$ to a given $v\in T(i,\la)$.

First, we need three technical lemmas.

\begin{Lemma}\label{sigma twist}
Suppose that  $e>2$, $k=3$, $i=0$, and $v\in T(i,\la)$. Then:
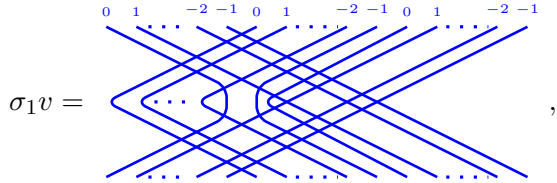
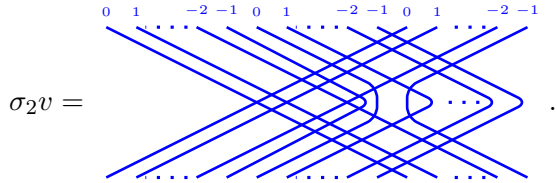
\begin{align*}
\sigma_1v&=
\begin{braid}\tikzset{baseline=9mm}
  \draw(0,5)node[above]{$0$}--(4,3)--(4,2)--(0,0);
  \draw(1,5)node[above]{$1$}--(11,0);
  \draw[dots](1.4,5)--(2.7,5);
  \draw[dots](1.6,2.5)--(2.9,2.5);
  \draw[dots](1.4,0)--(2.7,0);
  \draw(3,5)node[above]{$-2$}--(13,0);
  \draw(4,5)node[above]{$-1$}--(14,0);
  \draw(5,5)node[above]{$0$}--(0,2.5)--(5,0);
  \draw(6,5)node[above]{$1$}--(1,2.5)--(6,0);
  \draw[dots](6.3,5)--(7.7,5);
  \draw[dots](6.3,0)--(7.7,0);
  \draw(8,5)node[above]{$-2$}--(3,2.5)--(8,0);
  \draw(9,5)node[above]{$-1$}--(5,3)--(5,2)--(9,0);
  \draw(10,5)node[above]{$0$}--(5.2,2.5)--(10,0);
  \draw(11,5)node[above]{$1$}--(1,0);
  \draw[dots](11.3,5)--(12.7,5);
  \draw[dots](11.3,0)--(12.7,0);
  \draw(13,5)node[above]{$-2$}--(3,0);
  \draw(14,5)node[above]{$-1$}--(4,0);
\end{braid} ,\\\intertext{and}
\sigma_2v&=
\begin{braid}\tikzset{baseline=9mm}
  \draw(14,5)node[above]{$-1$}--(10,3)--(10,2)--(14,0);
  \draw(13,5)node[above]{$-2$}--(3,0);
  \draw[dots](12.6,5)--(11.3,5);
  \draw[dots](12.4,2.5)--(11.1,2.5);
  \draw[dots](12.6,0)--(11.3,0);
  \draw(11,5)node[above]{$1$}--(1,0);
  \draw(10,5)node[above]{$0$}--(0,0);
  \draw(9,5)node[above]{$-1$}--(14,2.5)--(9,0);
  \draw(8,5)node[above]{$-2$}--(13,2.5)--(8,0);
  \draw[dots](7.7,5)--(6.3,5);
  \draw[dots](7.7,0)--(6.3,0);
  \draw(6,5)node[above]{$1$}--(11,2.5)--(6,0);
  \draw(5,5)node[above]{$0$}--(9,3)--(9,2)--(5,0);
  \draw(4,5)node[above]{$-1$}--(8.8,2.5)--(4,0);
  \draw(3,5)node[above]{$-2$}--(13,0);
  \draw[dots](2.7,5)--(1.3,5);
  \draw[dots](2.7,0)--(1.3,0);
  \draw(1,5)node[above]{$1$}--(11,0);
  \draw(0,5)node[above]{$0$}--(10,0);
\end{braid} .
\end{align*}
\end{Lemma}

\begin{proof}
  We prove only the first identity for $\sigma_1$ as the proof of the second one is almost  identical. 
Let $D_1$ be the first diagram on the right hand side of the first equality.   Using more strings for clarity of exposition,
$$
D_1=\begin{braid}\tikzset{scale=1.3,baseline=12mm}
\draw(0,5)node[above]{$0$}--(5.9,2.9)--(5.9,2.1)--(0,0);
\draw(1,5)node[above]{$1$}--(15,0);
\draw(2,5)node[above]{$2$}--(16,0);
\draw[dots](2.6,5)--(3.9,5);
\draw[dots](2.6,2.5)--(3.9,2.5);
\draw[dots](2.6,0)--(3.9,0);
\draw(4,5)node[above]{$-3$}--(18,0);
\draw(5,5)node[above]{$-2$}--(19,0);
\draw(6,5)node[above]{$-1$}--(20,0);
\draw(7,5)node[above]{$0$}--(0,2.5)--(7,0);
\draw(8,5)node[above]{$1$}--(1,2.5)--(8,0);
\draw(9,5)node[above]{$2$}--(2,2.5)--(9,0);
\draw[dots](9.3,5)--(10.7,5);
\draw[dots](9.3,0)--(10.7,0);
\draw(11,5)node[above]{$-3$}--(4,2.5)--(11,0);
\draw(12,5)node[above]{$-2$}--(5,2.5)--(12,0);
\draw(13,5)node[above]{$-1$}--(6,2.5)--(13,0);
\draw[red](14,5)node[above]{$0$}--(7,2.5)--(14,0);
\draw(15,5)node[above]{$1$}--(1,0);
\draw(16,5)node[above]{$2$}--(2,0);
\draw[dots](16.3,5)--(17.7,5);
\draw[dots](16.3,0)--(17.7,0);
\draw(18,5)node[above]{$-3$}--(4,0);
\draw(19,5)node[above]{$-2$}--(5,0);
\draw(20,5)node[above]{$-1$}--(6,0);
\end{braid}.
$$
Pulling the rightmost $0$-string past the\,\, \iicrossing[1]-crossing immediately to its right gives zero because $\psi_{2e+1}v=0$ by Lemma~\ref{LNew}. Here, and in similar
situations below, we will omit such terms which arise when 
applying the braid relations (\ref{braid}). This observation shows that 
\begin{align*}
D_1&=
  \begin{braid}\tikzset{scale=1.3,baseline=12mm}
    \draw(0,5)node[above]{$0$}--(5.9,2.9)--(5.9,2.1)--(0,0);
    \draw[red](1,5)node[above]{$1$}--(6.9,2.9)--(6.9,2.1)--(1,0);
    \draw(2,5)node[above]{$2$}--(16,0);
    \draw[dots](2.6,5)--(3.9,5);
    \draw[dots](2.6,2.5)--(3.9,2.5);
    \draw[dots](2.6,0)--(3.9,0);
    \draw(4,5)node[above]{$-3$}--(18,0);
    \draw(5,5)node[above]{$-2$}--(19,0);
    \draw(6,5)node[above]{$-1$}--(20,0);
    \draw(7,5)node[above]{$0$}--(0,2.5)--(7,0);
    \draw(8,5)node[above]{$1$}--(1,2.5)--(8,0);
    \draw(9,5)node[above]{$2$}--(2,2.5)--(9,0);
    \draw[dots](9.3,5)--(10.7,5);
    \draw[dots](9.3,0)--(10.7,0);
    \draw(11,5)node[above]{$-3$}--(4,2.5)--(11,0);
    \draw(12,5)node[above]{$-2$}--(5,2.5)--(12,0);
    \draw(13,5)node[above]{$-1$}--(6,2.5)--(13,0);
    \draw[red](14,5)node[above]{$0$}--(7,2.5)--(14,0);
    \draw[red](15,5)node[above]{$1$}--(8,2.5)--(15,0);
    \draw(16,5)node[above]{$2$}--(2,0);
    \draw[dots](16.3,5)--(17.7,5);
    \draw[dots](16.3,0)--(17.7,0);
    \draw(18,5)node[above]{$-3$}--(4,0);
    \draw(19,5)node[above]{$-2$}--(5,0);
    \draw(20,5)node[above]{$-1$}--(6,0);
  \end{braid}\\
&=
  \begin{braid}\tikzset{scale=1.3,baseline=12mm}
    \draw(0,5)node[above]{$0$}--(5.9,2.9)--(5.9,2.1)--(0,0);
    \draw(1,5)node[above]{$1$}--(6.9,2.9)--(6.9,2.1)--(1,0);
    \draw[red](2,5)node[above]{$2$}--(7.9,2.9)--(7.9,2.1)--(2,0);
    \draw[dots](2.6,5)--(3.9,5);
    \draw[dots](2.6,2.5)--(3.9,2.5);
    \draw[dots](2.6,0)--(3.9,0);
    \draw(4,5)node[above]{$-3$}--(18,0);
    \draw(5,5)node[above]{$-2$}--(19,0);
    \draw(6,5)node[above]{$-1$}--(20,0);
    \draw(7,5)node[above]{$0$}--(0,2.5)--(7,0);
    \draw(8,5)node[above]{$1$}--(1,2.5)--(8,0);
    \draw(9,5)node[above]{$2$}--(2,2.5)--(9,0);
    \draw[dots](9.3,5)--(10.7,5);
    \draw[dots](9.3,0)--(10.7,0);
    \draw(11,5)node[above]{$-3$}--(4,2.5)--(11,0);
    \draw(12,5)node[above]{$-2$}--(5,2.5)--(12,0);
    \draw(13,5)node[above]{$-1$}--(6,2.5)--(13,0);
    \draw(14,5)node[above]{$0$}--(7,2.5)--(14,0);
    \draw[red](15,5)node[above]{$1$}--(8,2.5)--(15,0);
    \draw[red](16,5)node[above]{$2$}--(9,2.5)--(16,0);
    \draw[dots](16.3,5)--(17.7,5);
    \draw[dots](16.3,0)--(17.7,0);
    \draw(18,5)node[above]{$-3$}--(4,0);
    \draw(19,5)node[above]{$-2$}--(5,0);
    \draw(20,5)node[above]{$-1$}--(6,0);
  \end{braid}
\end{align*}
where for the second equality we pulled the rightmost $1$-string past the
\iicrossing[2]-crossing.  Continuing in this way and pulling the right most $(i-1)$-string past its
neighbouring\,\, \iicrossing-crossing, for $3\le i<e-1$, shows that 
$$D_1=
  \begin{braid}\tikzset{scale=1.3,baseline=12mm}
    \draw(0,5)node[above]{$0$}--(5.9,2.9)--(5.9,2.1)--(0,0);
    \draw(1,5)node[above]{$1$}--(6.9,2.9)--(6.9,2.1)--(1,0);
    \draw(2,5)node[above]{$2$}--(7.9,2.9)--(7.9,2.1)--(2,0);
    \draw[dots](2.6,5)--(3.9,5);
    \draw[dots](2.6,2.5)--(3.9,2.5);
    \draw[dots](2.6,0)--(3.9,0);
    \draw(4,5)node[above]{$-3$}--(9.9,2.9)--(9.9,2.1)--(4,0);
    \draw(5,5)node[above]{$-2$}--(10.9,2.9)--(10.9,2.1)--(5,0);
    \draw(6,5)node[above]{$-1$}--(20,0);
    \draw(7,5)node[above]{$0$}--(0,2.5)--(7,0);
    \draw(8,5)node[above]{$1$}--(1,2.5)--(8,0);
    \draw(9,5)node[above]{$2$}--(2,2.5)--(9,0);
    \draw[dots](9.3,5)--(10.7,5);
    \draw[dots](8.2,2.5)--(9.2,2.5);
    \draw[dots](9.3,0)--(10.7,0);
    \draw(11,5)node[above]{$-3$}--(4,2.5)--(11,0);
    \draw(12,5)node[above]{$-2$}--(5,2.5)--(12,0);
    \draw(13,5)node[above]{$-1$}--(6,2.5)--(13,0);
    \draw(14,5)node[above]{$0$}--(7,2.5)--(14,0);
    \draw[dots](14.3,5)--(15.7,5);
    \draw[dots](14.3,0)--(15.7,0);
    \draw(16,5)node[above]{$-5$}--(9,2.5)--(16,0);
    \draw(17,5)node[above]{$-4$}--(10,2.5)--(17,0);
    \draw(18,5)node[above]{$-3$}--(11,2.5)--(18,0);
    \draw(19,5)node[above]{$-2$}--(12,2.5)--(19,0);
    \draw(20,5)node[above]{$-1$}--(6,0);
  \end{braid}.
$$
Another application of the braid relation (\ref{braid}) yields
$$D_1=
  \begin{braid}\tikzset{scale=1.3,baseline=12mm}
    \draw(0,5)node[above]{$0$}--(5.9,2.9)--(5.9,2.1)--(0,0);
    \draw(1,5)node[above]{$1$}--(6.9,2.9)--(6.9,2.1)--(1,0);
    \draw(2,5)node[above]{$2$}--(7.9,2.9)--(7.9,2.1)--(2,0);
    \draw[dots](2.6,5)--(3.9,5);
    \draw[dots](2.6,2.5)--(3.9,2.5);
    \draw[dots](2.6,0)--(3.9,0);
    \draw(4,5)node[above]{$-3$}--(9.9,2.9)--(9.9,2.1)--(4,0);
    \draw(5,5)node[above]{$-2$}--(10.9,2.9)--(10.9,2.1)--(5,0);
    \draw[red](6,5)node[above]{$-1$}--(11.9,2.9)--(11.9,2.1)--(6,0);
    \draw(7,5)node[above]{$0$}--(0,2.5)--(7,0);
    \draw(8,5)node[above]{$1$}--(1,2.5)--(8,0);
    \draw(9,5)node[above]{$2$}--(2,2.5)--(9,0);
    \draw[dots](9.3,5)--(10.7,5);
    \draw[dots](8.2,2.5)--(9.2,2.5);
    \draw[dots](9.3,0)--(10.7,0);
    \draw(11,5)node[above]{$-3$}--(4,2.5)--(11,0);
    \draw(12,5)node[above]{$-2$}--(5,2.5)--(12,0);
    \draw(13,5)node[above]{$-1$}--(6,2.5)--(13,0);
    \draw(14,5)node[above]{$0$}--(7,2.5)--(14,0);
    \draw[dots](14.3,5)--(15.7,5);
    \draw[dots](14.3,0)--(15.7,0);
    \draw(16,5)node[above]{$-5$}--(9,2.5)--(16,0);
    \draw(17,5)node[above]{$-4$}--(10,2.5)--(17,0);
    \draw(18,5)node[above]{$-3$}--(11,2.5)--(18,0);
    \draw[red](19,5)node[above]{$-2$}--(19,0);
    \draw[red](20,5)node[above]{$-1$}--(20,0);
  \end{braid}.
$$
Applying (\ref{R3Psi}) we can straighten the rightmost $-3,-4,\dots,-1$ strings completely and
then pull the next $e+1$ strings to the right to give 
$$D_1=
  \begin{braid}\tikzset{scale=1.3,baseline=12mm}
    \draw(0,5)node[above]{$0$}--(5.9,2.9)--(5.9,2.1)--(0,0);
    \draw(1,5)node[above]{$1$}--(6.9,2.9)--(6.9,2.1)--(1,0);
    \draw(2,5)node[above]{$2$}--(7.9,2.9)--(7.9,2.1)--(2,0);
    \draw[dots](2.6,5)--(3.9,5);
    \draw[dots](2.6,0)--(3.9,0);
    \draw(4,5)node[above]{$-3$}--(9.9,2.9)--(9.9,2.1)--(4,0);
    \draw(5,5)node[above]{$-2$}--(10.9,2.9)--(10.9,2.1)--(5,0);
    \draw(6,5)node[above]{$-1$}--(11.9,2.9)--(11.9,2.1)--(6,0);
    \draw(7,5)node[above]{$0$}--(4,2.5)--(7,0);
    \draw(8,5)node[above]{$1$}--(5,2.5)--(8,0);
    \draw(9,5)node[above]{$2$}--(6,2.5)--(9,0);
    \draw(10,5)node[above]{$3$}--(7,2.5)--(10,0);
    \draw[dots](10.3,5)--(11.7,5);
    \draw[dots](8.2,2.5)--(9.2,2.5);
    \draw[dots](10.3,0)--(11.7,0);
    \draw(12,5)node[above]{$-2$}--(9,2.5)--(12,0);
    \draw(13,5)node[above]{$-1$}--(10,2.5)--(13,0);
    \draw(14,5)node[above]{$0$}--(11.4,2.5)--(14,0);
    \draw(15,5)node[above]{$1$}--(15,0);
    \draw(16,5)node[above]{$2$}--(16,0);
    \draw[dots](16.3,5)--(17.7,5);
    \draw[dots](16.3,0)--(17.7,0);
    \draw(18,5)node[above]{$-3$}--(18,0);
    \draw(19,5)node[above]{$-2$}--(19,0);
    \draw(20,5)node[above]{$-1$}--(20,0);
  \end{braid}.
$$
Now applying the quadratic relation (\ref{quad}) to the rightmost pair of $(-1,0)$-strings,   using Lemma~\ref{LNew}, and then applying the relation (\ref{psiy}), gives
\begin{align*}
  D_1&=
  \begin{braid}\tikzset{scale=1.3,baseline=12mm}
    \draw(0,5)node[above]{$0$}--(5.9,2.9)--(5.9,2.1)--(0,0);
    \draw(1,5)node[above]{$1$}--(6.9,2.9)--(6.9,2.1)--(1,0);
    \draw(2,5)node[above]{$2$}--(7.9,2.9)--(7.9,2.1)--(2,0);
    \draw[dots](2.6,5)--(3.9,5);
    \draw[dots](2.6,0)--(3.9,0);
    \draw(4,5)node[above]{$-3$}--(9.9,2.9)--(9.9,2.1)--(4,0);
    \draw(5,5)node[above]{$-2$}--(10.9,2.9)--(10.9,2.1)--(5,0);
    \draw[red](6,5)node[above]{$-1$}--(11.9,2.9)--(11.9,2.1)--(6,0);
    \draw(7,5)node[above]{$0$}--(4,2.5)--(7,0);
    \draw(8,5)node[above]{$1$}--(5,2.5)--(8,0);
    \draw(9,5)node[above]{$2$}--(6,2.5)--(9,0);
    \draw(10,5)node[above]{$3$}--(7,2.5)--(10,0);
    \draw[dots](10.3,5)--(11.7,5);
    \draw[dots](8.2,2.5)--(9.2,2.5);
    \draw[dots](10.3,0)--(11.7,0);
    \draw(12,5)node[above]{$-2$}--(9,2.5)--(12,0);
    \draw(13,5)node[above]{$-1$}--(10,2.5)--(13,0);
    \draw[red](14,5)node[above]{$0$}--(14,0);
    \draw(15,5)node[above]{$1$}--(15,0);
    \draw(16,5)node[above]{$2$}--(16,0);
    \draw[dots](16.3,5)--(17.7,5);
    \draw[dots](16.3,0)--(17.7,0);
    \draw(18,5)node[above]{$-3$}--(18,0);
    \draw(19,5)node[above]{$-2$}--(19,0);
    \draw(20,5)node[above]{$-1$}--(20,0);
    \greendot(11.95,2.5);
  \end{braid}\\
  &=
  \begin{braid}\tikzset{scale=1.3,baseline=12mm}
    \draw(0,5)node[above]{$0$}--(5.9,2.9)--(5.9,2.1)--(0,0);
    \draw(1,5)node[above]{$1$}--(6.9,2.9)--(6.9,2.1)--(1,0);
    \draw(2,5)node[above]{$2$}--(7.9,2.9)--(7.9,2.1)--(2,0);
    \draw[dots](2.6,5)--(3.9,5);
    \draw[dots](2.6,0)--(3.9,0);
    \draw(4,5)node[above]{$-3$}--(9.9,2.9)--(9.9,2.1)--(4,0);
    \draw(5,5)node[above]{$-2$}--(10.9,2.9)--(10.9,2.1)--(5,0);
    \draw[red](6,5)node[above]{$-1$}--(11.9,2.9)--(13,0);
    \draw(7,5)node[above]{$0$}--(4,2.5)--(7,0);
    \draw(8,5)node[above]{$1$}--(5,2.5)--(8,0);
    \draw(9,5)node[above]{$2$}--(6,2.5)--(9,0);
    \draw(10,5)node[above]{$3$}--(7,2.5)--(10,0);
    \draw[dots](10.3,5)--(11.7,5);
    \draw[dots](8.2,2.5)--(9.2,2.5);
    \draw[dots](10.3,0)--(11.7,0);
    \draw(12,5)node[above]{$-2$}--(9,2.5)--(12,0);
    \draw[red](13,5)node[above]{$-1$}--(12,2)--(6,0);
    \draw(14,5)node[above]{$0$}--(14,0);
    \draw(15,5)node[above]{$1$}--(15,0);
    \draw(16,5)node[above]{$2$}--(16,0);
    \draw[dots](16.3,5)--(17.7,5);
    \draw[dots](16.3,0)--(17.7,0);
    \draw(18,5)node[above]{$-3$}--(18,0);
    \draw(19,5)node[above]{$-2$}--(19,0);
    \draw(20,5)node[above]{$-1$}--(20,0);
  \end{braid}
\end{align*}
Repeating the same argument another $e-2$ times shows that 
$$D_1=\begin{braid}\tikzset{scale=1.3,baseline=12mm}
    \draw(0,5)node[above]{$0$}--(5.9,2.9)--(5.9,2.1)--(0,0);
    \draw(1,5)node[above]{$1$}--(6.9,2.9)--(6.9,2.1)--(1,0);
    \draw(2,5)node[above]{$2$}--(7.9,2.9)--(9,0);
    \draw[dots](2.6,5)--(4.9,5);
    \draw[dots](2.6,0)--(4.9,0);
    \draw(5,5)node[above]{$-2$}--(10.9,2.9)--(12,0);
    \draw(6,5)node[above]{$-1$}--(11.9,2.9)--(13,0);
    \draw(7,5)node[above]{$0$}--(4,2.5)--(7,0);
    \draw(8,5)node[above]{$1$}--(5,2.5)--(8,0);
    \draw(9,5)node[above]{$2$}--(8,2)--(2,0);
    \draw[dots](9.3,5)--(11.7,5);
    \draw[dots](8.4,2.5)--(11,2.5);
    \draw[dots](9.3,0)--(11.7,0);
    \draw(12,5)node[above]{$-2$}--(11,2)--(5,0);
    \draw(13,5)node[above]{$-1$}--(12,2)--(6,0);
    \draw(14,5)node[above]{$0$}--(14,0);
    \draw(15,5)node[above]{$1$}--(15,0);
    \draw(16,5)node[above]{$2$}--(16,0);
    \draw[dots](16.3,5)--(17.7,5);
    \draw[dots](16.3,0)--(17.7,0);
    \draw(18,5)node[above]{$-3$}--(18,0);
    \draw(19,5)node[above]{$-2$}--(19,0);
    \draw(20,5)node[above]{$-1$}--(20,0);
  \end{braid}
$$
A final application of (\ref{quad}) and (\ref{psiy}) now shows that
$D_1=\sigma_1v$ completing the proof.
\end{proof}

\begin{Lemma}\label{braiding}
Suppose that $k=3$, $i=0$, $v\in T(i,\la)$ and let 
  \begin{align*}
    E_1&=\begin{braid}\tikzset{scale=0.9,baseline=8mm}
      \draw(0,5)node[above]{$0$}--(10,0);
      \draw(1,5)node[above]{$1$}--(11,0);
      \draw[dots](1.5,5)--(2.8,5);
      \draw[dots](1.7,2.5)--(2.9,2.5);
      \draw[dots](1.5,0)--(2.8,0);
      \draw(3,5)node[above]{$-2$}--(13,0);
      \draw(4,5)node[above]{$-1$}--(14,0);
      \draw(5,5)node[above]{$0$}--(0,2.5)--(5,0);
      \draw(6,5)node[above]{$1$}--(1,2.5)--(6,0);
      \draw[dots](6.3,5)--(7.7,5);
      \draw[dots](6.3,2.5)--(7.7,2.5);
      \draw[dots](6.3,0)--(7.7,0);
      \draw(8,5)node[above]{$-2$}--(3,2.5)--(8,0);
      \draw(9,5)node[above]{$-1$}--(4,2.5)--(9,0);
      \draw(10,5)node[above]{$0$}--(0,0);
      \draw(11,5)node[above]{$1$}--(1,0);
      \draw[dots](11.3,5)--(12.7,5);
      \draw[dots](11.3,0)--(12.7,0);
      \draw(13,5)node[above]{$-2$}--(3,0);
      \draw(14,5)node[above]{$-1$}--(4,0);
    \end{braid},
    &E_1'&=\begin{braid}\tikzset{scale=0.9,baseline=8mm}
      \draw(0,5)node[above]{$0$}--(10,0);
      \draw(1,5)node[above]{$1$}--(11,0);
      \draw[dots](1.5,5)--(2.8,5);
      \draw[dots](1.7,2.5)--(2.9,2.5);
      \draw[dots](1.5,0)--(2.8,0);
      \draw(3,5)node[above]{$-2$}--(13,0);
      \draw(4,5)node[above]{$-1$}--(14,0);
      \draw(5,5)node[above]{$0$}--(0,2.5)--(5,0);
      \draw(6,5)node[above]{$1$}--(1,2.5)--(6,0);
      \draw[dots](6.3,5)--(7.7,5);
      \draw[dots](6.3,2.5)--(7.7,2.5);
      \draw[dots](6.3,0)--(7.7,0);
      \draw(8,5)node[above]{$-2$}--(3,2.5)--(8,0);
      \draw(9,5)node[above]{$-1$}--(12,2.5)--(9,0);
      \draw(10,5)node[above]{$0$}--(0,0);
      \draw(11,5)node[above]{$1$}--(1,0);
      \draw[dots](11.3,5)--(12.7,5);
      \draw[dots](11.3,0)--(12.7,0);
      \draw(13,5)node[above]{$-2$}--(3,0);
      \draw(14,5)node[above]{$-1$}--(4,0);
    \end{braid},\\
    E_2&=\begin{braid}\tikzset{scale=0.9,baseline=8mm}
      \draw(14,5)node[above]{$-1$}--(4,0);
      \draw(13,5)node[above]{$-2$}--(3,0);
      \draw[dots](12.5,5)--(11.2,5);
      \draw[dots](12.3,2.5)--(11.1,2.5);
      \draw[dots](12.5,0)--(11.2,0);
      \draw(11,5)node[above]{$1$}--(1,0);
      \draw(10,5)node[above]{$0$}--(0,0);
      \draw(9,5)node[above]{$-1$}--(14,2.5)--(9,0);
      \draw(8,5)node[above]{$-2$}--(13,2.5)--(8,0);
      \draw[dots](6.3,5)--(7.7,5);
      \draw[dots](6.3,2.5)--(7.7,2.5);
      \draw[dots](6.3,0)--(7.7,0);
      \draw(6,5)node[above]{$1$}--(11,2.5)--(6,0);
      \draw(5,5)node[above]{$0$}--(10,2.5)--(5,0);
      \draw(4,5)node[above]{$-1$}--(14,0);
      \draw(3,5)node[above]{$-2$}--(13,0);
      \draw[dots](3.7,5)--(2.3,5);
      \draw[dots](3.7,0)--(2.3,0);
      \draw(1,5)node[above]{$1$}--(11,0);
      \draw(0,5)node[above]{$0$}--(10,0);
    \end{braid},
   &E_2'&=\begin{braid}\tikzset{scale=0.9,baseline=8mm}
      \draw(14,5)node[above]{$-1$}--(4,0);
      \draw(13,5)node[above]{$-2$}--(3,0);
      \draw[dots](12.5,5)--(11.2,5);
      \draw[dots](12.3,2.5)--(11.1,2.5);
      \draw[dots](12.5,0)--(11.2,0);
      \draw(11,5)node[above]{$0$}--(1,0);
      \draw(10,5)node[above]{$1$}--(0,0);
      \draw(9,5)node[above]{$-1$}--(14,2.5)--(9,0);
      \draw(8,5)node[above]{$-2$}--(13,2.5)--(8,0);
      \draw[dots](6.3,5)--(7.7,5);
      \draw[dots](6.3,2.5)--(7.7,2.5);
      \draw[dots](6.3,0)--(7.7,0);
      \draw(6,5)node[above]{$0$}--(11,2.5)--(6,0);
      \draw(5,5)node[above]{$1$}--(2,2.5)--(5,0);
      \draw(4,5)node[above]{$-1$}--(14,0);
      \draw(3,5)node[above]{$-2$}--(13,0);
      \draw[dots](2.7,5)--(1.3,5);
      \draw[dots](2.7,0)--(1.3,0);
      \draw(1,5)node[above]{$1$}--(11,0);
      \draw(0,5)node[above]{$0$}--(10,0);
    \end{braid}
  \end{align*}
  Then $E_1=\sigma_1v+E_1'$ and $E_2=\sigma_2v+E_2'$.
\end{Lemma}

\begin{proof}
Both identities are proved similarly, so we consider only the first one. 
First consider the exceptional case $e=2$. Then we have to show that
\begin{equation}\label{E6910}
 \begin{braid}\tikzset{scale=0.9,baseline=8mm}
      \draw(0,5)node[above]{$0$}--(4,0);
      \draw(1,5)node[above]{$1$}--(5,0);
      \draw(2,5)node[above]{$0$}--(0,2.5)--(2,0);
      \draw(3,5)node[above]{$1$}--(1,2.5)--(3,0);
      \draw(4,5)node[above]{$0$}--(0,0);
      \draw(5,5)node[above]{$1$}--(1,0);
    \end{braid}
    =\sigma_1v+\begin{braid}\tikzset{scale=0.9,baseline=8mm}
      \draw(0,5)node[above]{$0$}--(4,0);
      \draw(1,5)node[above]{$1$}--(5,0);
      \draw(2,5)node[above]{$0$}--(0,2.5)--(2,0);
      \draw(3,5)node[above]{$1$}--(5,2.5)--(3,0);
      \draw(4,5)node[above]{$0$}--(0,0);
      \draw(5,5)node[above]{$1$}--(1,0);
    \end{braid}.
\end{equation}
Applying the braid relation (\ref{braid}), (the first line of) the quadratic relation (\ref{quad}) and Lemma~\ref{LNew}, shows that
$$
   \begin{braid}\tikzset{scale=0.9,baseline=8mm}
      \draw(0,5)node[above]{$0$}--(4,0);
      \draw(1,5)node[above]{$1$}--(5,0);
      \draw(2,5)node[above]{$0$}--(0,2.5)--(2,0);
      \draw(3,5)node[above]{$1$}--(1,2.5)--(3,0);
      \draw(4,5)node[above]{$0$}--(0,0);
      \draw(5,5)node[above]{$1$}--(1,0);
    \end{braid}
    =
   \begin{braid}\tikzset{scale=0.9,baseline=8mm}
      \draw(0,5)node[above]{$0$}--(4,0);
      \draw(1,5)node[above]{$1$}--(5,0);
      \draw(2,5)node[above]{$0$}--(0,2.5)--(2,0);
      \draw[red](3,5)node[above]{$1$}--(1.7,3.375)--(2.5,2.5)--(1.7,1.625)--(3,0);
      \draw(4,5)node[above]{$0$}--(0,0);
      \draw(5,5)node[above]{$1$}--(1,0);
    \end{braid}
    +\begin{braid}\tikzset{scale=0.9,baseline=8mm}
      \draw[red](0,5)node[above]{$0$}--(1.6,3)--(1.6,2)--(0,0);
      \draw(1,5)node[above]{$1$}--(5,0);
      \draw(2,5)node[above]{$0$}--(0,2.5)--(2,0);
      \draw[red](3,5)node[above]{$1$}--(1.8,2.5)--(3,0);
      \draw[red](4,5)node[above]{$0$}--(2,2.5)--(4,0);
      \draw(5,5)node[above]{$1$}--(1,0);
      \greendot(1.6,2.5);
    \end{braid}.   
$$
Applying (\ref{braid}) twice to the first summand and (\ref{psiy}) to the second summand gives
$$
   \begin{braid}\tikzset{scale=0.9,baseline=8mm}
      \draw(0,5)node[above]{$0$}--(4,0);
      \draw(1,5)node[above]{$1$}--(5,0);
      \draw(2,5)node[above]{$0$}--(0,2.5)--(2,0);
      \draw(3,5)node[above]{$1$}--(1,2.5)--(3,0);
      \draw(4,5)node[above]{$0$}--(0,0);
      \draw(5,5)node[above]{$1$}--(1,0);
    \end{braid}
    =
   \begin{braid}\tikzset{scale=0.9,baseline=8mm}
      \draw(0,5)node[above]{$0$}--(4,0);
      \draw(1,5)node[above]{$1$}--(5,0);
      \draw(2,5)node[above]{$0$}--(0,2.5)--(2,0);
      \draw[red](3,5)node[above]{$1$}--(3.8,4)--(2.5,2.5)--(3.8,1.0)--(3,0);
      \draw(4,5)node[above]{$0$}--(0,0);
      \draw(5,5)node[above]{$1$}--(1,0);
    \end{braid}
    +\begin{braid}\tikzset{scale=0.9,baseline=8mm}
      \draw[red](0,5)node[above]{$0$}--(0.8,4)--(0,2.5)--(2,0);
      \draw(1,5)node[above]{$1$}--(5,0);
      \draw[red](2,5)node[above]{$0$}--(1.2,4)--(1.6,2.5)--(0,0);
      \draw(3,5)node[above]{$1$}--(1.8,2.5)--(3,0);
      \draw(4,5)node[above]{$0$}--(2,2.5)--(4,0);
      \draw(5,5)node[above]{$1$}--(1,0);
    \end{braid}.
$$
The relations  (\ref{braid}), (\ref{psiy}) and Lemma~\ref{LNew} show that the first summand above equals the second summand on the right hand side of (\ref{E6910}) and the second  summand above equals $\si_1v$. 

Now consider the case when $e>2$. By~(\ref{braid}), $E_1$ is equal to 
$$\begin{braid}\tikzset{scale=1.0,baseline=9mm}
      \draw(0,5)node[above]{$0$}--(10,0);
      \draw(1,5)node[above]{$1$}--(11,0);
      \draw[dots](1.5,5)--(2.8,5);
      \draw[dots](1.7,2.5)--(2.9,2.5);
      \draw[dots](1.5,0)--(2.8,0);
      \draw(3,5)node[above]{$-2$}--(13,0);
      \draw(4,5)node[above]{$-1$}--(14,0);
      \draw(5,5)node[above]{$0$}--(0,2.5)--(5,0);
      \draw(6,5)node[above]{$1$}--(1,2.5)--(6,0);
      \draw[dots](6.3,5)--(7.7,5);
      \draw[dots](6.3,2.5)--(7.7,2.5);
      \draw[dots](6.3,0)--(7.7,0);
      \draw(8,5)node[above]{$-2$}--(3,2.5)--(8,0);
      \draw[red](9,5)node[above]{$-1$}--(4.7,2.85)--(5.6,2.5)--(4.7,2.15)--(9,0);
      \draw(10,5)node[above]{$0$}--(0,0);
      \draw(11,5)node[above]{$1$}--(1,0);
      \draw[dots](11.3,5)--(12.7,5);
      \draw[dots](11.3,0)--(12.7,0);
      \draw(13,5)node[above]{$-2$}--(3,0);
      \draw(14,5)node[above]{$-1$}--(4,0);
    \end{braid}
  +\begin{braid}\tikzset{scale=1.0,baseline=9mm}
      \draw[red](0,5)node[above]{$0$}--(3.6,3.2)--(3.6,1.8)--(0,0);
      \draw(1,5)node[above]{$1$}--(11,0);
      \draw[dots](1.5,5)--(2.8,5);
      \draw[dots](1.7,2.5)--(2.9,2.5);
      \draw[dots](1.5,0)--(2.8,0);
      \draw(3,5)node[above]{$-2$}--(13,0);
      \draw(4,5)node[above]{$-1$}--(14,0);
      \draw(5,5)node[above]{$0$}--(0,2.5)--(5,0);
      \draw(6,5)node[above]{$1$}--(1,2.5)--(6,0);
      \draw[dots](6.3,5)--(7.7,5);
      \draw[dots](6.3,2.5)--(7.7,2.5);
      \draw[dots](6.3,0)--(7.7,0);
      \draw(8,5)node[above]{$-2$}--(3,2.5)--(8,0);
      \draw[red](9,5)node[above]{$-1$}--(4,2.5)--(9,0);
      \draw[red](10,5)node[above]{$0$}--(5,2.5)--(10,0);
      \draw(11,5)node[above]{$1$}--(1,0);
      \draw[dots](11.3,5)--(12.7,5);
      \draw[dots](11.3,0)--(12.7,0);
      \draw(13,5)node[above]{$-2$}--(3,0);
      \draw(14,5)node[above]{$-1$}--(4,0);
    \end{braid}.
$$
By Lemma~\ref{sigma twist} the second summand is equal to $\sigma_1v$. Using 
the braid relations again, the first summand is equal to
$$\begin{braid}\tikzset{scale=1,baseline=9mm}
      \draw(0,5)node[above]{$0$}--(10,0);
      \draw(1,5)node[above]{$1$}--(11,0);
      \draw[dots](1.5,5)--(2.8,5);
      \draw[dots](1.5,2.5)--(2.7,2.5);
      \draw[dots](1.5,0)--(2.8,0);
      \draw(3,5)node[above]{$-2$}--(13,0);
      \draw(4,5)node[above]{$-1$}--(14,0);
      \draw(5,5)node[above]{$0$}--(0,2.5)--(5,0);
      \draw(6,5)node[above]{$1$}--(1,2.5)--(6,0);
      \draw[dots](6.3,5)--(7.7,5);
      \draw[dots](6.3,0)--(7.7,0);
      \draw(8,5)node[above]{$-2$}--(3,2.5)--(8,0);
      \draw[red](9,5)node[above]{$-1$}--(10.8,4.10)--(7.5,2.45)--(10.8,0.90)--(9,0);
      \draw(10,5)node[above]{$0$}--(0,0);
      \draw(11,5)node[above]{$1$}--(1,0);
      \draw[dots](11.3,5)--(12.7,5);
      \draw[dots](11.3,0)--(12.7,0);
      \draw(13,5)node[above]{$-2$}--(3,0);
      \draw(14,5)node[above]{$-1$}--(4,0);
    \end{braid}=E_1'+
    \begin{braid}\tikzset{scale=1,baseline=9mm}
      \draw(0,5)node[above]{$0$}--(10,0);
      \draw(1,5)node[above]{$1$}--(11,0);
      \draw[dots](1.5,5)--(2.8,5);
      \draw[dots](1.5,2.5)--(2.7,2.5);
      \draw[dots](1.5,0)--(2.8,0);
      \draw[red](3,5)node[above]{$-2$}--(8,2.5)--(3,0);
      \draw(4,5)node[above]{$-1$}--(14,0);
      \draw(5,5)node[above]{$0$}--(0,2.5)--(5,0);
      \draw(6,5)node[above]{$1$}--(1,2.5)--(6,0);
      \draw[dots](6.3,5)--(7.7,5);
      \draw[dots](6.3,0)--(7.7,0);
      \draw(8,5)node[above]{$-2$}--(3,2.5)--(8,0);
      \draw[red](9,5)node[above]{$-1$}--(7.9,2.5)--(9,0);
      \draw(10,5)node[above]{$0$}--(0,0);
      \draw(11,5)node[above]{$1$}--(1,0);
      \draw[dots](11.3,5)--(12.7,5);
      \draw[dots](11.3,0)--(12.7,0);
      \draw[red](13,5)node[above]{$-2$}--(8,2.5)--(13,0);
      \draw(14,5)node[above]{$-1$}--(4,0);
    \end{braid}.
$$
Using the braid relations to pull the rightmost
$-2$-string in the second summand above to the right and observing that the error term of the braid relation equals zero by (\ref{quad}), shows that the second summand equals 
$$D=\begin{braid}\tikzset{scale=1.0,baseline=9mm}
      \draw(0,5)node[above]{$0$}--(10,0);
      \draw(1,5)node[above]{$1$}--(11,0);
      \draw[dots](1.5,5)--(2.8,5);
      \draw[dots](1.5,2.5)--(2.7,2.5);
      \draw[dots](1.5,0)--(2.8,0);
      \draw(3,5)node[above]{$-2$}--(8,2.5)--(3,0);
      \draw(4,5)node[above]{$-1$}--(14,0);
      \draw(5,5)node[above]{$0$}--(0,2.5)--(5,0);
      \draw(6,5)node[above]{$1$}--(1,2.5)--(6,0);
      \draw[dots](6.3,5)--(7.7,5);
      \draw[dots](6.3,0)--(7.7,0);
      \draw(8,5)node[above]{$-2$}--(3,2.5)--(8,0);
      \draw(9,5)node[above]{$-1$}--(8.0,2.5)--(9,0);
      \draw(10,5)node[above]{$0$}--(0,0);
      \draw(11,5)node[above]{$1$}--(1,0);
      \draw[dots](11.3,5)--(12.7,5);
      \draw[dots](11.3,0)--(12.7,0);
      \draw[red](13,5)node[above]{$-2$}--(11,4)--(11,1)--(13,0);
      \draw(14,5)node[above]{$-1$}--(4,0);
    \end{braid}=0,
$$
where the last equality follows because $\psi_{3e-1}v=0$ in view of Lemma~\ref{LNew}.  Therefore,
$E_1=\sigma_1v+E_1'$ as claimed.
\end{proof}

\begin{Lemma}\label{shifting}
Suppose that $i\ne 0,-1$ and $v\in T(i,\la)$. Then
$$\begin{braid}\tikzset{scale=0.9,baseline=8mm}
   \draw(0,5)node[above]{$0$}--(11,0);
   \draw(1,5)node[above]{$1$}--(12,0);
   \draw[dots](1.3,5)--(2.7,5);
   \draw[dots](0.5,2.5)--(1.7,2.5);
   \draw[dots](1.3,0)--(2.7,0);
   \draw(3,5)node[above]{$-1$}--(14,0);
   \draw(4,5)node[above]{$0$}--(0,2.5)--(4,0);
   \draw[dots](4.3,5)--(5.7,5);
   \draw[dots](4.3,0)--(5.7,0);
   \draw(6,5)node[above]{$i{-}1$\,}--(2,2.5)--(6,0);
   \draw(7,5)node[above=0.4mm]{$i$}--(3,2.5)--(7,0);
   \draw(8,5)node[above]{\,$i{+}1$}--(12,2.5)--(8,0);
   \draw[dots](8.3,5)--(9.7,5);
   \draw[dots](12.3,2.5)--(13.7,2.5);
   \draw[dots](8.3,0)--(9.7,0);
   \draw(10,5)node[above]{$-1$}--(14,2.5)--(10,0);
   \draw(11,5)node[above]{$0$}--(0,0);
   \draw(12,5)node[above]{$1$}--(1,0);
   \draw[dots](12.3,5)--(13.7,5);
   \draw[dots](12.3,0)--(13.7,0);
   \draw(14,5)node[above]{$-1$}--(3,0);
   \end{braid}
 =\begin{braid}\tikzset{scale=0.9,baseline=8mm}
   \draw(0,5)node[above]{$0$}--(11,0);
   \draw(1,5)node[above]{$1$}--(12,0);
   \draw[dots](1.3,5)--(2.7,5);
   \draw[dots](0.5,2.5)--(1.7,2.5);
   \draw[dots](1.3,0)--(2.7,0);
   \draw(3,5)node[above]{$-1$}--(14,0);
   \draw(4,5)node[above]{$0$}--(0,2.5)--(4,0);
   \draw[dots](4.3,5)--(5.7,5);
   \draw[dots](4.3,0)--(5.7,0);
   \draw(6,5)node[above]{$i{-}1$\,}--(2,2.5)--(6,0);
   \draw(7,5)node[above=0.4mm]{$i$}--(11,2.5)--(7,0);
   \draw(8,5)node[above]{\,$i{+}1$}--(12,2.5)--(8,0);
   \draw[dots](8.3,5)--(9.7,5);
   \draw[dots](12.3,2.5)--(13.7,2.5);
   \draw[dots](8.3,0)--(9.7,0);
   \draw(10,5)node[above]{$-1$}--(14,2.5)--(10,0);
   \draw(11,5)node[above]{$0$}--(0,0);
   \draw(12,5)node[above]{$1$}--(1,0);
   \draw[dots](12.3,5)--(13.7,5);
   \draw[dots](12.3,0)--(13.7,0);
   \draw(14,5)node[above]{$-1$}--(3,0);
   \end{braid}.
$$
\end{Lemma}

\begin{proof}
Let $D$ be the left hand diagram. Then, using the braid relations, 
\begin{align*}D&=
\begin{braid}\tikzset{scale=1.3,baseline=12mm}
  \draw(0,5)node[above]{$0$}--(14,0);
   \draw[dots](0.4,5)--(1.6,5);
   \draw[dots](0.4,0)--(1.6,0);
  \draw[red](2,5)node[above]{$i{-}1$\,}--(8.4,2.71)--(8.4,2.29)--(2,0);
  \draw(3,5)node[above]{$i$}--(17,0);
  \draw(4,5)node[above]{\,$i{+}1$}--(18,0);
   \draw[dots](4.4,5)--(5.6,5);
   \draw[dots](4.4,0)--(5.6,0);
  \draw(6,5)node[above]{$-1$}--(20,0);
  \draw(7,5)node[above]{$0$}--(0,2.5)--(7,0);
   \draw[dots](7.4,5)--(8.6,5);
   \draw[dots](7.4,0)--(8.6,0);
  \draw(9,5)node[above]{$i{-}1\,$}--(2,2.5)--(9,0);
  \draw[red](10,5)node[above]{$i$}--(12.8,4.00)--(8.6,2.5)--(12.5,1.1)--(10,0);
  \draw(11,5)node[above]{$\,i{+}1$}--(17,2.5)--(11,0);
   \draw[dots](11.4,5)--(12.6,5);
   \draw[dots](11.4,2.5)--(12.6,2.5);
   \draw[dots](11.4,0)--(12.6,0);
  \draw(13,5)node[above]{$-1$}--(20,2.5)--(13,0);
  \draw(14,5)node[above]{$0$}--(0,0);
   \draw[dots](14.4,5)--(15.6,5);
   \draw[dots](14.4,0)--(15.6,0);
  \draw[red](16,5)node[above]{$i{-}1\,$}--(9,2.5)--(16,0);
  \draw(17,5)node[above]{$i$}--(3,0);
  \draw(18,5)node[above]{$\,i{+}1$}--(4,0);
   \draw[dots](18.4,5)--(19.6,5);
   \draw[dots](18.4,0)--(19.6,0);
  \draw(20,5)node[above]{$-1$}--(6,0);
\end{braid}
\\&+\begin{braid}\tikzset{scale=1.3,baseline=12mm}
  \draw(0,5)node[above]{$0$}--(14,0);
   \draw[dots](0.5,5)--(1.7,5);
   \draw[dots](0.6,2.5)--(1.8,2.5);
   \draw[dots](0.5,0)--(1.7,0);
  \draw(2,5)node[above]{$i{-}1$\,}--(16,0);
  \draw(3,5)node[above=0.4mm]{$i$}--(17,0);
  \draw(4,5)node[above]{\,$i{+}1$}--(18,0);
   \draw[dots](4.4,5)--(5.6,5);
   \draw[dots](4.4,0)--(5.6,0);
  \draw(6,5)node[above]{$-1$}--(20,0);
  \draw(7,5)node[above]{$0$}--(0,2.5)--(7,0);
   \draw[dots](7.4,5)--(8.6,5);
   \draw[dots](7.4,0)--(8.6,0);
  \draw(9,5)node[above]{$i{-}1\,$}--(2,2.5)--(9,0);
  \draw[red](10,5)node[above=0.4mm]{$i$}--(12.7,4.04)--(9.0,2.68)--(9.7,2.5)--(9,2.32)--(12.75,0.98)--(10,0);
  \draw(11,5)node[above]{$\,i{+}1$}--(18,2.5)--(11,0);
   \draw[dots](11.4,5)--(12.6,5);
   \draw[dots](11.4,0)--(12.6,0);
  \draw(13,5)node[above]{$-1$}--(20,2.5)--(13,0);
  \draw(14,5)node[above]{$0$}--(0,0);
   \draw[dots](14.4,5)--(15.6,5);
   \draw[dots](14.4,0)--(15.6,0);
  \draw(16,5)node[above]{$i{-}1\,$}--(2,0);
  \draw(17,5)node[above]{$i$}--(3,0);
  \draw(18,5)node[above]{$\,i{+}1$}--(4,0);
   \draw[dots](18.4,5)--(19.6,5);
   \draw[dots](18.3,2.5)--(19.5,2.5);
   \draw[dots](18.4,0)--(19.6,0);
  \draw(20,5)node[above]{$-1$}--(6,0);
\end{braid}.
\end{align*}
Let the first summand of $D$ be $D_1$ and the second one be $D_2$. Then by the braid relations, we have
\begin{align*}
D_1=&\begin{braid}\tikzset{scale=1.3,baseline=12mm}
   \draw(0,5)node[above]{$0$}--(14,0);
   \draw[dots](0.5,5)--(1.7,5);
   \draw[dots](0.6,2.5)--(1.8,2.5);
   \draw[dots](0.5,0)--(1.7,0);
  \draw(2,5)node[above]{$i{-}1$\,}--(8.4,2.71)--(8.4,2.29)--(2,0);
  \draw(3,5)node[above=0.4mm]{$i$}--(17,0);
  \draw(4,5)node[above]{\,$i{+}1$}--(18,0);
   \draw[dots](4.4,5)--(5.6,5);
   \draw[dots](4.4,0)--(5.6,0);
  \draw(6,5)node[above]{$-1$}--(20,0);
  \draw(7,5)node[above]{$0$}--(0,2.5)--(7,0);
   \draw[dots](7.4,5)--(8.6,5);
   \draw[dots](7.4,0)--(8.6,0);
  \draw(9,5)node[above]{$i{-}1\,$}--(2,2.5)--(9,0);
  \draw(10,5)node[above]{$i$}--(12.8,4.00)--(8.6,2.5)--(12.5,1.1)--(10,0);
  \draw(11,5)node[above]{$\,i{+}1$}--(18,2.5)--(11,0);
   \draw[dots](11.4,5)--(12.6,5);
   \draw[dots](11.4,0)--(12.6,0);
  \draw(13,5)node[above]{$-1$}--(20,2.5)--(13,0);
  \draw(14,5)node[above]{$0$}--(0,0);
   \draw[dots](14.4,5)--(15.6,5);
   \draw[dots](14.4,0)--(15.6,0);
  \draw[red](16,5)node[above]{$i{-}1\,$}--(9.7,2.75)--(10.5,2.5)
                                        --(9.7,2.25)--(16,0);
  \draw(17,5)node[above]{$i$}--(3,0);
  \draw(18,5)node[above]{$\,i{+}1$}--(4,0);
   \draw[dots](18.4,5)--(19.6,5);
   \draw[dots](18.3,2.5)--(19.5,2.5);
   \draw[dots](18.4,0)--(19.6,0);
  \draw(20,5)node[above]{$-1$}--(6,0);
\end{braid}\\
  +&\begin{braid}\tikzset{scale=1.3,baseline=12mm}
   \draw(0,5)node[above]{$0$}--(14,0);
   \draw[dots](0.5,5)--(1.7,5);
   \draw[dots](0.6,2.5)--(1.8,2.5);
   \draw[dots](0.5,0)--(1.7,0);
  \draw(2,5)node[above]{$i{-}1$\,}--(8.4,2.71)--(8.4,2.29)--(2,0);
  \draw[red](3,5)node[above=0.4mm]{$i$}--(9.4,2.71)--(9.4,2.29)--(3,0);
  \draw(4,5)node[above]{\,$i{+}1$}--(18,0);
   \draw[dots](4.4,5)--(5.6,5);
   \draw[dots](4.4,0)--(5.6,0);
  \draw(6,5)node[above]{$-1$}--(20,0);
  \draw(7,5)node[above]{$0$}--(0,2.5)--(7,0);
   \draw[dots](7.4,5)--(8.6,5);
   \draw[dots](7.4,0)--(8.6,0);
  \draw(9,5)node[above]{$i{-}1\,$}--(2,2.5)--(9,0);
  \draw(10,5)node[above]{$i$}--(12.8,4.00)--(8.6,2.5)--(12.5,1.1)--(10,0);
  \draw(11,5)node[above]{$\,i{+}1$}--(18,2.5)--(11,0);
   \draw[dots](11.4,5)--(12.6,5);
   \draw[dots](11.4,0)--(12.6,0);
  \draw(13,5)node[above]{$-1$}--(20,2.5)--(13,0);
  \draw(14,5)node[above]{$0$}--(0,0);
   \draw[dots](14.4,5)--(15.6,5);
   \draw[dots](14.4,0)--(15.6,0);
  \draw[red](16,5)node[above]{$i{-}1\,$}--(9.7,2.75)--(9.7,2.25)--(16,0);
  \draw[red](17,5)node[above]{$i$}--(10,2.5)--(17,0);
  \draw(18,5)node[above]{$\,i{+}1$}--(4,0);
   \draw[dots](18.4,5)--(19.6,5);
   \draw[dots](18.3,2.5)--(19.5,2.5);
   \draw[dots](18.4,0)--(19.6,0);
  \draw(20,5)node[above]{$-1$}--(6,0);
\end{braid}.
\end{align*}
The first summand is zero because we can use (\ref{braid}) to
pull the rightmost $(i-1)$-string to the top of the diagram and then use the fact that
$\psi_{2e+i}v=0$ by Lemma~\ref{LNew}. The second summand is zero by
because $\iisquare=0$ by (\ref{quad}). Hence, $D_1=0$. Now consider $D_2$. Using the braid relations to pull the middle $i$-string
in $D_2$ to the right, $D_2$ is equal to the diagram on the right hand side of the formula in the
statement of the lemma plus the following error term
\begin{align*}
\begin{braid}\tikzset{scale=1.3,baseline=12mm}
  \draw(0,5)node[above]{$0$}--(14,0);
   \draw[dots](0.5,5)--(1.7,5);
   \draw[dots](0.6,2.5)--(1.8,2.5);
   \draw[dots](0.5,0)--(1.7,0);
  \draw(2,5)node[above]{$i{-}1$\,}--(16,0);
  \draw(3,5)node[above=0.4mm]{$i$}--(17,0);
  \draw[red](4,5)node[above]{\,$i{+}1$}--(10.6,2.71)--(10.6,2.29)--(4,0);
   \draw[dots](4.4,5)--(5.6,5);
   \draw[dots](4.4,0)--(5.6,0);
  \draw(6,5)node[above]{$-1$}--(20,0);
  \draw(7,5)node[above]{$0$}--(0,2.5)--(7,0);
   \draw[dots](7.4,5)--(8.6,5);
   \draw[dots](7.4,0)--(8.6,0);
  \draw(9,5)node[above]{$i{-}1\,$}--(2,2.5)--(9,0);
  \draw[red](10,5)node[above=0.4mm]{$i$}--(13.8,3.64)--(10.6,2.5)--(13.8,1.36)--(10,0);
  \draw(11,5)node[above]{$\,i{+}1$}--(18,2.5)--(11,0);
   \draw[dots](11.4,5)--(12.6,5);
   \draw[dots](11.4,0)--(12.6,0);
  \draw(13,5)node[above]{$-1$}--(20,2.5)--(13,0);
  \draw(14,5)node[above]{$0$}--(0,0);
   \draw[dots](14.4,5)--(15.6,5);
   \draw[dots](14.4,0)--(15.6,0);
  \draw(16,5)node[above]{$i{-}1\,$}--(2,0);
  \draw(17,5)node[above]{$i$}--(3,0);
  \draw[red](18,5)node[above]{$\,i{+}1$}--(11,2.5)--(18,0);
   \draw[dots](18.4,5)--(19.6,5);
   \draw[dots](18.3,2.5)--(19.5,2.5);
   \draw[dots](18.4,0)--(19.6,0);
  \draw(20,5)node[above]{$-1$}--(6,0);
\end{braid},
\end{align*}
which using the (error term free) braid relations, equals
\begin{align*}
\begin{braid}\tikzset{scale=1.3,baseline=12mm}
  \draw(0,5)node[above]{$0$}--(14,0);
   \draw[dots](0.5,5)--(1.7,5);
   \draw[dots](0.6,2.5)--(1.8,2.5);
   \draw[dots](0.5,0)--(1.7,0);
  \draw(2,5)node[above]{$i{-}1$\,}--(16,0);
  \draw(3,5)node[above=0.4mm]{$i$}--(17,0);
  \draw(4,5)node[above]{\,$i{+}1$}--(10.6,2.71)--(10.6,2.29)--(4,0);
   \draw[dots](4.4,5)--(5.6,5);
   \draw[dots](4.4,0)--(5.6,0);
  \draw(6,5)node[above]{$-1$}--(20,0);
  \draw(7,5)node[above]{$0$}--(0,2.5)--(7,0);
   \draw[dots](7.4,5)--(8.6,5);
   \draw[dots](7.4,0)--(8.6,0);
  \draw(9,5)node[above]{$i{-}1\,$}--(2,2.5)--(9,0);
  \draw[red](10,5)node[above=0.4mm]{$i$}--(7.5,4.10)--(10.8,2.92)--(10.8,2.08)--(7.5,0.9)--(10,0);
  \draw(11,5)node[above]{$\,i{+}1$}--(18,2.5)--(11,0);
   \draw[dots](11.4,5)--(12.6,5);
   \draw[dots](11.4,0)--(12.6,0);
  \draw(13,5)node[above]{$-1$}--(20,2.5)--(13,0);
  \draw(14,5)node[above]{$0$}--(0,0);
   \draw[dots](14.4,5)--(15.6,5);
   \draw[dots](14.4,0)--(15.6,0);
  \draw(16,5)node[above]{$i{-}1\,$}--(2,0);
  \draw(17,5)node[above]{$i$}--(3,0);
  \draw(18,5)node[above]{$\,i{+}1$}--(11,2.5)--(18,0);
   \draw[dots](18.4,5)--(19.6,5);
   \draw[dots](18.3,2.5)--(19.5,2.5);
   \draw[dots](18.4,0)--(19.6,0);
  \draw(20,5)node[above]{$-1$}--(6,0);
\end{braid}.
\end{align*}
By the braid relations again, this equals 
\begin{align*}
\begin{braid}\tikzset{scale=1.3,baseline=12mm}
  \draw(0,5)node[above]{$0$}--(14,0);
   \draw[dots](0.5,5)--(1.7,5);
   \draw[dots](0.6,2.5)--(1.8,2.5);
   \draw[dots](0.5,0)--(1.7,0);
  \draw(2,5)node[above]{$i{-}1$\,}--(16,0);
  \draw(3,5)node[above=0.4mm]{$i$}--(17,0);
  \draw[red](4,5)node[above]{\,$i{+}1$}--(10.3,2.75)--(9.4,2.5)--(10.3,2.25)--(4,0);
   \draw[dots](4.4,5)--(5.6,5);
   \draw[dots](4.4,0)--(5.6,0);
  \draw(6,5)node[above]{$-1$}--(20,0);
  \draw(7,5)node[above]{$0$}--(0,2.5)--(7,0);
   \draw[dots](7.4,5)--(8.6,5);
   \draw[dots](7.4,0)--(8.6,0);
  \draw(9,5)node[above]{$i{-}1\,$}--(2,2.5)--(9,0);
   \draw(10,5)node[above=0.4mm]{$i$}--(7.5,4.10)--(10.8,2.92)--(10.8,2.08)--(7.5,0.9)--(10,0);
   \draw(11,5)node[above]{$\,i{+}1$}--(18,2.5)--(11,0);
   \draw[dots](11.4,5)--(12.6,5);
   \draw[dots](11.4,0)--(12.6,0);
  \draw(13,5)node[above]{$-1$}--(20,2.5)--(13,0);
  \draw(14,5)node[above]{$0$}--(0,0);
   \draw[dots](14.4,5)--(15.6,5);
   \draw[dots](14.4,0)--(15.6,0);
  \draw(16,5)node[above]{$i{-}1\,$}--(2,0);
  \draw(17,5)node[above]{$i$}--(3,0);
  \draw(18,5)node[above]{$\,i{+}1$}--(11,2.5)--(18,0);
   \draw[dots](18.4,5)--(19.6,5);
   \draw[dots](18.3,2.5)--(19.5,2.5);
   \draw[dots](18.4,0)--(19.6,0);
  \draw(20,5)node[above]{$-1$}--(6,0);
\end{braid}\\
+\begin{braid}\tikzset{scale=1.3,baseline=12mm}
  \draw(0,5)node[above]{$0$}--(14,0);
   \draw[dots](0.5,5)--(1.7,5);
   \draw[dots](0.6,2.5)--(1.8,2.5);
   \draw[dots](0.5,0)--(1.7,0);
  \draw(2,5)node[above]{$i{-}1$\,}--(16,0);
  \draw[red](3,5)node[above=0.4mm]{$i$}--(10,2.5)--(3,0);
  \draw[red](4,5)node[above]{\,$i{+}1$}--(10.2,2.78)--(10.2,2.21)--(4,0);
   \draw[dots](4.4,5)--(5.6,5);
   \draw[dots](4.4,0)--(5.6,0);
  \draw(6,5)node[above]{$-1$}--(20,0);
  \draw(7,5)node[above]{$0$}--(0,2.5)--(7,0);
   \draw[dots](7.4,5)--(8.6,5);
   \draw[dots](7.4,0)--(8.6,0);
  \draw(9,5)node[above]{$i{-}1\,$}--(2,2.5)--(9,0);
  \draw(10,5)node[above=0.4mm]{$i$}--(7.5,4.10)--(10.8,2.92)--(10.8,2.08)--(7.5,0.9)--(10,0);
  \draw(11,5)node[above]{$\,i{+}1$}--(18,2.5)--(11,0);
   \draw[dots](11.4,5)--(12.6,5);
   \draw[dots](11.4,0)--(12.6,0);
  \draw(13,5)node[above]{$-1$}--(20,2.5)--(13,0);
  \draw(14,5)node[above]{$0$}--(0,0);
   \draw[dots](14.4,5)--(15.6,5);
   \draw[dots](14.4,0)--(15.6,0);
  \draw(16,5)node[above]{$i{-}1\,$}--(2,0);
  \draw[red](17,5)node[above]{$i$}--(10.5,2.71)--(10.5,2.29)--(17,0);
  \draw(18,5)node[above]{$\,i{+}1$}--(11,2.5)--(18,0);
   \draw[dots](18.4,5)--(19.6,5);
   \draw[dots](18.3,2.5)--(19.5,2.5);
   \draw[dots](18.4,0)--(19.6,0);
  \draw(20,5)node[above]{$-1$}--(6,0);
\end{braid}
\end{align*}
The first summand is zero since $\psi_iv=0$ by Lemma~\ref{LNew}. The second term is  zero because of the quadratic relation
$\iisquare[i]=0$. The proof of the lemma is complete.
\end{proof}

We can now prove Theorem~\ref{sigma braid}. 

\vspace{2mm}
\begin{proof}[Proof of Theorem~\ref{sigma braid}.] 
Writing $\sigma_1\sigma_2\sigma_1v$ in terms of diagrams and using
Lemma~\ref{braiding} we have
\begin{align*}
  \sigma_1\sigma_2\sigma_1v=
   &\begin{braid}\tikzset{scale=1.0,baseline=9mm}
      \draw(0,5)node[above]{$0$}--(10,0);
      \draw(1,5)node[above]{$1$}--(11,0);
      \draw[dots](1.5,5)--(2.8,5);
      \draw[dots](1.7,2.5)--(2.9,2.5);
      \draw[dots](1.5,0)--(2.8,0);
      \draw(3,5)node[above]{$-2$}--(13,0);
      \draw(4,5)node[above]{$-1$}--(14,0);
      \draw(5,5)node[above]{$0$}--(0,2.5)--(5,0);
      \draw(6,5)node[above]{$1$}--(1,2.5)--(6,0);
      \draw[dots](6.3,5)--(7.7,5);
      \draw[dots](6.3,2.5)--(7.7,2.5);
      \draw[dots](6.3,0)--(7.7,0);
      \draw(8,5)node[above]{$-2$}--(3,2.5)--(8,0);
      \draw(9,5)node[above]{$-1$}--(4,2.5)--(9,0);
      \draw(10,5)node[above]{$0$}--(0,0);
      \draw(11,5)node[above]{$1$}--(1,0);
      \draw[dots](11.3,5)--(12.7,5);
      \draw[dots](11.3,0)--(12.7,0);
      \draw(13,5)node[above]{$-2$}--(3,0);
      \draw(14,5)node[above]{$-1$}--(4,0);
    \end{braid}\\
  =&\,\sigma_1+
    \begin{braid}\tikzset{scale=1.0,baseline=9mm}
      \draw(0,5)node[above]{$0$}--(10,0);
      \draw(1,5)node[above]{$1$}--(11,0);
      \draw[dots](1.5,5)--(2.8,5);
      \draw[dots](1.7,2.5)--(2.9,2.5);
      \draw[dots](1.5,0)--(2.8,0);
      \draw(3,5)node[above]{$-2$}--(13,0);
      \draw(4,5)node[above]{$-1$}--(14,0);
      \draw(5,5)node[above]{$0$}--(0,2.5)--(5,0);
      \draw(6,5)node[above]{$1$}--(1,2.5)--(6,0);
      \draw[dots](6.3,5)--(7.7,5);
      \draw[dots](6.3,2.5)--(7.7,2.5);
      \draw[dots](6.3,0)--(7.7,0);
      \draw(8,5)node[above]{$-2$}--(3,2.5)--(8,0);
      \draw(9,5)node[above]{$-1$}--(12,2.5)--(9,0);
      \draw(10,5)node[above]{$0$}--(0,0);
      \draw(11,5)node[above]{$1$}--(1,0);
      \draw[dots](11.3,5)--(12.7,5);
      \draw[dots](11.3,0)--(12.7,0);
      \draw(13,5)node[above]{$-2$}--(3,0);
      \draw(14,5)node[above]{$-1$}--(4,0);
    \end{braid}.
    \end{align*}
    Hence, applying Lemma~\ref{shifting} $(e-1)$ times,
\begin{align*}
\sigma_1\sigma_2\sigma_1v
  =&\,\sigma_1v+
    \begin{braid}\tikzset{scale=1.0,baseline=9mm}
      \draw(0,5)node[above]{$0$}--(10,0);
      \draw(1,5)node[above]{$1$}--(11,0);
      \draw[dots](1.5,5)--(2.8,5);
      \draw[dots](1.7,2.5)--(2.9,2.5);
      \draw[dots](1.5,0)--(2.8,0);
      \draw(3,5)node[above]{$-2$}--(13,0);
      \draw(4,5)node[above]{$-1$}--(14,0);
      \draw(5,5)node[above]{$0$}--(0,2.5)--(5,0);
      \draw(6,5)node[above]{$1$}--(1,2.5)--(6,0);
      \draw[dots](6.3,5)--(7.7,5);
      \draw[dots](6.3,2.5)--(7.7,2.5);
      \draw[dots](6.3,0)--(7.7,0);
      \draw(8,5)node[above]{$-2$}--(11,2.5)--(8,0);
      \draw(9,5)node[above]{$-1$}--(12,2.5)--(9,0);
      \draw(10,5)node[above]{$0$}--(0,0);
      \draw(11,5)node[above]{$1$}--(1,0);
      \draw[dots](11.3,5)--(12.7,5);
      \draw[dots](11.3,0)--(12.7,0);
      \draw(13,5)node[above]{$-2$}--(3,0);
      \draw(14,5)node[above]{$-1$}--(4,0);
    \end{braid}\\
    =\dots=&\,\sigma_1v+
    \begin{braid}\tikzset{scale=1.0,baseline=9mm}
      \draw(14,5)node[above]{$-1$}--(4,0);
      \draw(13,5)node[above]{$-2$}--(3,0);
      \draw[dots](12.5,5)--(11.2,5);
      \draw[dots](12.3,2.5)--(11.1,2.5);
      \draw[dots](12.5,0)--(11.2,0);
      \draw(11,5)node[above]{$0$}--(1,0);
      \draw(10,5)node[above]{$1$}--(0,0);
      \draw(9,5)node[above]{$-1$}--(14,2.5)--(9,0);
      \draw(8,5)node[above]{$-2$}--(13,2.5)--(8,0);
      \draw[dots](6.3,5)--(7.7,5);
      \draw[dots](6.3,2.5)--(7.7,2.5);
      \draw[dots](6.3,0)--(7.7,0);
      \draw(6,5)node[above]{$0$}--(11,2.5)--(6,0);
      \draw(5,5)node[above]{$1$}--(2,2.5)--(5,0);
      \draw(4,5)node[above]{$-1$}--(14,0);
      \draw(3,5)node[above]{$-2$}--(13,0);
      \draw[dots](2.7,5)--(1.3,5);
      \draw[dots](2.7,0)--(1.3,0);
      \draw(1,5)node[above]{$1$}--(11,0);
      \draw(0,5)node[above]{$0$}--(10,0);
    \end{braid}\\
    =&\,\sigma_1v-\sigma_2v+
    \begin{braid}\tikzset{scale=1.0,baseline=9mm}
      \draw(14,5)node[above]{$-1$}--(4,0);
      \draw(13,5)node[above]{$-2$}--(3,0);
      \draw[dots](12.5,5)--(11.2,5);
      \draw[dots](12.3,2.5)--(11.1,2.5);
      \draw[dots](12.5,0)--(11.2,0);
      \draw(11,5)node[above]{$1$}--(1,0);
      \draw(10,5)node[above]{$0$}--(0,0);
      \draw(9,5)node[above]{$-1$}--(14,2.5)--(9,0);
      \draw(8,5)node[above]{$-2$}--(13,2.5)--(8,0);
      \draw[dots](6.3,5)--(7.7,5);
      \draw[dots](6.3,2.5)--(7.7,2.5);
      \draw[dots](6.3,0)--(7.7,0);
      \draw(6,5)node[above]{$1$}--(11,2.5)--(6,0);
      \draw(5,5)node[above]{$0$}--(10,2.5)--(5,0);
      \draw(4,5)node[above]{$-1$}--(14,0);
      \draw(3,5)node[above]{$-2$}--(13,0);
      \draw[dots](3.7,5)--(2.3,5);
      \draw[dots](3.7,0)--(2.3,0);
      \draw(1,5)node[above]{$1$}--(11,0);
      \draw(0,5)node[above]{$0$}--(10,0);
    \end{braid},
\end{align*}
where the last equality follows using the identity $E_2=\sigma_2v+E_2'$ from
Lemma~\ref{braiding}. The diagram in the last equation is equal to
$\sigma_2\sigma_1\sigma_2v$, so this completes the proof of Theorem~\ref{sigma braid}.
\end{proof}

\subsection{\boldmath The elements $\tau$}\label{SSTau}
Let $1\le r<k$. Recall from the beginning of the section that $\sigma_r=\psi_{w_r}e(i,\la)\in R_{k\de}$. Define 
      $$\tau_r=(\si_r+1)e(i,\la)=\sigma_r+e(i,\la)
                    =(\psi_{w_r}+1)e(i,\la).$$ 
Quite remarkably, as we now show, the elements $\tau_1,\dots,\tau_{k-1}$ satisfy the usual
Coxeter relations for the symmetric group $\Si_k$ when they act on the block permutation space
$T(i,\la)$. Let $\Si_\la=\Si_{\la_1}\times\dots\Si_{\la_n}$ be the parabolic subgroup of
$\Si_k$ indexed by $\la$, $\triv{\la}$ the trivial representation of $\Si_\la$, and $\D_\la$ the set of the minimal length left coset representatives of $\Si_\la$ in $\Si_k$.

\begin{Theorem}\label{tau braid}
  Suppose that $1\le r,s<k$ and $v\in T(i,\la)$. Then
  \begin{enumerate}
  \item $\tau_{r}^2v=v$.
     \item If $|r-s|>1$ then 
        $\tau_r\tau_sv=\tau_s\tau_rv$.
     \item If $r< k-1$ then
        $\tau_r\tau_{r+1}\tau_rv= 
                \tau_{r+1}\tau_r\tau_{r+1}v$.
  \end{enumerate}
  Consequently, $\Si_k$ acts on $T(i,\la)$, and the elements $\tau_um(i,\la)$ for $u\in \Si_k$ are well-defined. Finally, 
  $T(i,\la)\cong\ind_{\O\Si_\la}^{\O\Si_k}\triv{\la}$ as $\O\Si_k$-modules, and $T(i,\la)$ has $\O$-basis $\{\tau_um(i,\la)\mid u\in\D_\la\}
$.  
\end{Theorem}

\begin{proof}
Part (i) comes from Corollary~\ref{CSiSq}. Part~(ii) follows directly from the definition of $\sigma_r$. For
  part~(iii), by definition
  \begin{align*}
  \tau_r\tau_{r+1}\tau_rv
     &=(\sigma_r+1)(\sigma_{r+1}+1)
         (\sigma_r+1)v\\
     &=(\sigma_r\sigma_{r+1}\sigma_r
     +\sigma_r\sigma_{r+1}+\sigma_{r+1}\sigma_r+\sigma_r^2
     +2\sigma_r+\sigma_{r+1}+1)v.
  \end{align*} 
Theorem~\ref{sigma braid} and Corollary~\ref{CSiSq} now imply (iii).
So we obtain the action of the symmetric group $\Si_k$ on $T(i,\la)$ with Coxeter generators $s_r$ acting as $\tau_r$ for all $r=1,\dots,k-1$. 

For the final statement of the proposition, consider the parabolic subgroup $\Si_\la\leq \Si_k$ generated by 
$$
\{s_r\mid r\neq \la_1+\dots+\la_a\ \text{for all $a=1,\dots,n$}\}.
$$ 
Note from the definition of $T(i,\la)$ that $T(i,\la)$ is an $\O$-span of all elements of the form
$\tau_{r_1}\dots\tau_{r_a}m(i,\la)$. Moreover, $\O\cdot m(i,\la)\cong\triv{\la}$ is the trivial module
of~$\Si_\la$ because if $s_r\in\Si_\la$ then $\tau_r m(i,\la)=m(i,\la)$ since $\sigma_rm(i,\la)=0$  by
(\ref{EZ}). So we have a surjective homomorphism from $\ind_{\O\Si_\la}^{\O\Si_k}\triv{\la}$ onto $T(i,\la)$,
which sends the natural cyclic generator of $\ind_{\O\Si_\la}^{\O\Si_k}\triv{\la}$ onto $m(i,\la)$. The
injectivity of this map follows from Theorem~\ref{TMBasis}, which describes an $\O$-basis for $M(i,\la)$,
together with the observation that the transition matrix for the change of basis from the products of the
$\tau_r$ to the corresponding products of the $\si_r$ is unitriangular. 
\end{proof}

\section{Homogeneous Garnir relations}\label{SHGR}
In this section we define universal graded (row) Specht modules $S^\mu$ for $R_\al$ by generators and relations, see
Definition~\ref{DSpecht}. This definition will be justified in Theorem~\ref{TMain} when
we show that these universal graded Specht modules are isomorphic to the usual graded Specht modules from~\cite{BKW,HM}.

\subsection{Row Garnir tableaux}\label{SGTab}
The definitions here differ slightly from those given in \cite{BKW} but match those in \cite{MathasB}.
Let $A=(a,b,m)$ be a node of $\mu\in\Par$. Then $A$ is a {\em (row) Garnir node} if $(a+1,b,m)$ is
also a node of $\mu$. 
The {\em (row)  $A$-Garnir belt} $\Belt^A$ is the set of nodes
$$\Belt^A=\set{(a,c,m)\in\mu| b\leq c\leq \mu^{(m)}_{a}}\cup
          \set{(a+1,c,m)\in\mu|1\leq c\leq b}. $$
For example,
if $A=(2,3,2)$ then the $A$-Garnir belt $\Belt^A$ for $\mu=((1),(7,7,4,1))$ is highlighted below:  
$$\begin{array}{l}
\YoungDiagram{1}\\[10pt]
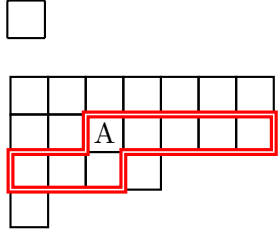
\begin{tikzpicture}[scale=0.5,draw/.append style={thick,black}]
  \newcount\col
  \foreach\Row/\row in {{,,,,,,}/0,{,,A,,,,}/-1,{,,,}/-2,{{}}/-3} {
     \col=1
     \foreach\k in \Row {
        \draw(\the\col,\row)+(-.5,-.5)rectangle++(.5,.5);
        \draw(\the\col,\row)node{\k};
        \global\advance\col by 1
      }
   }
   \draw[red,double,very thick]
     (2.5,-0.5)--++(5,0)--++(0,-1)--++(-4,0)--++(0,-1)--++(-3,0)--++(0,1)--++(2,0)--cycle;
\end{tikzpicture}
\end{array}$$
The (row) {\em $A$-Garnir tableau}  is the $\mu$-tableaux $\G^A$ defined as follows. Let $u=\T^\mu(a,b,m)$ and $v=\T^\mu(a+1,b,m)$. Now insert  the numbers $u,u+1,\dots,v$ into the nodes of the Garnir belt going from left bottom to top right, and the other numbers into the same positions as in $\T^\mu$. 
Continuing the previous example, 
$u=11, v=18$, and $T^\mu$ and the $(2,3,2)$-Garnir tableau are: 
$$
\T^\mu=\begin{array}{l}
\Tableau{{1}}\\[10pt]
\begin{tikzpicture}[scale=0.5,draw/.append style={thick,black}]
  \newcount\col
  \foreach\Row/\row in {{2,3,4,5,6,7,8}/0,{9,10,11,12,13,14,15}/-1,{16,17,18,19}/-2,{20}/-3} {
     \col=1
     \foreach\k in \Row {
        \draw(\the\col,\row)+(-.5,-.5)rectangle++(.5,.5);
        \draw(\the\col,\row)node{\k};
        \global\advance\col by 1
      }
   }
   \draw[red,double,very thick]
     (2.5,-0.5)--++(5,0)--++(0,-1)--++(-4,0)--++(0,-1)--++(-3,0)--++(0,1)--++(2,0)--cycle;
\end{tikzpicture}
\end{array},\quad
\G^A=\begin{array}{l}
\Tableau{{1}}\\[10pt]
\begin{tikzpicture}[scale=0.5,draw/.append style={thick,black}]
  \newcount\col
  \foreach\Row/\row in {{2,3,4,5,6,7,8}/0,{9,10,14,15,16,17,18}/-1,{11,12,13,19}/-2,{20}/-3} {
     \col=1
     \foreach\k in \Row {
        \draw(\the\col,\row)+(-.5,-.5)rectangle++(.5,.5);
        \draw(\the\col,\row)node{\k};
        \global\advance\col by 1
      }
   }
   \draw[red,double,very thick]
     (2.5,-0.5)--++(5,0)--++(0,-1)--++(-4,0)--++(0,-1)--++(-3,0)--++(0,1)--++(2,0)--cycle;
\end{tikzpicture}
\end{array}$$

\begin{Lemma}\label{LAgrees} 
Suppose that $\mu\in\Par$, $A$ is a Garnir node of~$\mu$, and $\Stab\in\St(\mu)$.
If $\Stab\gdom\G^A$ then $\Stab$ agrees with $\T^\mu$ outside the $A$-Garnir belt. 
\end{Lemma}
\begin{proof}
This follows from \cite[Lemma 3.9]{BKW} and  Lemma~\ref{L211108}. 
\end{proof}


The importance of the Garnir tableaux comes from the following:

\begin{Lemma} \label{LAndrew}%
Suppose that $\mu\in\Par_d$ and that $\T$ is a row-strict $\mu$-tableau which is not standard. Then
there exists a Garnir tableau $\G=\G^A$, for $A\in\mu$ a row Garnir node,  and $w\in S_d$ such that 
$\T=w\G$ and $\ell(w^\T)=\ell(w^\G)+\ell(w)$. 
\end{Lemma}
\begin{proof}
If $l=1$ this is \cite[Lemma 3.14]{MathasB}, and the general case follows  easily from the case $l=1$. 
\end{proof}

\subsection{Bricks}\label{SSBricks}
Fix $\mu\in\Par_d$ and a Garnir node $A=(a,b,m)\in\mu$. 
A (row $A$-){\em brick} is a set of $e$ successive nodes in the same row 
$$\{(c,d,m),(c,d+1,m),\dots,(c,d+e-1,m)\}\subseteq \Belt^A$$ such that
$\res(c,d,m)=\res A$. Note that $\Belt^A$ is a disjoint union of the bricks that it contains together
with less than $e$ nodes at the end of row $a$ which are not contained in a brick and less than $e$ nodes
at the beginning of row $a+1$ which are not contained in a brick. 

Let $k=k^A$ be the number of bricks in $\Belt^A$. We label the bricks 
$$B_1^A,B_2^A,\dots, B_k^A$$  
going from left to right along row $a+1$ and then from left to right along row $a$ of $\G^A$ as in the example above. Of course, it might happen that $\Belt^A$ does not contain any bricks (this is always true
if $e=0$), in which case $k=0$.

For example, the following diagram shows the bricks in the $(2,3,2)$-Garnir belt of
$\mu=\big((1),(7,7,4,1)\big)$ when $e=2$: 
$$\begin{array}{l}
\Tableau{{1}}\\[-20pt]
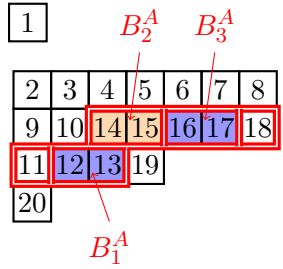
\begin{tikzpicture}[scale=0.5,draw/.append style={thick,black}]
  \fill[orange!30](2.5,-1.5)rectangle(4.5,-0.5);
  \fill[blue!40](4.5,-1.5)rectangle(6.5,-0.5);
  \fill[blue!40](1.5,-2.5)rectangle(3.5,-1.5);
  \newcount\col
  \foreach\Row/\row in {{2,3,4,5,6,7,8}/0,{9,10,14,15,16,17,18}/-1,{11,12,13,19}/-2,{20}/-3} {
     \col=1
     \foreach\k in \Row {
        \draw(\the\col,\row)+(-.5,-.5)rectangle++(.5,.5);
        \draw(\the\col,\row)node{\k};
        \global\advance\col by 1
      }
   }
   \draw[red,double,very thick](2.5,-1.5)--(2.5,-0.5)--(7.5,-0.5)--(7.5,-1.5)--(3.5,-1.5)
                     --(3.5,-2.5)--(0.5,-2.5)--(0.5,-1.5)--(2.5,-1.5);
   \draw[red,double,very thick](4.5,-0.5)--(4.5,-1.5);
   \draw[red,double,very thick](6.5,-0.5)--(6.5,-1.5);
   \draw[red,double,very thick](1.5,-1.5)--(1.5,-2.5);
   \draw[red,double,very thick](2.5,-1.5)--(3.5,-1.5);
   \draw[red,thin,->](3.85,1)node[red,above=0mm]{$B_2^A$}--(3.6,-0.8);
   \draw[red,thin,->](5.85,1)node[red,above=0mm]{$B_3^A$}--(5.6,-0.8);
   \draw[red,thin,->](3.05,-3.5)node[red,below=0mm]{$B_1^A$}--(2.6,-2.2);
\end{tikzpicture}
\end{array}$$
Note that $k=3$, there are two bricks $B_2^A$, $B_3^A$ in row 2 and one brick $B_1^A$ in row 3 of
the second component. 
Further, $(3,1,2)$ and $(2,7,2)$ are the only nodes in the $(2,3,2)$-Garnir belt of~$\G^A$ which are not contained in a brick.

Assume now that $k>0$ and let $n=n^A$ be the smallest entry in $\G^A$ which is contained in a brick in
$\Belt^A$. In the example above, $n=12$. Extending (\ref{w_r}), define
\begin{equation}\label{EWRA}
w_r^A=\prod_{a=n+re-e}^{n+re-1}(a,a+e)\in\Si_d\qquad (1\le r<k).
\end{equation}
Informally, $w_r^A$ swaps the bricks
$B_r^A$ and $B_{r+1}^A$. The elements 
$w_1^A,w_2^A,\dots,w_{k-1}^A$ are the Coxeter generators of the
symmetric group 
$$\Si^A:=\langle w_1^A,w_2^A,\dots,w_{k-1}^A\rangle\cong \Si_k.$$ 
We call $\Si^A$ the (row) {\em brick permutation group}.  By convention,
$\Si^A$ is the trivial group if $k=0$. 

Let $\Gar^A$ be the set of all row-strict $\mu$-tableaux which are obtained from the Garnir tableau $\G^A$ by brick permutations; that is, by acting with the brick permutation group $\Si^A$ on $\G^A$. Note that all of the tableaux in $\Gar^A$, except for $\G^A$, are standard. Moreover, $\G^A$ is the minimal element of $\Gar^A$, with respect to the Bruhat order, and there is a unique maximal tableaux $\T^A$ in $\Gar^A$. Further, by definition, if $\T\in\Gar^A$ then $\bi(\T)=\bi(\G^A)$. Consequently, we let $\bi^A=\bi(\G^A)$ be this common residue sequence. 

Define $f=f^A$ to be the number of $A$-bricks in row $a$ of the Garnir belt $\Belt^A$. Finally, let $\D^A$ be
the set of minimal length left coset representations of $\Si_f\times\Si_{k-f}$ in $\Si^A\cong\Si_k$. Note that by definition $\Si^A$ is a subgroup of $\Si_d$, so $\D^A$ is a subset of $\Si_d$ and, in particular, its elements act on $\mu$-tableaux. Note that
\begin{equation}\label{EGarD}
\Gar^A=\{w\T^A\mid w\in\D^A\}. 
\end{equation}

Continuing the example above, $\T^A$ is the tableau
$$\T^A=\begin{array}{l}
\Tableau{{1}}\\[-20pt]
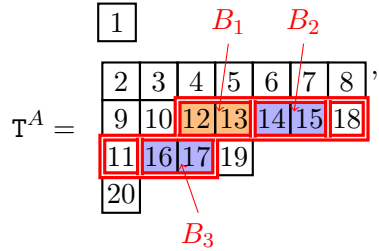
\begin{tikzpicture}[scale=0.5,draw/.append style={thick,black},baseline=1mm]
  \fill[orange!50](2.5,-1.5)rectangle(4.5,-0.5);
  \fill[blue!30](4.5,-1.5)rectangle(6.5,-0.5);
  \fill[blue!30](1.5,-2.5)rectangle(3.5,-1.5);
  \newcount\col
  \foreach\Row/\row in {{2,3,4,5,6,7,8}/0,{9,10,12,13,14,15,18}/-1,{11,16,17,19}/-2,{20}/-3} {
     \col=1
     \foreach\k in \Row {
        \draw(\the\col,\row)+(-.5,-.5)rectangle++(.5,.5);
        \draw(\the\col,\row)node{\k};
        \global\advance\col by 1
      }
   }
   \draw[red,double,very thick](2.5,-1.5)--(2.5,-0.5)--(7.5,-0.5)--(7.5,-1.5)--(3.5,-1.5)
                     --(3.5,-2.5)--(0.5,-2.5)--(0.5,-1.5)--(2.5,-1.5);
   \draw[red,double,very thick](4.5,-0.5)--(4.5,-1.5);
   \draw[red,double,very thick](6.5,-0.5)--(6.5,-1.5);
   \draw[red,double,very thick](1.5,-1.5)--(1.5,-2.5);
   \draw[red,double,very thick](2.5,-1.5)--(3.5,-1.5);
   \draw[red,thin,->](3.85,1)node[red,above=0mm]{$B_1$}--(3.6,-0.8);
   \draw[red,thin,->](5.85,1)node[red,above=0mm]{$B_2$}--(5.6,-0.8);
   \draw[red,thin,->](3.05,-3.5)node[red,below=0mm]{$B_3$}--(2.6,-2.2);
\end{tikzpicture},
\end{array}$$
and 
$\Gar^A=\{\T^A,\Stab:=w_2^A\T^A,\G^A=w_1^Aw_2^A\T^A\}.$
Recall from section~\ref{SSPar} that the
residues of the nodes are determined by a fixed choice of the multicharge~$\kappa$. If we take
$\kappa=(0,0)$ in our example above with $e=2$ then the residues of the nodes in $\mu$ are as follows:
$$\begin{array}{l}
\Tableau{{0}}\\ \Tableau{{0,1,0,1,0,1,0},{1,0,1,0,1,0,1},{0,1,0,1},{1}}
\end{array}$$
Recalling the notation (\ref{EPsiT}) and using Khovanov-Lauda diagrams, we have
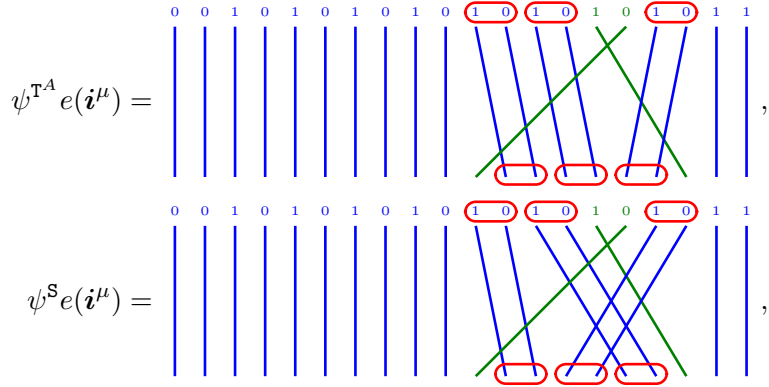
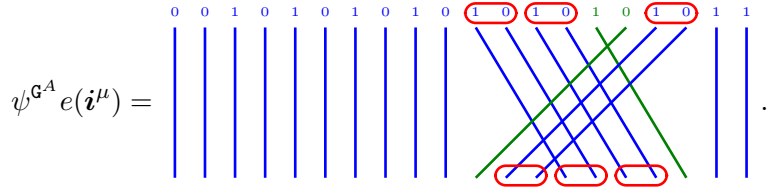
\begin{align*}
\psi^{\T^A}e(\bi^\mu)&=\begin{braid}
     \tikzset{baseline=9mm}
     \foreach \x/\r in {1/0,2/0,3/1,4/0,5/1,6/0,7/1,8/0,9/1,10/0,19/1,20/1} {
       \draw(\x,5)node[above]{$\r$}--(\x,0);
     }
     \draw(11,5)node[above]{$1$}--(12,0);
     \draw(12,5)node[above]{$0$}--(13,0);
     \draw(13,5)node[above]{$1$}--(14,0);
     \draw(14,5)node[above]{$0$}--(15,0);
     \draw[darkgreen](15,5)node[above]{$1$}--(18,0);
     \draw[darkgreen](16,5)node[above]{$0$}--(11,0);
     \draw(17,5)node[above]{$1$}--(16,0);
     \draw(18,5)node[above]{$0$}--(17,0);
     \draw[red](10.65,5.15) rectangle++(1.7,0.65);
     \draw[red](12.65,5.15) rectangle++(1.7,0.65);
     \draw[red](16.65,5.15) rectangle++(1.7,0.65);
     \draw[red](11.65,-0.25) rectangle++(1.7,0.65);
     \draw[red](13.65,-0.25) rectangle++(1.7,0.65);
     \draw[red](15.65,-0.25) rectangle++(1.7,0.65);
\end{braid},\\
\psi^{\Stab}e(\bi^\mu)&=\begin{braid}
     \tikzset{baseline=9mm}
     \foreach \x/\r in {1/0,2/0,3/1,4/0,5/1,6/0,7/1,8/0,9/1,10/0,19/1,20/1} {
       \draw(\x,5)node[above]{$\r$}--(\x,0);
     }
     \draw(11,5)node[above]{$1$}--(12,0);
     \draw(12,5)node[above]{$0$}--(13,0);
     \draw(13,5)node[above]{$1$}--(16,0);
     \draw(14,5)node[above]{$0$}--(17,0);
     \draw[darkgreen](15,5)node[above]{$1$}--(18,0);
     \draw[darkgreen](16,5)node[above]{$0$}--(11,0);
     \draw(17,5)node[above]{$1$}--(14,0);
     \draw(18,5)node[above]{$0$}--(15,0);
     \draw[red](10.65,5.15) rectangle++(1.7,0.65);
     \draw[red](12.65,5.15) rectangle++(1.7,0.65);
     \draw[red](16.65,5.15) rectangle++(1.7,0.65);
     \draw[red](11.65,-0.25) rectangle++(1.7,0.65);
     \draw[red](13.65,-0.25) rectangle++(1.7,0.65);
     \draw[red](15.65,-0.25) rectangle++(1.7,0.65);
\end{braid},\\
\intertext{and}
\psi^{\G^A}e(\bi^\mu)&=\begin{braid}
     \tikzset{baseline=9mm}
     \foreach \x/\r in {1/0,2/0,3/1,4/0,5/1,6/0,7/1,8/0,9/1,10/0,19/1,20/1} {
       \draw(\x,5)node[above]{$\r$}--(\x,0);
     }
     \draw(11,5)node[above]{$1$}--(14,0);
     \draw(12,5)node[above]{$0$}--(15,0);
     \draw(13,5)node[above]{$1$}--(16,0);
     \draw(14,5)node[above]{$0$}--(17,0);
     \draw[darkgreen](15,5)node[above]{$1$}--(18,0);
     \draw[darkgreen](16,5)node[above]{$0$}--(11,0);
     \draw(17,5)node[above]{$1$}--(12,0);
     \draw(18,5)node[above]{$0$}--(13,0);
     \draw[red](10.65,5.15) rectangle++(1.7,0.65);
     \draw[red](12.65,5.15) rectangle++(1.7,0.65);
     \draw[red](16.65,5.15) rectangle++(1.7,0.65);
     \draw[red](11.65,-0.25) rectangle++(1.7,0.65);
     \draw[red](13.65,-0.25) rectangle++(1.7,0.65);
     \draw[red](15.65,-0.25) rectangle++(1.7,0.65);
\end{braid}.
\end{align*}
The circles in these diagrams correspond to the bricks in the Garnir belt.

The degree statement in the following lemma is what will guarantee the homogeneity of our Garnir relations. This result is implicit in the proof of \cite[Proposition~3.16]{BKW}. 

\begin{Lemma}\label{L10910}
  Suppose that $\mu\in\Par_d$ and $A\in\mu$ is a Garnir node. Then
$$
\Gar^A\setminus\{\G^A\}=\{\T\in\St(\mu)\mid \T\gedom\G^A\ \text{and}\ \bi(\T)=\bi^A\}.
$$
Moreover, $\deg(\T)=\deg(\G^A)$ for all $\T\in\Gar^A$. 
\end{Lemma}



\subsection{\boldmath The row permutation modules $M^\mu$}
Let $\mu\in \Par_\al$ be a multipartition with (non-empty) rows $R_1,\dots,R_g$ counted from
top to bottom. If a row $R_a$ has length $N$ and the leftmost node of $R_a$ has residue $i$
we associate the segment $\br(a):=\bs(i,N)$ to $R_a$. Let $\vec{\br}=(\br(1),\dots,\br(g))$,
and, recalling the definitions from section~\ref{SSPerm}, put 
$$
M^\mu=M^\mu(\O):=M(\vec{\br})\,\<\deg \T^\mu\>.
$$ 
Note the degree shift by $\deg \T^\mu$, the significance of which is explained by Theorem~\ref{TMain} below. 
The module $M^\mu$ is generated by the vector
$
m^\mu:=m(\vec{\br}) 
$
of degree $\deg \T^\mu$. Recalling (\ref{EPsiT}), for any $\mu$-tableau $\T$ we define 
$$
m^\T:=\psi^\T m^\mu.
$$
The following is a special case of Theorem~\ref{TMBasis}:

\begin{Theorem} \label{TMMUBasis}
Suppose that $\al\in Q_+$ and $\mu\in\Par_\al$. Then 
$$
\{m^\T\mid \text{$\T$ is a row-strict $\mu$-tableau}\}
$$
is an $\O$-basis of $M^\mu$.  
\end{Theorem}

\subsection{\boldmath Universal row Specht modules $S^\mu$}

Fix a Garnir node $A\in\mu$, and let $\Si^A$ be the corresponding block permutation group with
generators $w_1^A,\dots,w_{k-1}^A$ as defined in section~\ref{SSBricks}. Using the notation of~(\ref{EWRA}), we define 
\begin{equation}\label{ESiTauA}
\si_r^A:=\psi_{w_r^A}e(\bi^A)\quad\text{and}\quad \tau_r^A:=(\si_r^A+1)e(\bi^A),
\end{equation}
cf. section~\ref{SSTau}.
Any element $u\in\Si^A$ can written as a reduced product $u=w_{r_1}^A\dots w_{r_a}^A$ of simple generators $w_1^A,\dots,w_{k-1}^A$ of $\Si^A$. 
In general, the elements $\tau_r^A$ do not have to satisfy Coxeter relations. However, if $u$ is fully commutative then the element
$$
\tau_u^A:=\tau_{r_1}^A\dots \tau_{r_a}^A
$$
is well-defined, since $\tau_r^A$ and $\tau_s^A$ commute for $|r-s|>1$. 
In particular, we have well-defined elements 
$$
\{\tau_u^A\mid u\in \D^A\}.
$$
({\em As  operators on the brick permutation space $T^{\mu,A}$}, defined below, the elements $\tau_r^A$ do satisfy Coxeter relations, see Theorem~\ref{TTauA}(ii).)

Recall from
(\ref{EGarD}) that $\Gar^A$ is the set of row-strict tableaux obtained from the tableau $\T^A$
by acting with the elements of $\D^A$. Note that for any $\Stab\in\Gar^A$,
we can write $w^\Stab=u^\Stab w^{\T^A}$ so that $\ell(w^\Stab)=\ell(u^\Stab)+\ell(w^{\T^A})$
and $u^\Stab\in\D^A$. Moreover, in view of Lemma~\ref{LFullComm}, all elements
$w^\Stab,u^\Stab,$ and $w^{\T^A}$ are fully commutative so the elements
$\psi_{u^\Stab}$, $\psi^{\T^A}$ and
$\psi^\Stab=\psi_{u^\Stab}\psi^{\T^A}$ are all independent of the choice of preferred reduced
decomposition.
Set 
$$
m^A:=m^{\T^A}=\psi^{\T^A}m^\mu\in M^\mu. 
$$

\begin{Definition}
Suppose that $\mu\in\Par_\al$ and $A\in\mu$ is a Garnir node. The {\em (row) Garnir element} is  
$$
g^A:=
\sum_{u\in\D^A}\tau_u^A\psi^{\T^A}\in R_\al.
$$
\end{Definition}

In the module $M^\mu$ we have 
$$g^Am^\mu=\sum_{u\in\D^A}\tau_u^Am^A.$$ 
By Lemma~\ref{L10910} all of the summands on the right hand side have the same degree.  
Finally, if $\D^A=\{1\}$, we have $\G^A=\T^A$ and $g^A=\psi^{\G^A}$. 

\begin{Definition} \label{DSpecht}
Let $\al\in Q_+$, $d=\height(\al)$, and $\mu\in\Par_\al$. Define the {\em universal graded (row) Specht module}\, $S^\mu=S^\mu(\O)$ to be the graded $R_\al$-module generated  by the vector $z^\mu$ of degree $\deg(\T^\mu)$ subject only to the following relations:
\begin{enumerate}
\item[{\rm (i)}] $e(\bj)z^\mu=\de_{\bj,\bi^\mu} z^\mu$ for all $\bj\in I^\al$;
\item[{\rm (ii)}] $y_rz^\mu=0$ for all $r=1,\dots,d$;
\item[{\rm (iii)}] $\psi_rz^\mu=0$ for all $r=1,\dots,d-1$ such that  $r\rightarrow_{\T^\mu}  r+1$;
\item[{\rm (iv)}] ({\em homogeneous Garnir relations})\, $g^A z^\mu=0$, for all (row) Garnir nodes\, $A$ in~$\mu$. 
\end{enumerate}
\end{Definition}

In other words, $S^\mu= (R_\al/J_\al^\mu)\<\deg(\T^\mu)\>$, where $J_\al^\mu$ is the (homogeneous) left ideal of $R_\al$ generated by the elements 
\begin{enumerate}
\item[{\rm (i)}] $e(\bj)-\de_{\bj,\bi^\mu}$ for all $\bj\in I^\al$;
\item[{\rm (ii)}] $y_r$ for all $r=1,\dots,d$;
\item[{\rm (iii)}] $\psi_r$ for all $r=1,\dots,d-1$ such that  $r\rightarrow_{\T^\mu}  r+1$;
\item[{\rm (iv)}] $g^A$\, for all Garnir nodes\, $A\in \mu$. 
\end{enumerate}
In view of (\ref{E:M(s)}), the elements (i)-(iii) generate a left ideal $K^\mu$ such that $R_\al/K^\mu\cong
M^\mu$. So we have a natural surjection $M^\mu\onto S^\mu$, with the kernel $J^\mu$ of this
surjection generated by the Garnir relations $g^Am^\mu=0$. This surjection maps $m^\mu$ to $z^\mu$ and
$J^\mu=J_\al^\mu m^\mu$.

\begin{Remark} 
Our homogeneous Garnir relations are simpler than the ones defined by Young and Garnir in that they have fewer
summands. For example, if $\G^A$ is the only tableau in $\Gar^A$, then the Garnir relation is simply saying
that $\psi^{\G^A}z^\mu=0$. Note that we always have $\Gar^A=\{\G^A\}$ when $e=0$ or $e>d$. 
\end{Remark}

Our main goal is to obtain a basis for the universal Specht modules and to relate them to the
usual Specht modules for cyclotomic Hecke algebras.

\subsection{\boldmath Row brick permutation space $T^{\mu,A}$} \label{SSBrickPerm}
Continuing with the notation of the previous subsection, define the {\em (row) brick permutation space}\, $T^{\mu,A}\subseteq M^\mu$ to be the $\O$-span of all elements of the form 
$
\si_{r_1}^A\dots\si_{r_a}^Am^A,
$
cf. section~\ref{SSBlPermSp}.

\begin{Theorem} \label{TTauA}
Suppose that $\al\in Q_+$, $\mu\in\Par_\al$, $A\in\mu$ is a Garnir node and 
let~$k=k^A$ and~$f=f^A$.  Then:
\begin{enumerate}
\item[{\rm (i)}] $T^{\mu,A}$ is the $\O$-span of all elements of the form 
$
\tau_{r_1}^A\dots\tau_{r_a}^Am^A. 
$
In particular, the elements $\tau_1^A,\dots,\tau_{k-1}^A$ act on $T^{\mu,A}$. 
\item[{\rm (ii)}]  As $\O$-linear operators on  $T^{\mu,A}$, the elements $\tau_1^A,\dots,\tau_{k-1}^A$ satisfy the Coxeter relations for the symmetric group $\Si_k$. Thus, we can consider $T^{\mu,A}$ as an $\O\Si_k$-module. 
\item[{\rm (iii)}] Let $\triv{(f,k-f)}$ be the trivial ${\O\Si_{f,k-f}}$-module. There is an isomorphism of $\O\Si_k$-modules 
$$T^{\mu,A}\cong\ind_{\O\Si_{(f,k-f})}^{\O\Si_k}\triv{(f,k-f)},$$ under which $m^A\in T^{\mu,A}$ corresponds to the natural cyclic generator of the induced module on the right hand side. 
\item[{\rm (iv)}] $\{\tau_u^A m^A\mid u\in\D^A\} 
$
is an $\O$-basis of $T^{\mu,A}$. 
\end{enumerate}
\end{Theorem}
\begin{proof}
Let $i=\res A$, and let $n=n^A$ be, as before, the smallest entry in $\G^A$ which is contained in a brick in
$\Belt^A$. . Set $\bi=\bi^A:=(i_1\dots,i_d)$. For any $\bj=(j_1,\dots,j_{ke})\in I^{k\de}$ define the tuple 
$
\bj^A=(j_1,\dots,j_d)
$ 
where $j_t=i_t$ for $t<n$ and $t\geq n+ek$, and $j_{n+s}=j_{s+1}$ for all $s=0,1,\dots,ke-1$. 
There is a (non-unital) embedding of algebras $\iota^A:R_{k\de}\into R_\al$ such that 
$$
\iota^A: \psi_s\mapsto\psi_{s+n-1}, y_t\mapsto y_{t+n-1},\ e(\bj)\mapsto e(\bj^A)
$$
for all admissible $s,t$ and $\bj$. 
From now on we are going to suppress the notation $\iota^A$ and simply identify $R_{k\de}$ with the subalgebra $\iota^A(R_{k\de})$ inside $R_\al$. 

Consider the $R_{k\de}$-module $R_{k\de}\cdot m^A$ generated by the vector $m^A\in M^\mu$.  We claim that this module is isomorphic to the permutation module $M(i,(f,k-f))$ defined in section~\ref{SBI}. Indeed, it is easy to check that $e(\bj)m^A=\de_{\bj,\bs(i,ke)}m^A$, $y_tm^A=0$, and $\psi_sm^A=0$ unless $s=ef$. This shows that there is an $R_{k\de}$-homomorphism from $M(i,(f,k-f))$ onto $R_{k\de}\cdot m^A$ which maps $m(\vec{\bs}(i,(f,k-f)))$ to $m^A$. An application of Theorem~\ref{TMBasis} now implies that this homomorphism is an isomorphism. Hence, the result follows from Theorem~\ref{tau braid}. 
\end{proof}

\begin{Corollary} \label{C8910}
Suppose that $\mu\in\Par_\al$, $A\in\mu$ is a Garnir node of $\mu$, and $\Stab=u\T^A\in\Gar^A$ for some $u\in\D^A$, cf. (\ref{EGarD}). Then
$$
m^\Stab=\tau_u^A m^A+\sum_{w\in\D^A,\ w\lhd u} c_w\tau_w^Am^A
$$
for some $c_w\in\O$. In particular, $\{m^\T \mid \T\in\Gar^A\}$  
is an $\O$-basis of $T^{\mu,A}$.
\end{Corollary}
\begin{proof}
Let $u=w_{r_1}^A\dots w_{r_a}^A$ be a reduced decomposition in $\Si^A$. Then, using (\ref{ESiTauA}), 
\begin{align*}
m^\Stab&=\psi^\Stab m^\mu=\psi_u\psi^{\T^A}m^\mu=\psi_u m^A=\si_{r_1}^A\dots\si_{r_a}^Am^A
\\
&=(\tau_{r_1}^A-1)\dots(\tau_{r_a}^A-1)m^A,
\end{align*}
which implies the result in view of Theorem~\ref{TTauA}. 
\end{proof}

\subsection{A spanning set for the universal row Specht module}\label{SSspanning}
Let $\al\in Q_+$ with $\height(\al)=d$, and $\mu\in \Par_\al$. Recall that 
$S^\mu\cong M^\mu/J^\mu$ and $z^\mu=m^\mu+J^\mu$. 
Also set 
$$z^A:=m^A+J^\mu\in S^\mu
$$
for any Garnir node $A\in\mu$. 

Recall from (\ref{EPsiT}) that for each $\mu$-tableau $\T$ we have defined the element $\psi^\T\in R_\al$,
which depends on a fixed choice of reduced decomposition of $w^\T\in\Si_d$. Hence, we can associate to $\T$
the homogeneous element 
$$ v^\T:=\psi^\T z^\mu =m^\T+J^\mu  \in S^\mu.  $$

\begin{Lemma} \label{L8910}
Suppose that  $\mu\in\Par_\al$ and $A$ be a Garnir node of $\mu$. Then 
$$
v^{\G^A}=\sum_{\T\in\Gar^A,\ \T\gdom\G^A} c_\T v^\T
$$
for some $c_T\in\O$. 
\end{Lemma}
\begin{proof}
In view of (\ref{EGarD}), for each $\T\in\Gar^A$, there exists a unique $u^\T\in\D^A$ with $\T=u^\T\T^A$. By Corollary~\ref{C8910}, there exist $d_\T\in \O$ such that 
\begin{align*}
v^{\G^A}&=\psi^{\G^A}z^\mu=\Big(\tau_{u^{\G^A}}^A+\sum_{\T\in\Gar^A,\ \T\neq\G^A}d_\T\tau_{u^\T}^A\Big)z^A
\\
&=g^A z^\mu+\sum_{\T\in\Gar^A,\ \T\neq\G^A}(d_\T-1)\tau_{u^\T}^Az^A.
\end{align*}
Since $g^A z^\mu=0$ by Definition~\ref{DSpecht}(iv), the result now follows by (inverting the equations in) Corollary~\ref{C8910}.
\end{proof}

We now make the first step towards describing a standard homogeneous basis of $S^\mu$. In Corollary~\ref{COBasis} below we show that (\ref{ESBasis}) is a basis. 

\begin{Proposition} \label{PSpan}
Let $\mu\in\Par_\al$. The elements of the set
\begin{equation}\label{ESBasis}
\{v^\T\mid \T\in\St(\mu)\}
\end{equation}
span $S^\mu$ over $\O$. Moreover, we have $\deg(v^\T)=\deg(\T)$ for all $\T\in \St(\mu)$. 
\end{Proposition}
\begin{proof}
Note that $v^\T=\psi^\T e(\bi^\mu)z^\mu$. Now, using \cite[Corollary 3.14]{BKW}, we have  
$\deg(\psi^\T e(\bi^\mu))=\deg(\T)-\deg(\T^\mu)$, which implies the second statement of the proposition, as $\deg(z^\mu)=\deg(\T^\mu)$ by definition. 

By Theorem~\ref{TMMUBasis}, it suffices to show that for every row-strict tableau $\T$ of shape $\mu$, the vector $v^\T\in S^\mu$ is an $\O$-linear combination of elements in (\ref{ESBasis}). We prove this by inverse induction on the Bruhat order on the row-strict tableaux $\T$. The induction starts when $\T=\T^\mu$, the unique maximal row-strict tableau. In this case $\T$ is standard so there is nothing to prove. 

For the inductive step, assume that the result has been proved for all row standard tableaux $\U\gdom \T$. If the row strict tableau $\T$ is standard then there is  nothing to prove, so suppose that $\T$ is not standard. Then by Lemma~\ref{LAndrew}, there exists a Garnir tableaux $\G$ of shape $\mu$ and $w\in \Si_d$ such that $\T=w\G$ and $\ell(w^\T)=\ell(w^\G)+\ell(w)$. Using Proposition~\ref{PSubtle} for the second equality, and then Lemma~\ref{L8910} for the last equality, we get
\begin{align*}
v^\T&=\psi^\T e(\bi^\mu)z^\mu=(\psi_w\psi^\G e(\bi^\mu)+xe(\bi^\mu))z^\mu=\psi_wv^G+xz^\mu
\\
&=\sum_{\T\gdom\G} c_\T \psi_w v^\T+xz^\mu,
\end{align*}
where $x$ is a linear combination of elements of the form $\psi_uf(y)e(\bi)$ such that $u< w$ and $f(y)$ is a polynomial in $y_1,\dots,y_d$. The result now follows by induction.
\end{proof}

\section{Cyclotomic Hecke algebras and Specht modules}\label{SNot}
Recall from Definition~\ref{DSpecht} that we have defined  by generators and relations the universal graded (row) Specht modules $S^\mu=S^\mu(\O)$ for the KLR algebra $R_\al$ for all multipartitions $\mu\in\Par_\al$. In this section we connect these universal Specht modules to the usual Specht modules for the affine Hecke algebras $H_\al$ via the isomorphism between the cyclotomic quotients of the KLR algebras and of the affine Hecke algebras constructed in \cite{BK}. This will allow us to obtain a standard homogeneous basis for $S^\mu(\O)$ using \cite{BKW,HM}. 

In this section we will need to  distinguish between the universal graded (row) Specht modules $S^\mu=S^\mu(\O)$ for $R_\al$ and the usual graded (row) Specht modules for $H_\al$, which we will denote $S^\mu_H$. The Specht modules $S^\mu_H$ are defined as cell modules for the cellular algebra $H_\al^\La$. We review their properties below. 

\subsection{Ground field and parameters}
Let $F$ be a field, and $\xi \in F^\times$ be an invertible element. Let $e$ be the  smallest positive integer such that $1+\xi+\dots+\xi^{e-1} = 0,$ setting $e := 0$ if no such integer exists. This $e$ allows us to use the Lie theoretic notation of section~\ref{SSLTN}.  In particular, we have $I=\Z/e\Z$, $\Ga$, $Q_+$, $P_+$, etc. 

For $i\in I$ define the scalar $\nu(i)\in F$ as follows:
\begin{equation}\label{ENu}
  \nu(i):= \begin{cases}i &\text{if }\xi=1,\\\xi^i &\text{if }\xi\neq1.\end{cases}
\end{equation}

\subsection{Cyclotomic Hecke algebra}
Let $H_d=H_d(F,\xi)$ be
the {\em affine Hecke algebra} over the ground field $F$ associated to the symmetric group $\Si_d$ with parameter $\xi$. 
Thus, if $\xi\neq 1$, then 
$H_d$ is the $F$-algebra generated by 
$$T_1,\dots,T_{d-1},X_1^{\pm 1}, \dots, X_d^{\pm 1}$$
subject only to the relations 
\begin{eqnarray}
\label{QECoxeterQuad}
T_r^2 = (\xi -1) T_r +\xi\qquad (1\leq r<d),\\
T_rT_{r+1}T_r=T_{r+1}T_rT_{r+1}\qquad(1\leq r<d-1),\label{QCoxeterClose}\\
T_rT_s=T_sT_r \qquad(1\leq r,s<d,\ |r-s|>1).\label{QCoxeterFar}
\\
X_r^{\pm1} X_s^{\pm1}=X_s^{\pm1} X_r^{\pm1}\qquad(1\leq r,s\leq d), \\ 
X_rX_r^{-1}=1\qquad(1\leq r\leq d),\\
T_r X_r T_r = \xi X_{r+1}\qquad(1\leq r< d),\\
T_r X_s = X_{s}T_r \qquad(1\leq r<d,1\leq s\leq d,s \neq r,r+1).
\end{eqnarray}
If $\xi = 1$, then 
$H_d$ is the $F$-algebra generated by 
$$T_1,\dots,T_{d-1}, X_1,\dots,X_d$$ subject only to the relations (\ref{QECoxeterQuad})--(\ref{QCoxeterFar}) and the relations:
\begin{eqnarray}
X_rX_s=X_sX_r\qquad(1\leq r,s\leq d),
\\
\label{EDAHA}
T_r X_{r+1} = X_r T_r + 1\qquad(1\leq r< d),\\
\label{EDAHA2}
 T_r X_s = X_s T_r \qquad(1\leq r<d,1\leq s\leq d,s \neq r,r+1).
\end{eqnarray}

Recall that in (\ref{EKappa}) and (\ref{ELa}) we have fixed a level $l$, a tuple $\kappa=(k_1,\dots,k_l)$, and the corresponding weight $\La=\La_{k_1}+\dots+\La_{k_l}\in P_+$. 
The {\em cyclotomic Hecke algebra} $H_d^\La=H_d^\La(F,\xi)$ 
is the quotient 
\begin{equation}\label{ECHA}
H_d^\La := H_d \Big/ \big\langle \,\textstyle\prod_{i\in I}(X_1-\nu(i))^{(\La,\al_i)}\,\big\rangle= H_d\Big/\langle \textstyle\prod_{m=1}^l (X_1-\nu(k_m))\big\rangle.
\end{equation}



\subsection{Weight spaces and idempotents}
Let $\bi=(i_1,\dots,i_d)\in I^d$, and let~$M$ be a finite dimensional $H_d^\La$-module. Define the
{\em $\bi$-weight space} of $M$ as follows:
$$
M_\bi=\{v\in M\mid (X_r-\nu(i_r))^Nv=0\ \text{for $N\gg0$ and $r=1,\dots,d$}\}.
$$
It is known (see e.g. \cite[Lemma 4.7]{Gr} and \cite[Lemma 7.1.2]{Kbook}) that 
all eigenvalues of $X_1,\dots,X_d$ in $M$ are of the form $\nu(i)$, for $i\in I$, and so we have a {\em weight space decomposition}:
$$
M=\bigoplus_{\bi\in I^d} M_\bi.
$$
Using the weight space decomposition of the left regular $H_d^\La$-module, one gets a system of orthogonal idempotents
\begin{equation}\label{EIdempotents}
\{e(\bi)\mid \bi\in I^d\}
\end{equation}
in $H_d^\La$, all but finitely many of which are zero, such that
$
\sum_{\bi\in I^d}e(\bi)=1,
$
and 
$$
e(\bi)M=M_\bi\qquad(\bi\in I^d)
$$
for any finite dimensional $H_d^\La$-module $M$.

 If $\al\in Q_+$ is of height $d$, define 
$
e_{\alpha} := \sum_{\bi \in I^\alpha} e(\bi) \in H^\La_d.
$ 
By \cite{LM} and \cite[Theorem~1]{cyclo},
$e_{\alpha}$ is either zero or it is a primitive central idempotent
in $H^\La_d$.
Hence the algebra
\begin{equation}\label{fe}
H^\La_\alpha := e_{\alpha} H^\La_d
\end{equation}
is either zero or it is a single {\em block}
of the algebra $H^\La_d$. 

\subsection{The Isomorphism Theorem}
Define elements of $H^\La_\al$ as follows: 
\begin{equation}\label{QEPolKL}
  y_r:=\begin{cases}
\sum_{\bi\in I^\al}(1-\nu(i_r)^{-1} X_r)e(\bi),&\text{if $\xi \neq 1$},\\[3mm]
\sum_{\bi\in I^\al}(X_r-\nu(i_r)) e(\bi),&\text{if $\xi=1$},
\end{cases}
\end{equation}
for $1\leq r\leq d$.
Next, if $1\le r<d$ and $\bi\in I^d$ we define
\begin{equation}\label{QQCoxKL}
\psi_r:=
\textstyle \sum_{\bi\in I^\al}(T_r+P_r(\bi))Q_r(\bi)^{-1}e(\bi). 
\end{equation}
where $P_r(\bi)$ and $Q_r(\bi)^{-1}$ are certain polynomials in $F[y_r,y_{r+1}]$ which are explicitly defined
in~\cite{BK}. This gives us the following elements of $H_\al^\La$:
\begin{equation}\label{EKLElts}
\{e(\bi)\mid \bi\in I^\al\}\cup\{y_1,\dots,y_d\}\cup\{\psi_1,\dots,\psi_{d-1}\}.
\end{equation}
Note that these elements have the same names as the generators of the KLR algebras in (\ref{EKLGens}). This is not a coincidence in view of the following Isomorphism Theorem:

\begin{Theorem} \label{TBK} {\rm \cite[Theorem 1.1]{BK}} 
Suppose that $\al\in Q_+$ has height $d$ and $\La\in P_+$. Then $H_\al^\La$ is generated by the elements (\ref{EKLElts}) subject only to the relations (\ref{R1})--(\ref{ERCyc}). In other words, $H_\al^\La(F,\xi)\cong R_\al^\La(F)$. 
\end{Theorem}

In what follows we identify $H_\al^\La(F,\xi)$ and $R_\al^\La(F)$. In particular, $H_\al^\La(F,\xi)$ is now $\Z$-graded.

\subsection{\boldmath Graded Specht modules $S^\mu_H$ for Hecke algebras}\label{SSHSpecht}
Let $\al\in Q_+$ be of height $d$ and fix a multipartition $\mu\in\Par_\al$. The {\em graded (row) Specht module} $S^\mu_H=S^\mu_H(F)$ for $H_\al^\La$ is defined in \cite{BKW}. These graded Specht modules turn out to be the {\em cell modules}\, for $H_\al^\La$ considered as a graded cellular algebra as in \cite{HM}. We will not need the exact definition, only the following key properties of these modules.  Recall the notation of section~\ref{SS:tableaux}. 

\begin{Lemma}\label{LZ1} {\rm \cite[Proposition 3.7]{JM}} 
Let $\mu\in \Par_\al$. There is a homogeneous generator  $z_H^\mu$ of $S^\mu_H$ with $\deg(z_H^\mu)=\deg(\T^\mu)$, $z_H^\mu\in e(\bi^\mu)S^\mu_H$, and $y_r z_H^\mu=0$, for all $r=1,\dots,d$.
\end{Lemma}

Let $\T$ be a $\mu$-tableau. Recall from (\ref{EPsiT}) that we have defined the element 
$\psi^\T=\psi_{w^\T}$ in $R_\al^\La(F)=H_\al^\La(F)$. Set 
$$
v_H^\T:=\psi^\T z_H^\mu\in S^\mu_H.
$$
Just like $\psi_{w}$ the vector $v_H^\T$ will, in general, depend upon on the choice of preferred reduced decomposition of $w^\T\in\Si_d$. Note that $z_H^\mu=v_H^{\T^\mu}$. 

\begin{Lemma}\label{LPsiStr} {\rm \cite[Lemma 4.9]{BKW}} 
Let $\mu\in\Par_d$ and $\T \in\St(\mu)$. 
If $r\downarrow_\T  r+1$ or $r\rightarrow_\T  r+1$ 
then
$$
\psi_rv_H^\T =\sum_{\Stab\in \St(\mu),\,\Stab\gdom \T,\,\bi(\Stab)=\bi(s_r\T)}a_\Stab v_H^\Stab.
$$
for some $a_\Stab\in F$. 
In particular, $\psi_rz_H^\mu=0$ whenever $r\rightarrow_{\T^\mu}  r+1$. 
\end{Lemma}

\begin{Theorem}\label{PVZ} {\rm \cite{BKW}} 
Suppose that $\mu\in\Par_\al$. Then 
\begin{enumerate}
\item[{\rm (i)}] For any $\mu$-tableau $\T$ we have $v_H^\T\in e(\bi(\T))S^\mu_H$. 
\item[{\rm (ii)}] If $\mu\in\St(\mu)$, then $\deg(v_H^\T)=\deg(\T)$. 
\item[{\rm (iii)}] $\{v_H^\T\mid \T\in \St(\mu)\}$ is a basis of  $S_H^\mu(F)$. Moreover, for 
any $\mu$-tableau~$\T$, we have
$$
v_H^\T=\sum_{\substack{\Stab\in \St(\mu)\\\Stab\gedom \T,\bi(\Stab)=\bi(\T)}}b_\Stab v_H^\Stab
$$
for some constants $b_\Stab\in F$. 
\end{enumerate}
\end{Theorem}

The following corollary should be compared with Lemma~\ref{L8910}. 

\begin{Corollary} \label{C10910}
Suppose that $\mu\in\Par_d$ and that $\G=\G^A$, where $A\in \mu$ is a Garnir node. Then 
$$
v_H^{\G}=\sum_{\substack{\T\in\Gar^A,\ \T\gdom \G}}c_\T v_H^\T\qquad
$$
for some $c_\T\in F$. 
\end{Corollary}
\begin{proof}
This comes from Theorem~\ref{PVZ} and Lemma~\ref{L10910}.
\end{proof}


\subsection{Connecting the universal row Specht modules with the cell modules}\label{SSconnection}
Since we have identified the algebras $H_\al^\La(F)$ and $R_\al^\La(F)$ we may consider the Specht modules $S^\mu_H(F)$ as
an $R_\al^\La(F)$-module. Inflating from $R_\al^\La(F)$ to the affine KLR algebra $R_\al(F)$, we can now
consider $S^\mu_H(F)$ as an $R_\al(F)$-module. The following theorem shows  that, as a graded $R_\alpha(F)$-module, $S^\mu_H(F)$ is
isomorphic to the universal row Specht module $S^\mu(F)$ from Definition~\ref{DSpecht}.

\begin{Theorem} \label{TMain}
Let $\mu\in\Par_\al$. Then the linear map, which sends the basis elements $v_H^\T\in S^\mu_H(F)$ to $v^\T\in S^\mu(F)$ for all $\T\in\St(\mu)$, is a homogeneous isomorphism $S^\mu_H(F)\iso S^\mu(F)$ of graded $R_\al(F)$-modules. 
\end{Theorem}
\begin{proof}
In this proof all modules and algebras are vector spaces over $F$, so we will suppress $F$ from our notation. We will construct the isomorphism in the other direction: $S^\mu\iso S^\mu_H$. 
By Lemmas~\ref{LZ1}, \ref{LPsiStr} and the defining relations for $M^\mu$, cf. (\ref{E:M(s)}), there exists a surjective degree zero homogeneous homomorphism $\pi: M^\mu\surj S^\mu_H$ of graded $R_\al$-modules which maps $m^\T$ to $v_H^\T$ for any row-strict $\mu$-tableau~$\T$.  By Theorem~\ref{PVZ} and Proposition~\ref{PSpan}, it now suffices to check that the homogeneous Garnir relations $g^Az_H^\mu=0$ hold in $S^\mu_H$, for all Garnir nodes $A\in\mu$. 

Fix a Garnir node $A$.
Let $k=k^A$, $f=f^A$ and $\Si^A\cong \Si_k$ be the brick permutation group defined in section~\ref{SSBricks}. 
By Corollary~\ref{C8910}, $\{m^\T\mid \T\in\Gar^A\}$ is an $F$-basis of $T^{\mu,A}$. Note that $\G^A$ is the only non-standard tableaux in $\Gar^A$. As $\pi(T^{\mu,A})$ is spanned by the vectors $\{v_H^\T=\pi(m^\T)\mid\T\in\Gar^A\}$, Corollary~\ref{C10910} shows that $\{v_H^\T\mid\T\in\Gar^A\setminus\{\G^A\}\}$ is a basis of $\pi(T^{\mu,A})$. So $\dim \pi(T^{\mu,A})=\dim T^{\mu,A}-1$. 

Recall from Theorem~\ref{TTauA} that the group $\Si^A$ acts on the brick permutation subspace $T^{\mu,A}$ with
its simple reflections acting as $\tau_1^A,\dots,\tau_{k-1}^A$. Moreover, with respect to this action,
$T^{\mu,A}\cong \ind_{F\Si_{(f,k-f)}}^{F\Si_k}F_{(f,k-f)}$.  Since the elements of $\Si^A$ act on $T^{\mu,A}$
as specific elements of $R_\al$, and $\pi$ is an $R_\al$-homomorphism,~$\pi$ induces an $F\Si^A$-homomorphism
$T^{\mu,A}\longrightarrow \pi(T^{\mu,A})$. By the dimension observations in the previous paragraph, the
kernel of this map is a one dimensional $F\Si_k$ submodule of
$T^{\mu,A}=\ind_{F\Si_{(f,k-f)}}^{F\Si_k}F_{(f,k-f)}$. Therefore, unless $k=2$, $f=1$, and $\cha F\neq 2$,
this kernel is the unique trivial submodule of $\ind_{F\Si_{(f,k-f)}}^{F\Si_k}F_{(f,k-f)}$. Hence, in this
case, $\ker\pi$ is spanned by $\sum_{u\in \D^A}\tau_u^Am^A=g^Am^\mu$. Hence $g^Az_H^\mu=\pi(g^Am^\mu)=0$, so that
the Garnir relation holds in the Specht module $S^\mu_H$, as desired. 

It remains to consider the exceptional case $k=2,f=1,\cha F\neq 2$. In this case we claim that
$(\tau_1^A+1)z^A=0$, for this we need to rule out the possibility that $(\tau_1^A-1)z^A=0$.
Since $\tau_1^Az^A=(\si_1^A+1)z^A$, we just need to prove that $\si_1^Az^A\neq 0$. Let
$A=(a,b,m)$, and $r$ be the entry which occupies the node $(a+1,b,m)$ in $\G^A$. But by
Lemma~\ref{LPsiSi}, we have $$\psi_r \si_1^Az^A=-2\psi_rz^A=-2v^{s_r\T^A}\neq 0,$$ since the
tableau $s_r\T^A$ is standard and $\cha F\ne2$. 
\end{proof}

We can now improve on Proposition~\ref{PSpan}: 

\begin{Corollary} \label{COBasis}
Let $\mu\in\Par_\al$. Then the universal row Specht module  $S^\mu(\O)$ for $R_\al(\O)$  has $\O$-basis
\begin{equation}\label{ESBasis2}
\{v^\T\mid \T\in\St(\mu)\}.
\end{equation}
\end{Corollary}
\begin{proof}
As $S^\mu(\O)\cong S^\mu(\Z)\otimes_\Z\O$, we may assume that $\O=\Z$. By Proposition~\ref{PSpan}, the elements (\ref{ESBasis2}) span $S^\mu(\Z)$. Suppose that we have a relation $ \sum_{\T\in\St(\mu)} c_\T v^\T=0$ with $c_\T\in\Z$.  Extending scalars to $\C$, we get the relation $\sum_{\T\in\St(\mu)} c_\T v^\T=0$ in $S^\mu(\C)$. Pick a parameter $\xi\in \C$ which is a primitive $e$th root of unity in $\C$ if $e>0$ and not a root of unity if $e=0$. Then by Theorem~\ref{TMain}, we get the relation $\sum_{\T\in\St(\mu)} c_\T v^\T_H=0$ in $S^\mu_H(\C)$, which is the usual Specht module for $H_\al^\La(\C,\xi)$. By Theorem~\ref{PVZ}(iii), $c_\T=0$ for all $\T\in\St(\mu)$.  
\end{proof}

\begin{Corollary} 
Let $\mu\in\Par_\al$. The universal row Specht module $S^\mu(\O)$ factors through the natural surjection $R_\al(\O)\onto  R_\al^\La(\O)$ so that $S^\mu(\O)$ is naturally a graded  $R_\al^\La(\O)$-module. 
\end{Corollary}
\begin{proof}
In view of (\ref{ERCyc}), we just need to prove that 
$$
y_1^{(\La,\al_{i_1})}e(\bi)S^\mu(\O)=0 \qquad(\bi=(i_1,\dots,i_d)\in I^\al).
$$
We may assume that $\O=\Z$. Next, since $S^\mu(\Z)\into S^\mu(\Z)\otimes\C=S^\mu(\C)$, we may now assume that
$\O=\C$.  Choose $\xi\in \C$ as in the proof of Corollary~\ref{COBasis}. Then by Theorem~\ref{TMain}, we have
$S^\mu(\C)=S^\mu_H(\C)$, which is the usual Specht module for $H_\al^\La(\C,\xi)$. Hence, $S^\mu(\C)$ is a
$R^\Lambda_\alpha(\C)$ module since $H_\alpha^\Lambda(\C,\xi)\cong R_\alpha^\lambda(\C)$. Hence, the action of
$R_\alpha^\alpha(\C)$ satisfies the cyclotomic relation (\ref{ERCyc}), implying that $S^\mu(\O)$ is an
$R_\al^\La(\O)$-module as we wanted to show.
\end{proof}

Now the following is clear:

\begin{Corollary} 
Let $\mu\in\Par_\al$. 
\begin{enumerate}
\item[{\rm (i)}] As a graded $R_\al^\La(\O)$-module, the universal row Specht module $S^\mu(\O)$ is generated by the homogeneous element $z^\mu$ of degree $\deg(\T^\mu)$ subject only to the relations (i)--(iv) from Definition~\ref{DSpecht}. 
\item[{\rm (ii)}] As a graded  $H_\al^\La$-module, the row Specht module $S^\mu_H$ is generated  by the homogeneous element $z^\mu$ of degree $\deg(\T^\mu)$ subject only to the relations (i)--(iv) from Definition~\ref{DSpecht}. 
\end{enumerate}
\end{Corollary}

\section{Column Specht modules}\label{SDualSpecht}
Having a presentation for a module does not automatically imply a presentation for the dual module. In this section, we define a column version $S_\mu$ of the universal graded Specht module corresponding to a multipartition $\mu$. Then in Theorem~\ref{Tduality} we show that the universal column Specht module $S_\mu$ is isomorphic to (a degree shift of) the homogeneous dual $(S^\mu)^\circledast$ of the universal row Specht module $S^\mu$.


In the section we again work over an arbitrary commutative unital ground ring $\O$,  unless otherwise stated. We fix
$
\al\in Q_+
$, $\mu\in\Par_\al$, and set $d:=\height(\al)$. 

\subsection{Column block intertwiners}\label{Signed block intertwiners}
In this section we assume that $e>0$.  Recall from (\ref{E:nullRoot}) that $\delta$ is the null root and observe
that $\delta'=\delta$ in the notation of section~\ref{SSSign}. Therefore, 
$\sgn$ is an automorphism of $R_{k\de}$, see~(\ref{epsilon}).

Fix $i\in I$ and a composition $\la=(\la_1,\dots,\la_n)$ of $k$. Define
$$\vec{\bs}(i,-\la):=(\bs(i,-e\la_1),\dots,\bs(i,-e\la_n)).$$ We consider the corresponding permutation module
$M(i,-\la):=M(\vec{\bs}(i,-\la))$ for $R_{k\de}$ as in section~\ref{SSPerm}. Let
$\bj=(j_1,\dots,j_{ke}):=\bj(\vec{\bs}(i,-\la))$ as defined in (\ref{EBJ}). We have 
$\bj=\bs(i,-ke). 
$
Let 
$
e(i,-\la):=e(\bj)
$ 
and $m(i,-\la)=m(\vec{\bs}(i,-\la))\in M(i,-\la)$ as in (\ref{EZ}). 

Recall from section~\ref{SSSign} that if $M$ is an $R_{k\delta}$-module then $M^\sgn$
is the $R_{k\delta}$-module obtained from $M$ by twisting with the sign automorphism~$\sgn$.


\begin{Lemma}\label{L:permIso}
We have 
  \begin{enumerate}
     \item $\sgn(e(i,-\la))=e(-i,\la)$.
     \item There is an isomorphism $M(i,-\la) \cong M(-i,\la)^\sgn$ of graded $R_{k\delta}$-modules, under which $m(i,-\la)$ corresponds to $m(-i,\la)$.
  \end{enumerate}
\end{Lemma}

\begin{proof}
(i) is clear from (\ref{epsilon}),  and (ii) is clear from~(\ref{E:M(s)}).
\end{proof}

From section~\ref{SBI}, we have the elements $w_r\in\Si_{ke}$, $\sigma_r=\psi_{w_r}e(-i,\la)$ and
$\tau_r=(\si_r+1)e(-i,\la)$, for $1\le r<k$. Set 
$$\si^r:=(-1)^{e}\psi_{w_r}e(i,-\la)\quad\text{and}\quad\tau^r=(\si^r+1)e(i,-\la),$$ 
for $r=1,\dots,k-1$, and define the {\em column block permutation subspace} $T(i,-\la)\subseteq M(i,-\la)$ to be the $\O$-span of all vectors of the form 
\begin{equation}\label{ESi^r}
\si^{r_1}\dots\si^{r_a}m(i,-\la).
\end{equation}


\begin{Lemma}\label{rho translation}
We have
\begin{enumerate}
\item $\sgn(\si^r)=\si_r$ and $\sgn(\tau^r)=\tau_r$, for $1\le r<k$.
\item Under the isomorphism of Lemma~\ref{L:permIso}(ii),  $T(i,-\la)$ corresponds  to $T(-i,\la)$
\end{enumerate}
\end{Lemma}

\begin{proof}
Since $\ell(w_r)=e^2$, we have $\sgn(\psi_{w_r})=(-1)^{e^2}\psi_{w_r}=(-1)^{e}\psi_{w_r}$. So Lemma~\ref{L:permIso}(i) yields~(i). Part~(ii) follows 
  from ~(i) and Lemma~\ref{L:permIso}(ii).
\end{proof}

%
%
%



Lemma~\ref{rho translation} and 
Theorem~\ref{tau braid} now imply the following.

\begin{Proposition}\label{signed tau braid}
  Suppose that $1\le r,s<k$ and $v\in T(i,-\la)$. Then
  \begin{enumerate}
  \item $(\tau^r)^2v=v$.
     \item If $|r-s|>1$ then 
        $\tau^r\tau^sv=\tau^s\tau^rv$.
     \item If $r< k-1$ then
        $\tau^r\tau^{r+1}\tau^rv= 
                \tau^{r+1}\tau^r\tau^{r+1}v$.
  \end{enumerate}
  Consequently, $\Si_k$ acts on $T(i,-\la)$, and the elements $\tau^um(i,-\la)$ for $u\in \Si_k$ are well-defined. Finally,  $T(i,-\la)\cong\ind_{\O\Si_\la}^{\O\Si_k}\triv{\la}$ as $\O\Si_k$-modules, and $T(i,-\la)$ has $\O$-basis $\{\tau^um(i,-\la)\mid u\in\D_\la\}
$.  

\end{Proposition}



\subsection{Column Garnir tableau}
We now rework the combinatorics of row Garnir tableaux for column Garnir tableaux. 
A node $A=(a,b,m)\in\mu$ is a
{\em column Garnir node} of $\mu$ if $(a,b+1,m)$ is
a node of $\mu$.  The (column) {\em  $A$-Garnir belt} $\Belt_A$ is the set of nodes 
$$\Belt_A=\set{(c,b,m)\in\mu|c\geq a}\cup\set{(c,b+1,m)\in\mu |c\leq a}.$$

Recall from (\ref{EWT}) that if $\T\in\St(\mu)$ then $\T_\mu\ledom\T$ and $w_\T\in\Si_d$ is the permutation such
that $\T=w_\T T_\mu$.  Let $u=\T_\mu(a,b,m)$ and $v=\T_\mu(a+1,b,m)$.  
The {\em (column) $A$-Garnir tableau} $\G_A$  is the $\mu$-tableaux which agrees with $T_\mu$ outside of $\Belt_A$ and where the numbers $u,u+1,\dots,v$ are inserted into the Garnir belt in order, from top right to left bottom.


Just as in section~\ref{SGTab} we have the following two results.

\begin{Lemma}\label{signed LAgrees} 
Suppose that $A\in \mu$ is a column Garnir node and $\Stab\in\St(\mu)$. 
If $\G_A\ldom\Stab$ then $\Stab$ agrees with $\T_\mu$ outside of $\Belt_A$. 
\end{Lemma}

\begin{Lemma} \label{signed LAndrew}%
Suppose that 
$\T$ is a column strict $\mu$-tableaux which is not standard. Then
there exists a column Garnir tableaux $\G$ and $w\in S_d$ such that
$\T=w\G$ and $\ell(w_\T)=\ell(w_\G)+\ell(w)$. 
\end{Lemma}

\subsection{Column bricks}\label{SSColBricks}
A (column $A$-){\em brick} is a set of $e$ nodes 
$$\{(c,d,m),(c+1,d,m),\dots,(c+e-1,d,m)\}\subseteq \Belt_A$$
such that $\res(c,d,m)=\res A$. The Garnir belt
$\Belt_A$ is a disjoint union of the bricks that it contains together with
less than $e$ nodes at the bottom of column~$a$ which are not contained in a brick and less than $e$ 
nodes at the top of column~$a+1$ which are not contained in a brick. 

For example, if $e=2$, then the $(3,1,2)$-Garnir belt of
$\mu=\big((1),(7,7,4,1)\big)$ contains two bricks: 
$$\G_A=\begin{array}{l}
\Tableau{{20}}\\[10pt]
\begin{tikzpicture}[scale=0.5,draw/.append style={thick,black}]
  \newcount\col
  \foreach\Row/\row in {{1,3,8,11,14,16,18}/0,{2,4,9,12,15,17,19}/-1,{6,5,10,13}/-2,{7}/-3} {
     \col=1
     \foreach\k in \Row {
        \draw(\the\col,\row)+(-.5,-.5)rectangle++(.5,.5);
        \draw(\the\col,\row)node{\k};
        \global\advance\col by 1
      }
   }
   \draw[red,double,very thick]
     (0.5,-3.5)--++(0,2)--++(1,0)--++(0,2)--++(1,0)--++(0,-3)--++(-1,0)--++(0,-1)--cycle;
   \draw[red,double,very thick](1.5,-2.5)--(1.5,-1.5);
   \draw[red,double,very thick](1.5,-0.5)--(2.5,-0.5);
\end{tikzpicture}
\end{array}.$$

Let $k=k_A$ be the number of bricks in $\Belt_A$. Label the bricks $B^1_A,B^2_A,\dots B^k_A$ in $\Belt_A$
from top to bottom first down column~$b+1$ and then down column~$b$ of~$\mu$.  
Set $k=0$ if $\Belt_A$ does not contain any bricks.

If $k>0$ let $n=n_A$ be the smallest number in $\G_A$ which is contained in a brick in
$\Belt_A$. In the example above, $k=2$ and $n=4$. 
Define
$$w^r_A=\prod_{a=n+re-e}^{n+re-1}(a,a+e)\in\Si_d\qquad (1\le r<k).$$ 
The {\em (column) brick permutation group} is the subgroup $\Si_A$ of $\Si_d$ generated by 
$w^1_A,w^2_A,\dots,w^{k-1}_A$. Then $\Si_A\cong\Si_k$. 

Let $\Gar_A$ be the set of all column-strict $\mu$-tableaux which are obtained from the Garnir tableau
$\G_A$ by acting with the brick permutation group $\Si_A$ on $\G_A$. All tableaux in $\Gar_A$ are standard except for $\G_A$, $\G_A$ is the maximal element of
$\Gar_A$, and there is a unique minimal tableaux $\T_A$ in $\Gar_A$. If $\T\in\Gar_A$ then $\bi(\T)=\bi(\G_A)$. We let $\bi_A:=\bi(\G_A)$.

Define $f=f_A$ to be the number of $A$-bricks in column $b$ of the Garnir belt $\Belt_A$ 
and let
$\D_A$ be the set of minimal length left coset representations of $\Si_f\times\Si_{k-f}$ in
$\Si_A\cong\Si_k$. Just as in (\ref{EGarD}), we have
\begin{equation}\label{signed EGarD}
\Gar_A=\{w\T_A\mid w\in\D_A\}. 
\end{equation}

Finally, as in Lemma~\ref{L10910}, we have:

\begin{Lemma}\label{signed L10910}
Let 
$A\in\mu$ be a column Garnir node. Then
$$
\Gar_A\setminus\{\G_A\}=\{\T\in\St(\mu)\mid \T\ledom\G_A\ \text{and}\ \bi(\T)=\bi_A\}.
$$
Moreover, $\codeg(\T)=\codeg(\G_A)$ for all $\T\in\Gar_A$. 
\end{Lemma}

\subsection{\boldmath The column permutation modules $M_\mu$}
Let 
$C_1,\dots,C_g$ be the non-empty columns of $\mu$ counted from left to right in the component $\mu^{(l)}$, then from left to right in the component $\mu^{(l-1)}$, and so on, until from left to right in the component $\mu^{(1)}$ of~$\mu$.
{\em We emphasize that the order of the components
of~$\mu$ is reversed here}. 

To each $1\leq a\leq g$ we associate the segment $\bc(a):=\bs(i,-N)$, where the column
$C_a$ has length $N$ and $i$ is the residue of the top node of $C_a$.  Let $\vec{\bc}=(\bc(1),\dots,\bc(g))$,
and, recalling the definitions from section~\ref{SSPerm}, set 
$$
M_\mu=M_\mu(\O):=M(\vec{\bc})\,\<\codeg \T_\mu\>.
$$ 
The module $M_\mu$ is generated by the vector
$
m_\mu:=m(\vec{\bc}) 
$
of degree $\codeg \T_\mu$. For any $\mu$-tableau $\T$, define $m_\T:=\psi_\T m_\mu$. 
As a special case of Theorem~\ref{TMBasis}, we have:

\begin{Theorem} \label{signed TMMUBasis}
$
\{m_\T\mid \text{$\T$ is a column-strict $\mu$-tableau}\}
$
is an $\O$-basis of $M_\mu$.  
\end{Theorem}

\subsection{\boldmath Universal column Specht modules $S_\mu$}\label{SSdual}

Fix a column Garnir node $A\in\mu$, and let $\Si_A=\langle w^1_A,\dots,w^{k-1}_A\rangle$ be the corresponding block permutation group. 
For
any $\Stab\in\Gar_A$, we can write $w_\Stab=u_\Stab w_{\T_A}$ with
$\ell(w_\Stab)=\ell(u_\Stab)+\ell(w_{\T_A})$ and $u_\Stab\in\D_A$. 
By Lemma~\ref{LFullComm}, 
$w_\Stab,u_\Stab,$ and $w_{\T_A}$ are fully commutative so we have
elements $\psi_\Stab,\psi_{u_\Stab}$ and $\psi_{\T_A}$, with
$\psi_\Stab=\psi_{u_\Stab}\psi_{\T_A}$, each of which is independent of the choice of preferred
decomposition.

Set $m_A:=m_{\T_A}=\psi_{\T_A}m_\mu$ 
and define 
$$
\si^r_A:=(-1)^e\psi_{w^r_A}e(\bi_A)\quad\text{and}\quad \tau^r_A:=(\si^r_A+1)e(\bi_A). 
$$
Any element $u\in\Si_A$ can written as a reduced product $u=w_{r_1}^A\dots w_{r_m}^A$.
If~$u$ is fully
commutative then 
$\tau^u_A:=\tau^{r_1}_A\dots \tau^{r_m}_A$
is independent of the choice of the reduced expression by Lemma~\ref{LFullComm},
so we have well-defined elements 
$\{\tau^u_A\mid u\in \D_A\}$.

\begin{Definition}\label{dual Specht relations}
Suppose that 
$A\in\mu$ is a column Garnir node. The {\em column Garnir
element} is  
$$
g_A:=\sum_{u\in\D_A}\tau^u_A\psi_{\T_A}\in R_\al.  
$$
\end{Definition}

Since $\psi_{\T_A}m_\mu=m_A$, we have 
$g_Am_\mu=\sum_{u\in\D_A}\tau^u_Am_A
$, and, by Lemma~\ref{signed L10910}, all summands on the right hand side have  the same degree.  
If $k=0$ then $\D_A=\{1\}$, $\G_A=\T_A$ and $g_A=\psi_{\G_A}$. 

\begin{Definition} \label{signed Specht}
The {\em universal graded column Specht module}\, $S_\mu=S_\mu(\O)$ is the graded $R_\al$-module generated  by the vector $z_\mu$ of degree $\codeg(\T_\mu)$ subject only to the following relations:
\begin{enumerate}
\item[{\rm (i)}] $e(\bj)z_\mu=\de_{\bj,\bi_\mu} z_\mu$ for all $\bj\in I^\al$;
\item[{\rm (ii)}] $y_rz_\mu=0$ for all $r=1,\dots,d$;
\item[{\rm (iii)}] $\psi_rz_\mu=0$ for all $r=1,\dots,d-1$ such that  $r\downarrow_{\T_\mu}  r+1$;
\item[{\rm (iv)}] ({\em homogeneous (column) Garnir relations})\, $g_A z_\mu=0$\, for all (column) Garnir nodes\, $A$ in~$\mu$. 
\end{enumerate}
In other words, $S_\mu= (R_\al/J_{\al,\mu})\<\codeg(\T_\mu)\>$, where $J_{\al,\mu}$ is the left ideal of $R_\al$ generated by the elements 
\begin{enumerate}
\item[{\rm (i)}] $e(\bj)-\de_{\bj,\bi_\mu}$ for all $\bj\in I^\al$;
\item[{\rm (ii)}] $y_r$ for all $r=1,\dots,d$;
\item[{\rm (iii)}] $\psi_r$ for all $r=1,\dots,d-1$ such that  $r\downarrow_{\T_\mu}  r+1$;
\item[{\rm (iv)}] $g_A$\, for all column Garnir nodes\, $A\in \mu$. 
\end{enumerate}
\end{Definition}

Since the elements (i)-(iii) generate the left ideal $K_{\mu}$ with $R_\al/K_{\mu}\cong M_\mu$,
we have a natural surjection $M_\mu\onto S_\mu$, which maps $m_\mu$ to $z_\mu$, and the kernel $J_\mu=J_{\al,\mu} m_\mu$ of this surjection is generated by the
Garnir relations. 

\subsection{\boldmath Column brick permutation space $T_{\mu,A}$} \label{signed SSBrickPerm}
The {\em (column) brick permutation space}\,
$T_{\mu,A}\subseteq M_\mu$ is the $\O$-span of all elements of the form 
$\si^{r_1}_A\dots\si^{r_a}_Am_A$.
Repeating the argument of Theorem~\ref{TTauA} now gives:

\begin{Theorem} \label{signed TTauA}
Suppose that 
$A\in\mu$ is a column Garnir node, and 
let~$k=k_A$ and~$f=f_A$. Then:
\begin{enumerate}
\item[{\rm (i)}] $T_{\mu,A}$ the $\O$-span of all elements of the form 
$
\tau^{r_1}_A\dots\tau^{r_a}_Am_A. 
$
In particular, the elements $\tau^1_A,\dots,\tau^{k-1}_A$ act on $T_{\mu,A}$. 
\item[{\rm (ii)}]  As $\O$-linear operators on  $T_{\mu,A}$, the elements $\tau^1_A,\dots,\tau^{k-1}_A$ satisfy the Coxeter relations for the symmetric group $\Si_k$. Thus, we can consider $T_{\mu,A}$ as an $\O\Si_k$-module. 
\item[{\rm (iii)}] There is an isomorphism of $\O\Si_k$-modules 
$$T_{\mu,A}\cong\ind_{\O\Si_{(f,k-f})}^{\O\Si_k}\triv{(f,k-f)}$$ under which $m_A$ corresponds to the natural cyclic generator of the induced module on the right hand side. 
\item[{\rm (iv)}] $\{\tau^u_A m_A\mid u\in\D_A\} 
$
is an $\O$-basis of $T_{\mu,A}$. 
\end{enumerate}
\end{Theorem}

%

\begin{Corollary} \label{signed C8910}
Suppose that 
$A\in\mu$ is a column Garnir node of $\mu$, and $\Stab=u\T_A\in\Gar_A$ for some $u\in\D_A$. 
Then
$$
\psi_\Stab m_\mu=\tau^u_A m_A+\sum_{w\in\D_A,\ w\lhd u} c_w\tau^w_Am_A
$$
for some $c_w\in\O$. In particular, $\{m_\T \mid \T\in\Gar_A\}$  
is an $\O$-basis of $T_{\mu,A}$.
\end{Corollary}

\subsection{A spanning set for the universal column Specht module}
Recall from section~\ref{SSdual}
that $S_\mu\cong M_\mu/J_\mu$ and $z_\mu=m_\mu+J_\mu$. Also set 
$z_A:=m_A+J_\mu\in S_\mu$
for any column Garnir node $A\in\mu$. 
Recall from (\ref{EPsiT}) that for each $\mu$-tableau $\T$ we have defined the element $\psi_\T\in R_\al$,
which depends on a fixed choice of reduced decomposition of $w_\T\in\Si_d$. We associate to $\T$ the
homogeneous element $v_\T:=\psi_\T z_\mu \in S_\mu$.

Adapting the arguments from section~\ref{SSspanning} we obtain the 
following result.


\begin{Proposition} \label{signed PSpan}
The elements 
$\{v_\T\mid \T\in\St(\mu)\}$
span $S_\mu$ over $\O$. Moreover, we have $\deg(v_\T)=\codeg(\T)$ for all $\T\in \St(\mu)$. 
\end{Proposition}

%

\subsection{\boldmath Graded column Specht modules $S_{\mu}^{H}$ for Hecke algebras}\label{SS graded H dual}
The 
{\em graded column Specht modules} 
for the cyclotomic Hecke algebra $H_\al^\La$ were defined in \cite[\S6]{HM} as cell modules for certain graded cellular structure on $H_\al^\La$ (different from the one used to define cell modules $S^{\mu}_{H}$). 
We review the key properties of these modules, paralleling section~\ref{SSHSpecht}. 

Recall the definition of the conjugate multipartition $\mu'$ and conjugate tableaux from section~\ref{SS:tableaux}.
If $\mu\in\Par_\al$ then in general $\mu'\notin\Par_\al$. Because of
this we will use a {\em different labelling of the column Specht modules than}\,~\cite{HM}.
Let $S_\mu^{HM}(F)$ be the graded column Specht module for $H_\al^\La(F)$ constructed in \cite[\S6.4]{HM}. That is, $S_\mu^{HM}(F)$ is a graded cell module with basis $\set{\psi'_\T|\T\in\St(\mu)}$, where $\deg(\psi'_\T)=\codeg(\T')$ for $\T\in\St(\mu)$, using the cellular basis notation of \cite[\S6.4]{HM}. 
Define 
$$
S^H_\mu(F):=S_{\mu'}^{HM}(F)\qquad \text{and}\qquad 
z^H_\mu:=\psi'_{\T^{\mu'}}.
$$ 
The following lemma was proved in \cite[Proposition 3.7]{JM}:

\begin{Lemma}\label{signed LZ1} 
Let $\al\in Q_+$, $d=\height(\al)$, and $\mu\in \Par_\al$. As an $H_\al^\La$-module, $S_\mu^H$ is generated by  $z^H_\mu$, 
$\deg(z^H_\mu)=\codeg(\T_\mu)$, $z^H_\mu\in e(\bi_\mu)S_\mu^H$, and $y_r z^H_\mu=0$, for all $r=1,\dots,d$.
\end{Lemma}

\begin{proof}
  Since $(\T^{\mu'})'=\T_\mu$, the first three claims follow from   
  \cite[Proposition~6.10]{HM}. That $y_rz^H_\mu=0$ for all $r$, can be deduced from \cite[(6.2)]{HM} and (\ref{QEPolKL}). 
 Alternatively, it is a special case of \cite[Corollary~3.11]{HM2}.
\end{proof}

For each $\mu$-tableau~$\T$, define $v^H_\T=\psi_{T}z^H_\mu\in S_\mu^H$. By \cite[Definition~6.9]{HM}, $v^H_\T$ is the same as the element $\psi'_{\T'}$ in the notation of \cite{HM}. 
In particular, $z^H_\mu=v^H_{\T_\mu}$. 

\begin{Lemma}\label{signed LPsiStr} 
Suppose that 
$\T \in\St(\mu)$. 
If $r\downarrow_\T  r+1$ or $r\rightarrow_\T  r+1$ 
then
$$
\psi_rv^H_\T =\sum_{\Stab\in \St(\mu),\,\Stab\ldom \T,\,\bi(\Stab)=\bi(s_r\T)}a_\Stab v^H_\Stab.
$$
for some $a_\Stab\in F$. 
In particular, $\psi_rz^H_\mu=0$ whenever $r\downarrow_{\T_\mu}  r+1$. 
\end{Lemma}

\begin{proof}
  This can be deduced from \cite[Proposition~6.10(c)]{HM} using standard properties of the
  (ungraded) dual Murphy basis.  Alternatively, it follows immediately from
  \cite[Corollary~3.12]{HM2}.
\end{proof}

The next result is the analogue of Theorem~\ref{PVZ}. 

\begin{Theorem}\label{signed PVZ} 
We have 
\begin{enumerate}
\item[{\rm (i)}] If $\T$ is a $\mu$-tableau then $v^H_\T\in e(\bi(\T))S_\mu^H$. 
\item[{\rm (ii)}] If $\T\in\St(\mu)$ then $\deg(v^H_\T)=\codeg(\T)$. 
\item[{\rm (iii)}] $\{v^H_\T\mid \T\in \St(\mu)\}$ is a basis of  $S_\mu^H(F)$. Moreover, for any 
$\mu$-tableau~$\T$, 
$$
v^H_\T=\sum_{\Stab\in \St(\mu),\Stab\ledom \T,\bi(\Stab)=\bi(\T)}b_\Stab v^H_\Stab
$$
for some constants $b_\Stab\in F$. 
\end{enumerate}
\end{Theorem}

\begin{proof}
  Everything except for the second part of~(iii) is clear from results in~\cite{HM} and the
  remarks above. Part~(iii) can be deduced from \cite[Proposition~6.10(c)]{HM}. Alternatively,
  it can be deduced from  \cite[Theorem~3.9]{HM2}.
\end{proof}


\begin{Corollary}\label{signed C10910}
Suppose that 
$A\in \mu$ is a column Garnir node. Then 
$$
v^H_{\G_A}=\sum_{\T\in\Gar^A,\ \T\ldom \G_A}c_\T v^H_\T\qquad
$$
for some $c_\T\in F$. 
\end{Corollary}

\subsection{Connecting the universal column Specht modules with the cell modules}
As in the last section let $S_\mu^H(F)$ be the graded column Specht module for
$H_\al^\La$, where $F$ is a field. As in section~\ref{SSconnection} we consider $S_\mu^H(F)$ as an
$R_\al(F)$-module. 

Mimicking the proof of Theorem~\ref{TMain} and using, in particular, the results in 
section~\ref{SS graded H dual}, Theorem~\ref{signed TTauA} and Corollary~\ref{signed C10910}, we can now show that
$S_\mu^H(F)\cong S_\mu(F)$ as an $R_\al(F)$-module. As the argument is similar to the proof of Theorem~\ref{TMain} we
leave the details to the reader.

\begin{Theorem} \label{signed TMain}
There is a homogeneous isomorphism
$S_\mu^H(F)\iso S_\mu(F)$ of graded $R_\al(F)$-modules, which maps $v^H_\T\in
S_\mu^H(F)$ to $v_\T\in S_\mu(F)$ for all $\T\in\St(\mu)$. 
\end{Theorem}

The following three Corollaries of Theorem~\ref{signed TMain} are proved in exactly the same way as the corresponding
results in section~\ref{SSconnection}.

\begin{Corollary} \label{signed COBasis}
$\{v_\T\mid \T\in\St(\mu)\}$ 
is an $\O$-basis of 
$S_\mu(\O)$. 
\end{Corollary}

\begin{Corollary} 
The universal column Specht module $S_\mu(\O)$ factors through the natural surjection $R_\al(\O)\onto  R_\al^\La(\O)$ so that $S_\mu(\O)$ is naturally a graded  $R_\al^\La(\O)$-module. 
\end{Corollary}
%

\begin{Corollary} \label{C:dualRels}
We have 
\begin{enumerate}
\item[{\rm (i)}] As a graded $R_\al^\La(\O)$-module, the universal column Specht module $S_\mu(\O)$ is generated by the homogeneous element $z_\mu$ of degree $\codeg(\T_\mu)$ subject only to the relations (i)--(iv) from Definition~\ref{signed Specht}. 
\item[{\rm (ii)}] As a graded  $H_\al^\La$-module, the column Specht module $S_\mu^H$ is generated  by the homogeneous element $z_\mu$ of degree $\codeg(\T_\mu)$ subject only to the relations (i)--(iv) from Definition~\ref{signed Specht}. 
\end{enumerate}
\end{Corollary}

\subsection{Contragredient duality for Specht modules}
Recall from section~\ref{SSisomorphisms} that $M^\circledast$ denotes the graded dual of the
$R_\al$-module $M$. We now use \cite{HM} to show that $S^\mu(\O)^\circledast\cong S_\mu(\O)$, up to an explicit degree shift, as graded 
$R_\alpha(\O)$-modules for any integral domain~$\O$. 

Recall that
$\set{v^\T|\T\in\St(\mu)}$ is a basis of $S^\mu(\O)$ and that $\set{v_\T|\T\in\St(\mu)}$ is a basis of
$S_\mu(\O)$. Let $\set{f_\T|\T\in\St(\mu)}$ and $\set{f^\T|\T\in\St(\mu)}$ be the corresponding dual bases of
$S^\mu(\O)^\circledast$ and $S_\mu(\O)^\circledast$, respectively, so that
$$f_\Stab(v^\T)=\delta_{\Stab,\T}=f^\Stab(v_\T),$$
where $\Stab,\T\in\St(\mu)$. 
By definition, $\deg f_\T=-\deg
v^\T=-\deg\T$ and $\deg f^\T=-\codeg\T$. Recalling (\ref{EDegCodeg}), we now have
\begin{equation}\label{E:DualDegrees}
   \deg f_\T=\codeg\T-\defect\al\qquad\text{and}\qquad\deg f^\T=\deg\T-\defect\al.
\end{equation}

\begin{Lemma}\label{L:generation}
 As $R_\alpha$-modules,
  $S^\mu(\O)^\circledast$ is generated by $f_{\T_\mu}$ and
  $S_\mu(\O)^\circledast$ is generated by $f^{\T^\mu}$.
\end{Lemma}

\begin{proof}
  We only prove that $S^\mu(\O)^\circledast=R_\alpha f_{\T_\mu}$. The proof of the second
  statement is similar. 

  We claim that if $\T\in\St(\mu)$ then there exist scalars $c_\Stab\in\O$ such that
  $$f_\T = \psi_\T f_{\T_\mu} + \sum_{\Stab\in\St(\mu)}c_\Stab f_\Stab,$$
  where $c_\Stab\ne0$ only if $\ell(w^\Stab)>\ell(w^\T)$.
  The claim implies that $f_\T\in R_\alpha f_{\T_\mu}$, for all
  $\T\in\St(\mu)$, so that $S^\mu(\O)^\circledast=R_\alpha f_{\T_\mu}$ by the remarks above. 
  
  To prove the claim we argue by downwards induction on the dominance order. If $\T=\T_\mu$
  then $w_{\T_\mu}=1=\psi_{\T_\mu}$ so that indeed 
  $f_{\T_\mu}=\psi_{\T_\mu}f_{\T_\mu}$. 
  Next suppose that $\T\ldom\T_\mu$ and let
  $\Stab\in\St(\mu)$. Then, by definition,
  $$(\psi_\T f_{\T_\mu})(v^\Stab)=f_{\T_\mu}(\psi_{w_\T^{-1}} v^\Stab)
                         =f_{\T_\mu}(\psi_{w_\T^{-1}}\psi_{w^\Stab} z^\mu).$$
  By Lemma~\ref{eBruhat}(iii), $w^{\T_\mu}=w_\T^{-1}w^\T$ and $\ell(w^{\T_\mu})=\ell(w_\T^{-1})+\ell(w^\T)$.
  Consequently, if $\ell(w^\Stab)\le\ell(w^\T)$ and $\Stab\ne\T$ then $w^{\T_\mu}$ can not appear as a
  subexpression of~$w_\T^{-1}w^\Stab$ so that $(\psi_\T f_{\T_\mu})(v^\Stab)=0$ by Proposition~\ref{PSubtle}.
  Therefore, the coefficient of $f_\Stab$ in $\psi_\T f_{\T_\mu}$ is zero whenever $\ell(w^\Stab)<\ell(w^\T)$.
  Finally, consider the case when $\Stab=\T$. By Proposition~\ref{PSubtle}, there exist polynomials
  $p_u(y)\in\O[y_1,\dots,y_d]$ such that 
  \begin{align*}
  \psi_{w_\T^{-1}}\psi_{w^\T}z^\mu
        &=\psi^{\T_\mu}z^\mu +\sum_{u<w^{\T_\mu}}\psi_u p_u(y)z^\mu\\
        &=v^{\T_\mu} +\sum_{u<w^{\T_\mu}}p_u\psi_u z^\mu,
  \end{align*}
  where $p_u=p_u(0)\in\O$ and the last equality follows from Lemma~\ref{LZ1}. 
  It follows that $(\psi_\T f_{\T_\mu})(v^\T)=1$. 
  Hence, if
  we write $\psi_\T f_{\T_\mu}$ with respect to the basis $\{f_\Stab\}$ then $f_\T$ appears with
  coefficient~$1$. This completes the proof of the claim and,  hence, of the lemma.
\end{proof}

We can now prove the main result of this section.

\begin{Theorem}\label{Tduality}
As graded $R_\al(\O)$-modules, 
  $$S^\mu(\O)\cong S_\mu(\O)^\circledast\<\defect\alpha\>\quad\text{and}\quad
    S_\mu(\O)\cong S^\mu(\O)^\circledast\<\defect\alpha\>$$
\end{Theorem}

\begin{proof}The two isomorphisms are equivalent so we consider only the first isomorphism.  By
  Lemma~\ref{L:generation} and (\ref{E:DualDegrees}) it is enough to show that~$f^{\T^\mu}$ satisfies the
  defining relations from Definition~\ref{DSpecht} for the element $z^{\mu}$ as, taking into account our basis results, this will imply that there
  is a unique isomorphism $S^\mu(\O)\bijection S_\mu(\O)^\circledast\<\defect\alpha\>$
  which sends~$z^{\mu}$ to~$f^{\T^\mu}$. From the definitions, $e(\bi)f^{\T^\mu}=\delta_{\bi \bi^\mu}$, 
  so it remains to show that 
  \begin{enumerate}
    \item[{\rm (ii)}] $y_rf^{\T^\mu}=0$ for all $r=1,\dots,d$;
    \item[{\rm (iii)}] $\psi_rf^{\T^\mu}=0$ for all $r=1,\dots,d-1$ such that  $r\rightarrow_{\T^\mu}  r+1$;
    \item[{\rm (iv)}]  $g^A f^{\T^\mu}=0$\, for all row Garnir nodes $A$ in~$\mu$. 
  \end{enumerate}
  By freeness it is sufficient to consider the case when $\O=\Z$ and, since $S_\mu(\Z)^\circledast$ embeds into
  $S_\mu(\C)^\circledast$, it is enough to verify the relations when $\O=\C$. 
  
  As in Section~\ref{SSconnection}, let $\xi=\exp(2\pi i/e)$ if $e>0$ and if $e=0$ take~$\xi$ to be any
  non-root of unity in~$\C$. Then, by Theorem~\ref{TBK}, $R_\al^\Lambda(\C)\cong H_\al^\Lambda(\C,\xi)$,   so we
  can invoke results from~\cite{HM}. Hence, as graded $R_\al$-modules,
  \begin{xalignat*}{2}
  S^\mu(\C)&\cong S^\mu_H(\C),                      &&\text{by Theorem~\ref{TMain}},\\
      &\cong S_\mu^H(\C)^\circledast\<\defect\al\>, &&\text{by \cite[Proposition~6.19]{HM}},\\
      &\cong S_\mu(\C)^\circledast\<\defect\al\>,   &&\text{by Theorem~\ref{signed TMain}}.
  \end{xalignat*}
  To complete the proof we scrutinize the second isomorphism above. 

  In our notation, the proof of \cite[Proposition~6.19]{HM} shows that there exists a homogeneous associative
  bilinear form
  $$\{\ ,\ \}:S^\mu_H\times S_\mu^H\<\defect\al\>\longrightarrow\C;\quad (a,b)\mapsto\{a,b\},$$
  such that $\{v_H^\Stab,v^H_\T\}=0$ unless $\T\gedom\Stab$. (When comparing our notation with~\cite{HM} the reader
  should remember that $S_\mu^H=S_{\mu'}^{HM}$ as defined in section~\ref{SS graded H dual}.) The isomorphism
  $S^\mu_H(\C)\bijection S_\mu^H(\C)^\circledast\<\defect\al\>$ is then the map which sends $a\in S^\mu_H(\C)$
  to $\varphi_a\in S_\mu^H(\C)^\circledast\<\defect\al\>$, where $\varphi_a(b)=\{a,b\}$, for all
  $b\in  S_\mu^H(\C)^\circledast\<\defect\al\>$. Observe that the triangularity of the form $\{\ ,\
  \}$ implies that $\varphi_{v^{\T^\mu}}$ is a scalar multiple of~$f^{\T^\mu}$. Therefore, since the
  map $a\mapsto\varphi_a$ is an isomorphism, it follows from Definition~\ref{DSpecht} and
  Corollary~\ref{C:dualRels} that~$\varphi_{v^{\T^\mu}}$, and hence $f^{\T^\mu}\in S_\mu(\C)^\circledast$, 
  satisfies the 
  three relations (ii)--(iv) above. Consequently,  $f^{\T^\mu}\in S_\mu(\Z)^\circledast$ also satisfies these 
  relations and the theorem is proved.
\end{proof}

\begin{Remark}
  In principle, it should be possible to prove Theorem~\ref{Tduality} directly by verifying that $f^{\T^\mu}$
  satisfies the relations in Definition~\ref{DSpecht}. This appears to be an involved calculation.
\end{Remark}

%

\section{Two applications}\label{SApp}

In this section we work again over an arbitrary commutative unital ring~$\O$. 

\subsection{Specht modules for higher levels as induced modules}

Let  $\mu=(\mu^{(1)},\dots,\mu^{(l)})\in\Par_\al$ with $\al^{(m)}=\cont(\mu^{(m)})$, for $m=1,\dots,l$. 
Then $\al=\al^{(1)}+\dots+\al^{(l)}$. 
Consider each partition $\mu^{(m)}$ as an element of ${\mathscr P}_{\al^{(m)}}^{(k_m)}$; that is, as a
partition whose $(1,1)$-node has residue $k_m$.  Then we have the universal graded Specht modules
$S^{\mu^{(m)}}$ for the algebras $R_{\al^{(m)}}^{\La_{k_m}}$, for $m=1,\dots,l$. Inflating along the surjection
$R_{\al^{(m)}}\onto R_{\al^{(m)}}^{\La_{k_m}}$ we may consider $S^{\mu^{(m)}}$ as a graded
$R_{\al^{(m)}}$-module.  Note that this graded module is generated by the element $z^{\mu^{(m)}}$ of degree
$\deg (\T^{\mu^{(m)}})$. 

As in (\ref{ECircProd}), considered the graded  $R_\al$-module 
$S(\mu^{(1)})\circ\dots\circ S(\mu^{(l)})$ which is generated by the element 
\begin{equation}\label{E10910}
z^{\mu^{(1)},\dots,\mu^{(l)}}:=1\otimes\big(z^{\mu^{(1)}}\otimes\dots\otimes z^{\mu^{(l)}}\big)
\end{equation}
of degree $\deg \T^{\mu^{(1)}}+\dots+\deg \T^{\mu^{(l)}}$. 

Our new definition of Specht modules by generators and relations makes the following useful result almost obvious. Note that in \cite{Vaz}, \cite[(3.24)]{BKariki} the right hand side of (\ref{EIso}) was taken as the definition of the Specht module.  

\begin{Theorem}\label{TSpechtHigherLevel}
Suppose that $\mu=(\mu^{(1)},\dots,\mu^{(l)})\in\Par_\al$. Then
\begin{equation}\label{EIso}
S^\mu\cong S^{\mu^{(1)}}\circ\dots\circ S^{\mu^{(l)}}\langle d_\mu\rangle,
\end{equation}
where
$$d_\mu:=\deg(\T^\mu)-\deg (\T^{\mu^{(1)}})-\dots-\deg (\T^{\mu^{(l)}}).$$
as graded $R_\al$-modules. In particular, $S^{\mu^{(1)}}\circ\dots\circ S^{\mu^{(l)}}\langle d_\mu\rangle $ factors through the surjection $R_\al\onto R_\al^\La$, and the isomorphism (\ref{EIso}) is also an isomorphism of graded $R_\al^\La$-modules. 
\end{Theorem}
\begin{proof}
The vector $z^{\mu^{(1)},\dots,\mu^{(l)}}$ from (\ref{E10910}) 
satisfies the defining relations on the vector $z^\mu\in S^\mu$ from  Definition~\ref{DSpecht}. This yields a homogeneous module homomorphism $S^\mu\longrightarrow S^{\mu^{(1)}}\circ\dots\circ S^{\mu^{(l)}}\langle d_\mu\rangle$ which maps $z^\mu$ onto $z^{\mu^{(1)},\dots,\mu^{(l)}}$.  
To construct the inverse homomorphism, by Frobenius Reciprocity, it suffices to construct a homomorphism of $R_{\al^{(1)},\dots,\al^{(l)}}$-modules 
$$
S^{\mu^{(1)}}\boxtimes\dots\boxtimes S^{\mu^{(l)}}\longrightarrow \Res^{\al}_{\al^{(1)},\dots,\al^{(m)}}S^{\mu},
$$
which maps $z^{\mu^{(1)}}\otimes\dots\otimes z^{\mu^{(l)}}$ onto $z^\mu$. Such homomorphism arises by Definition~\ref{DSpecht} again, using defining relations for the modules $
S^{\mu^{(1)}},\dots, S^{\mu^{(l)}}$. 
\end{proof}

\subsection{Column Specht modules as signed row Specht modules}\label{SSSign}
In this final section we investigate the analogue of tensoring the Specht modules with the sign representation.
Recall the isomorphism $\sgn:R_\al\to R_{\alpha'}$ from (\ref{epsilon}) 
and the $\sgn$-twist $M^\sgn\in \Mod{R_\al}$ of a module $M\in \Mod{R_{\al'}}$. 
We determine what happens to the Specht modules of $R_{\al'}$ under this twist. 

For each $\mu\in\Par_\al$ we have row Specht module $S^\mu$ and column  Specht module $S_\mu$ with
bases $\{v^\T\}$ and $\{v_\T\}$, respectively, parametrized by $\T\in\St(\mu)$. Similarly, for
each $\nu\in\Par[']_{\al'}$ we have row Specht module $S^\nu$ and column Specht module $S_\nu$,
with bases $\{v^\Stab\}$ and $\{v_\Stab\}$, respectively, parametrized by $\Stab\in\St(\nu)$.
The definition of these modules and bases depends on~$\kappa$
and~$\kappa'$, respectively. 

In section~\ref{SSPar} we defined the conjugate multipartition $\mu'\in\Par[']_{\al'}$ of the multipartition~$\mu\in\Par_\al$. Recall from section~\ref{SS:tableaux} that the
definition of degree and codegree of a tableau~$\T$ depends on~$\kappa$. We
write $\deg^\kappa(\T), \codeg^\kappa(\T), \res^\kappa A$, etc., when we want to emphasize dependence on~$\kappa$. Finally, the conjugate tableau $\T'$ is defined in section~\ref{SS:tableaux}, and if $\bi\in I^\al$ then $-\bi\in I^{\al'}$ is defined in (\ref{E-I}). 

For any node $A=(a,b,m)$ define $A':=(b,a,l-m+1)$. 
Note that $A\in\mu$ if and only if $A'\in \mu'$, in which case $\T(A)=\T'(A')$. Moreover, we have $\res^\kappa A\equiv -\res^{\kappa'}A'$ 
  by (\ref{ERes}), and $A$ is above $B$ if and only if $A'$ is below $B'$. The following lemma now follows from definitions. 

\begin{Lemma}\label{L:signs}
  Suppose that $\mu\in\Par_\al$ and $\T\in\St(\mu)$. Then $\bi^{\kappa'}(\T')=-\bi^\kappa(\T)$, $\deg^\kappa(\T)=\codeg^{\kappa'}(\T')$ and
  $\codeg^\kappa(\T)=\deg^{\kappa'}(\T')$.
\end{Lemma}


The main result of this section is:

\begin{Theorem}\label{TSignDualSpecht}
  Suppose that $\al\in Q_+$ with $d=\height(\al)$, and $\mu\in\Par_\alpha$. Then
  $$S^\mu\cong(S_{\mu'})^\sgn \quad\text{and}\quad S_\mu\cong(S^{\mu'})^\sgn$$ 
  as graded $R_\al(\O)$-modules.
\end{Theorem}

\begin{proof}
  We claim that there are degree zero homomorphisms  of graded  $R_\alpha(\O)$-modules
  $$ \theta^\mu:S^\mu\longrightarrow(S_{\mu'})^\sgn\quad\text{and}\quad
     \theta_{\mu'}:(S_{\mu'})^\sgn\longrightarrow S^\mu$$
  such that $\theta^\mu(z^\mu)=z_{\mu'}$ and $\theta_{\mu'}(z_{\mu'})=z^\mu$. As $z^\mu$
  and $z_{\mu'}$ generate the two Specht modules, this claim implies the theorem.  
  
Note that  
$$\deg z^\mu=\deg^\kappa(\T^\mu)=\codeg^{\kappa'}(\T_{\mu'})=\deg z_{\mu'}$$ 
  by Lemma~\ref{L:signs}. 
 So to prove the existence of $\theta^\mu$, it suffices to check that $z_{\mu'}\in (S_{\mu'})^\sgn$ satisfies the defining relations of~$S^\mu$ from Definition~\ref{DSpecht}. The map $\theta_{\mu'}$ is constructed similarly using Definition~\ref{signed Specht} instead, so we only give details for $\theta^\mu$. 

  By Lemma~\ref{L:signs}, $\bi^\mu=\bi(\T^\mu)=-\bi(\T_{\mu'})=-\bi_{\mu'}$.  Therefore, if
  $\bj\in I^\al$ then
  $$e(\bj)\cdot z_{\mu'}=e(-\bj)z_{\mu'}=\delta_{-\bj,\bi_{\mu'}}z_{\mu'}=\delta_{-\bj,-\bi^{\mu}}z_{\mu'}=\delta_{\bj,\bi^\mu}z_{\mu'}.$$
  Therefore, $z_{\mu'}$ satisfies Definition~\ref{DSpecht}(i).
  Moreover, $y_r\cdot z_{\mu'}=-y_r z_{\mu'}=0$ for all $1\le r\le d$. Next observe that if
  $1\le r<d$ then $r\rightarrow_{\T^\mu}r+1$ if and only if $r\downarrow_{\T_{\mu'}}r+1$.
  Hence,
  $\psi_r\cdot z_{\mu'} =-\psi_rz_{\mu'}=0,$
  by Definition~\ref{DSpecht}(iii) and Definition~\ref{signed Specht}(iii). 

  It remains to check that $z_{\mu'}\in (S_{\mu'})^\sgn$ satisfies the row Garnir relations from Definition~\ref{DSpecht}(iv). Recall the node correspondence $A\leftrightarrow A'$
  defined before Lemma~\ref{L:signs} which sends a node $A\in\mu$ to $A'\in\mu'$.
  If $A\in\mu$ is a row Garnir node then $A'\in\mu'$ is a column
  Garnir node and, further, this correspondence sends row bricks in~$\mu$ to column
  bricks in~$\mu'$.  In particular, $k^A=k_{A'}$, where~$k^A$ is the number of row bricks
  in~$\Belt^A$ and~$k_{A'}$ is the number of column bricks in~$\Belt_{A'}$. Moreover,
  $\sgn(\tau_r^A)=\tau_{A'}^r$ by Lemma~\ref{rho translation}(i), for $1\le r<k_A$, 
  so that $\sgn(g^A)=g_{A'}$. Therefore,
  $$g^A\cdot z_{\mu'}=\sgn(g^A)z_{\mu'}=g_{A'}z_{\mu'}=0,$$ 
  where the last equality is a column Garnir relation in $S_{\mu'}$.
\end{proof}


\end{document}